\documentclass[hidelinks,onefignum,onetabnum]{siamart251216}

\usepackage{amsopn}
\usepackage{lipsum}
\usepackage{amssymb}
\usepackage{amsfonts}
\usepackage{multirow}
\usepackage{algorithmic}
\usepackage[section]{placeins}
\usepackage[T1]{fontenc}

\ifpdf
  \DeclareGraphicsExtensions{.eps,.pdf,.png,.jpg}
\else
  \DeclareGraphicsExtensions{.eps}
\fi

\DeclareMathOperator*{\argmin}{arg\,min}
\DeclareMathOperator*{\argmax}{arg\,max}

\DeclareMathOperator{\Tr}{Tr}
\newcommand{\R}{\mathbb{R}}
\newcommand{\E}{\mathbb{E}}

\newsiamremark{remark}{Remark}
\newsiamremark{hypothesis}{Hypothesis}
\crefname{hypothesis}{Hypothesis}{Hypotheses}
\Crefname{hypothesis}{Hypothesis}{Hypotheses}
\newsiamthm{claim}{Claim}
\crefname{claim}{Claim}{Claim}
\Crefname{claim}{Claim}{Claim}
\newsiamremark{fact}{Fact}
\crefname{fact}{Fact}{Facts}
\Crefname{fact}{Fact}{Facts}
\newsiamremark{assumption}{Assumption}
\crefname{assumption}{Assumption}{Assumptions}
\Crefname{assumption}{Assumption}{Assumptions}
\crefname{appendix}{Appendix}{Appendices}
\Crefname{appendix}{Appendix}{Appendices}

\headers{Viscosity-Informed Generative Actor-Critic}{A. E. Golpashin, G. Puthumanaillam, M. Ornik, and B. A. Conway}

\title{Viscosity-Informed Generative Actor-Critic for High-Dimensional Stochastic Optimal Control%
\thanks{\funding{The work of first and fourth authors was funded by the U.S. Air Force Research Laboratory (AFRL) under grant no.~FA8651-23-1-0001.}}}

\author{
Alen E. Golpashin\thanks{Department of Aerospace Engineering, University of Illinois, Urbana, IL
  (\email{agolpa2@illinois.edu}, \email{gokulp2@illinois.edu}, \email{mornik@illinois.edu}, \email{bconway@illinois.edu}).}
\and Gokul Puthumanaillam\footnotemark[1]
\and Melkior Ornik\footnotemark[1]
\and Bruce A. Conway\footnotemark[1]
}

\ifpdf
\hypersetup{
  pdftitle={Viscosity-Informed Generative Actor-Critic for High-Dimensional Stochastic Optimal Control},
  pdfauthor={A. E. Golpashin, G. Puthumanaillam, M. Ornik, and B. A. Conway}
}
\fi

\begin{document}

\maketitle

% REQUIRED
\begin{abstract}
We introduce a method for approximating viscosity solutions of stationary degenerate elliptic Hamilton--Jacobi--Bellman equations on bounded domains arising in stochastic exit-time control. Viscosity enforcement is formulated as a min--max problem over an envelope-generated test family parameterized by symmetric positive definite matrices. Under structural and asymptotic assumptions, any uniform limit point of the value function approximations satisfies the viscosity inequalities on the sampled test family. Numerical experiments show that the proposed method reduces empirical viscosity violations and improves robustness under perturbed dynamics.
\end{abstract}

% REQUIRED

% REQUIRED

%%%%%%%%%%%%%%%%%%%%%%%%%%%%%%%%%%%%%%%%%%%%%%%%%%%%%%%%%%%%%%%%%%%%%%%%%%%%%%%%%%%%%%%%

\section{Introduction}The formulation of stochastic optimal control problems for dynamical systems whose state evolution is modeled by It\^o processes naturally leads to the analysis of second-order, fully nonlinear, degenerate elliptic Hamilton--Jacobi--Bellman (HJB) equations. While one might hope for \(C^2\) solutions, the value functions for such control problems need not possess this regularity and are therefore naturally understood within the framework of viscosity solutions. In the uniformly elliptic subclass, classical $C^{2,\alpha}$ regularity can follow under additional coefficient regularity and structural hypotheses \cite{evans1982,krylov1984}; whenever these stringent requirements are not met, such regularity cannot generally be assumed. This lack of guaranteed smoothness poses a challenge for any methodology that presumes classical differentiability.

This paper considers learning-based approximation of viscosity solutions of stationary second-order HJB equations arising from stochastic optimal control of controlled It\^o diffusions on bounded domains. Under the hypotheses stated below, dynamic programming characterizes the discounted exit-time value function as the unique continuous viscosity solution of the associated degenerate elliptic Dirichlet problem \cite{fleming2006}. We construct parametrized approximations of this value function together with associated feedback controls from trajectory data.

The central idea behind our computation is to use the viscosity formulation as a training signal. Instead of   enforcing the HJB equation in strong form, our method enforces the viscosity inequalities themselves. This is accomplished primarily through a hybrid, viscosity-informed loss function for the critic network, which combines the data-driven temporal-difference error with a viscosity-informed penalty. This penalty enforces sampled HJB sub- and supersolution inequalities using test functions generated by Moreau envelopes. To make this viscosity-based loss computationally tractable, we introduce an amortization scheme where a generative proximal network is trained to approximate the argmin mapping, thereby avoiding inner-loop optimization. These components are integrated to form the Viscosity-Informed Generative Actor-Critic (V-GAC) framework, which employs an adversarial three-network architecture with a policy-gradient actor update.

Residual-penalized actor-critic methods such as HJBPPO \cite{mukherjee2023} provide a natural baseline for comparison, but they remain tied to strong-form residual minimization which implicitly assumes differentiability of the value function and does not encode the viscosity selection mechanism. More fundamentally, HJB equations generally admit multiple weak solutions, and only the viscosity formulation selects the unique physically meaningful solution through comparison principles formulated in the uniform topology. Strong-form residual minimization does not encode this selection mechanism. Our approach is therefore formulated within the viscosity framework, which provides the canonical context for well-posedness and for developing solution methods robust to the expected lack of regularity.

For numerical approximation of viscosity solutions, Barles and Souganidis \cite{barles1991} prove convergence of monotone, stable, and consistent schemes to viscosity solutions under comparison. Esteve-Yag{\"u}e et al.~\cite{esteve-yague2025} use neural losses based on monotone finite-difference schemes for first-order Hamilton--Jacobi equations. Our analysis establishes conditional consistency on sampled SPD-envelope jets without assuming scheme monotonicity. Other learning-based PDE methods encode known convexity properties of value functions through input-convex or partially input-convex networks or convexity penalties \cite{liu2023}. For general linear and nonlinear PDEs, the weak adversarial networks of Zang et al.~\cite{zang2020} replace residual minimization with min--max formulations based on weak solution theory. For HJB equations, Zhou et al.~\cite{zhou2021} develop a neural actor-critic method assuming \(C^2\) value functions and uniform ellipticity. Related Hamilton--Jacobi-based optimization exploits Moreau envelope and Hopf--Lax connections to approximate proximal mappings from objective evaluations \cite{osher2023}.

The remainder of the paper is organized as follows. \Cref{sec:problem_formulation} presents the problem formulation and viscosity background, \cref{sec:method} introduces the proposed method, \cref{sec:experiments} reports numerical results, and the appendix contains technical details.

%%%%%%%%%%%%%%%%%%%%%%%%%%%%%%%%%%%%%%%%%%%%%%%%%%%%%%%%%%%%%%%%%%%%%%%%%%%%%%%%%%%%%%%%

\section{Problem Formulation and Background}\label{sec:problem_formulation}We introduce the stochastic optimal control problem and the associated HJB equation.

\subsection{Controlled It\^o diffusions and the HJB equation}The evolution is governed by
\begin{equation}\label{eq:sde}
dX_t = f(X_t,u_{t}) \,dt + \Sigma(X_t,u_{t})\, dW_t, \qquad X_0=x,
\end{equation}
on a filtered probability space $(\Omega',\mathcal F,\{\mathcal F_t\}_{t\ge0},\mathbb P)$ with the usual augmentation; $W$ is an $n$-dimensional $\{\mathcal F_t\}$-Brownian motion, $u=(u_t)_{t\ge0}\in\mathcal U$ is $\{\mathcal F_t\}$-progressively measurable, $X_0$ is $\mathcal F_0$-measurable. The reachable domain $D\subset\R^{n}$ is a bounded open set with a closed interior target $\mathcal{T}\Subset D$. In particular, $\Omega:=D\setminus\mathcal{T}$ with $\partial\Omega = \partial D \cup \partial\mathcal T$ denoting the boundary of this set. The cost functional is 
\begin{equation}\label{eq:J}
J(x;u)
:= \mathbb{E}\!\left[
\int_{0}^{\tau} e^{-\beta t}\,\ell\!\left(X_t,u_t\right)\,dt\, +\, \mathbf{1}_{\{\tau<\infty\}}\,e^{-\beta\tau}\,g(X_{\tau})\,\Big|\, X_0=x\right],
\end{equation}
where $\ell$ is scalar-valued, $\tau:\Omega'\to[0,\infty]$ is an
$\{\mathcal F_t\}$-stopping time taking values in the extended nonnegative
reals, and $\beta\in\mathbb{R}_0^+$. Here $\tau$ denotes the termination rule; in the derivation below we specialize to the exit-time case $\tau:=\tau_\Omega$. For each initial state $X_0\in\R^{n}$, the functional $u \;\mapsto\; J(x;u)$ is minimized over all admissible control trajectories in $\mathcal U$. For $\phi \in C^2(\Omega)$ and $c \in U\Subset\R^{m}$, the infinitesimal generator of the controlled
diffusion $X_t$ is
\begin{equation} \label{eq:generator}
\mathcal L^{c} \phi(\cdot)
:= f(\cdot,c)\cdot\nabla \phi(\cdot) + \tfrac{1}{2}\,\operatorname{tr}\!\big(a(\cdot,c)\,\nabla^2\phi(\cdot)\big),
\end{equation}
where $a(\cdot,c) = \Sigma(\cdot,c)\,\Sigma(\cdot,c)^{\top}$. Assume there exists $L>0$, uniform in $c\in U$, such that
\begin{gather}
\|f(X,c)-f(\tilde X,c)\| + \|\Sigma(X,c)-\Sigma(\tilde X,c)\|_F \le L\,\|X-\tilde X\|, \label{eq:lipschitz} \\
\|f(X,c)\|^2 + \|\Sigma(X,c)\|_F^2 \le L^2\,(1+\|X\|^2), \label{eq:growth}
\end{gather}
where $\|\cdot\|$ is the Euclidean norm and $\|\cdot\|_F$ denotes the Frobenius norm. Under \eqref{eq:lipschitz}--\eqref{eq:growth}, for every admissible control $u\in\mathcal U$, \eqref{eq:sde} admits a unique strong solution $X^u$ \cite[Thm.~5.2.1]{oksendal2003}.

The dynamic-programming principle formally yields the HJB Dirichlet
problem
\begin{equation}\label{eq:HJB}
\begin{cases}
\beta V(x)-\inf_{c\in U}
\Big\{\ell(x,c)+\mathcal L^cV(x)\Big\}=0,
& x\in\Omega,\\[3pt]
V(x)=g(x), & x\in\partial\Omega.
\end{cases}
\end{equation}
We write the interior equation of \eqref{eq:HJB} as
\begin{equation*}
F(x,V,\nabla V,\nabla^2 V)=0,
\end{equation*}
with the nonlinear operator $F:\Omega\times\R\times\R^{n}\times\mathbb{S}(n)\to\R$,
\begin{equation}\label{eq:F-def}
F(x,r,p,A):=\beta r - \inf_{c\in U}H(x,p,A;c),
\end{equation}
where  $H(x,p,A;c):= \ell(x,c)+p^\top f(x,c)+\tfrac12\operatorname{tr}\!\big(a(x,c)A\big)$ is the control Hamiltonian. Throughout,
\begin{align*}
\mathbb{S}(n) &:= \{A\in\mathbb{R}^{n\times n} : A^\top=A\},\\
\mathbb{S}_{+}(n) &:= \{A\in\mathbb{S}(n) : A\succeq 0\},\\
\mathbb{S}_{++}(n) &:= \{A\in\mathbb{S}(n) : A\succ 0\}.
\end{align*}
denote the space of symmetric matrices and its positive-semidefinite
and positive-definite cones, respectively, with the Loewner order $A\succeq B \Leftrightarrow A-B\succeq 0$. We use $X_t$ for the state process and $x\in\Omega$ for spatial points where the PDE is evaluated.

We make the following standing assumptions.
\begin{assumption} \label[assumption]{assumps:structure} For Dirichlet problem \eqref{eq:HJB}, the following hold:
\begin{romannum}
\item \label[assumption]{assump:discount-or-elliptic} Either $\beta>0$, or there exist $0<\lambda\le\Lambda<\infty$ such that
\[
\lambda I\ \preceq\ \Sigma(x,c)\Sigma(x,c)^\top\ \preceq\ \Lambda I
\quad\text{for all }(x,c)\in\overline{\Omega}\times U.
\]

\item\label[assumption]{assump:attainment-continuity}
$U$ is compact and the coefficients $(x,c)\mapsto f(x,c)$, $(x,c)\mapsto \Sigma(x,c)$, and $(x,c)\mapsto \ell(x,c)$ are continuous on $\overline{\Omega}\times U$.
\end{romannum}
\end{assumption}

The HJB equation \eqref{eq:HJB} provides the pointwise minimization criterion for feedback selection. Define the optimal value function by \(V^*(x):=\inf_{u\in\mathcal U}J(x;u)\), \(x\in\overline{\Omega}\). Under \eqref{eq:lipschitz}--\eqref{eq:growth} and \cref{assumps:structure}, if $V\in C^2(\Omega)\cap C(\overline{\Omega})$ is a classical solution of \eqref{eq:HJB} and a measurable feedback \(\mu^*\) attains the infimum such that \(u_t^*=\mu^*(X_t^*)\) is admissible, then classical verification gives \(V(x)=J(x;u^*)=V^*(x)\) \cite[Thm.~IV.5.1 and Cor.~IV.5.1]{fleming2006}. This value--feedback relation underlies the actor-critic architecture: the actor network approximates the optimal feedback \(\mu^*\), while the critic \(V_\theta\) approximates the corresponding optimal value \(V^*\). Since \(V^*\) need not belong to \(C^2(\Omega)\), we formulate and analyze \eqref{eq:HJB} in the viscosity-solution framework.

The viscosity formulation and the comparison principle used below
rely on the ellipticity and properness of the operator $F$
\cite[eqs.~(0.1)--(0.3) and (3.13)]{crandall1992}.

\begin{definition}\label{def:structure}
We say that $F$ is degenerately elliptic if it is nonincreasing with
respect to the Loewner order: whenever $A,B\in\mathbb S(n)$ with
$A\succeq B$,
\[
F(x,r,p,A)\le F(x,r,p,B),
\]
or, equivalently,
\(
F(x,r,p,A+Q)\le F(x,r,p,A)\) for all \(Q\in\mathbb S_+(n)\).
We say that $F$ is proper if it is nondecreasing in $r$: whenever
$r\le s$,
\[
F(x,r,p,A)\le F(x,s,p,A).
\]
If, in addition, there exists $\beta>0$ such that
\( F(x,s,p,A)-F(x,r,p,A)\ge\beta(s-r)\) for all \(s>r\), then $F$ is strictly proper. In particular, if $\partial_rF$ exists,
then $\partial_rF\ge\beta>0$.
\end{definition}

For the operator \eqref{eq:F-def}, $F$ is degenerately elliptic and
proper; under \cref{assumps:structure} (i), it is either
strictly proper or uniformly elliptic. Moreover, under
\cref{assumps:structure} (ii), compactness of $U$ and continuity
of the coefficients imply that the infimum in \eqref{eq:F-def} is
attained and that $F$ is continuous. We use the standard
Crandall--Ishii--Lions test-function definition
\cite[Def.~2.2 and Rem.~2.3]{crandall1992}.

\begin{definition}\label{def:visc_ineq}
Let $F$ be the operator defined in \eqref{eq:F-def}. A function
$V:\Omega\to\mathbb R$ is called a
\begin{romannum}
\item viscosity subsolution in $\Omega$ if $V$ is upper
semicontinuous (u.s.c.) in $\Omega$ and, for every
$\phi\in C^2(\Omega)$ and every $\zeta\in\Omega$ at which $V-\phi$
attains a local maximum,
\[
F\big(\zeta,V(\zeta),\nabla\phi(\zeta),
\nabla^2\phi(\zeta)\big)\le0;
\]
\item viscosity supersolution in $\Omega$ if $V$ is lower
semicontinuous (l.s.c.) in $\Omega$ and, for every
$\phi\in C^2(\Omega)$ and every $\zeta\in\Omega$ at which $V-\phi$
attains a local minimum,
\[
F\big(\zeta,V(\zeta),\nabla\phi(\zeta),
\nabla^2\phi(\zeta)\big)\ge0.
\]
\end{romannum}
We call $V$ a viscosity solution in $\Omega$ if it is both a
viscosity subsolution and a viscosity supersolution; in particular,
$V\in C(\Omega)$.
\end{definition}

\subsection{Comparison principle}The following comparison principle is a specialization of Jensen's maximum principle \cite{jensen1989} to the Bellman operator \eqref{eq:F-def}.
\begin{theorem}
\label{thm:jensen-comparison}
Assume \eqref{eq:lipschitz} and \cref{assumps:structure}. Let
\(V_1,V_2:\overline{\Omega}\to\mathbb R\) be bounded with \(V_{1}\)
u.s.c. and \(V_2\) l.s.c. on \(\overline{\Omega}\), where \(V_{1}\) is a
viscosity subsolution and \(V_{2}\) is a viscosity supersolution of
\(F=0\) in \(\Omega\). Then
\begin{equation}
\label{eq:comparison-estimate}
\sup_{x\in\Omega}\max\{V_1(x)-V_2(x),0\}
\le
\sup_{x\in\partial\Omega}\max\{V_1(x)-V_2(x),0\}.
\end{equation}
Consequently, if \(V_{1}\le V_{2}\) on \(\partial\Omega\), then
\(V_{1}\le V_{2}\) on \(\overline{\Omega}\). In particular, the Dirichlet
problem \eqref{eq:HJB} has at most one viscosity solution
\(V\in C(\overline{\Omega})\) satisfying \(V=g\) on
\(\partial\Omega\).
\end{theorem}
\begin{proof}
Jensen uses the opposite ellipticity and subsolution sign convention;
consequently, our sub- and supersolutions for \(F=0\) satisfy his
respective inequalities for the pointwise-negative equation \(-F=0\).
We therefore apply \cite[Cor.~5.11]{jensen1989}. Since the infimum is
achieved,
\(-F(x,r,p,A)=\min_{c\in U}\{H(x,p,A;c)-\beta r\}\).
This is the Bellman special case of Jensen's form~\textnormal{(5.6)},
with diffusion factors \(\Sigma(\cdot,c)/\sqrt{2}\), drift
\(f(\cdot,c)\), coefficient \(\beta\) of \(r\), and \(h=-\ell\) in
Jensen's notation. By \eqref{eq:lipschitz}, the diffusion factors and
drifts are uniformly Lipschitz, while
\cref{assumps:structure}\textnormal{(ii)} implies that
\(\{-\ell(\cdot,c)\}_{c\in U}\) is equicontinuous.

Write
\(\|f\|_\infty:=
\sup_{(x,c)\in\overline{\Omega}\times U}\|f(x,c)\|<\infty\).
For \(A_1,A_2\in\mathbb S(n)\) with \(A_1\preceq A_2\) and
\(r_1\ge r_2\), choose
\(c_2\in\argmin_{c\in U}H(x,p_2,A_2;c)\). Jensen's difference for
\(-F\) satisfies
\[
F(x,r_2,p_2,A_2)-F(x,r_1,p_1,A_1)
\le
\tfrac12\operatorname{tr}\!\big(a(x,c_2)(A_1-A_2)\big)
+\|f\|_\infty\|p_1-p_2\|-\beta(r_1-r_2).
\]
If \(\beta>0\), then \(a(x,c_2)\succeq0\) bounds the right-hand side
by
\(\beta(r_2-r_1)+\|f\|_\infty\|p_1-p_2\|\).
If \(a(x,c)\succeq\lambda I\) uniformly, it is bounded by
\(\frac{\lambda}{2}\operatorname{tr}(A_1-A_2)
+\|f\|_\infty\|p_1-p_2\|\), since
\(-\beta(r_1-r_2)\le0\). Thus Jensen's condition~\textnormal{(4.1)} holds with
\(K_3=\|f\|_\infty\), with \(K_2=\beta\) and any \(K_1>0\) in the
discounted case, or with \(K_1=\lambda/2\) and any \(K_2>0\) in the
uniformly elliptic case. Corollary~5.11 therefore yields
\eqref{eq:comparison-estimate}.
\end{proof}

Comparison gives at most one solution attaining \(g\), but does not establish such attainment \cite[sect.~1]{barles1991}.

\subsection{Boundary attainment and well-posedness}\label{sec:boundary-regularity}We first specify the notions of strong boundary attainment and \(g\)-regularity.

\begin{definition} An interior viscosity solution \(V\in C(\Omega)\) strongly attains
\(g\) if
\[
    \lim_{\Omega\ni x\to\xi}V(x)=g(\xi)
    \qquad\text{for every }\xi\in\partial\Omega .
\]
Let \(\underline V_g\) and \(\overline V_g\) denote the lower and upper Perron envelopes for \(F=0\) with boundary datum \(g\), in the sense of \cite[sect.~2]{ramaswamy1994}. A point \(\xi\in\partial\Omega\) is \(g\)-regular if
\[
 \lim_{\Omega\ni x\to\xi}\underline V_g(x)
 =
 \lim_{\Omega\ni x\to\xi}\overline V_g(x)
 =g(\xi).
\]
\end{definition}

\begin{assumption}\label[assumption]{assumps:boundary}The boundary \(\partial\Omega\) is Lipschitz, \(g\in C(\partial\Omega)\), and every \(\xi\in\partial\Omega\) is \(g\)-regular for \(F\).
\end{assumption}

\begin{proposition}
\label{prop:strong-dirichlet}
Under \eqref{eq:lipschitz} and
\cref{assumps:structure,assumps:boundary}, problem
\eqref{eq:HJB} admits a unique viscosity solution
\(V\in C(\overline{\Omega})\). Moreover,
\(\lim_{\Omega\ni x\to\xi}V(x)=g(\xi)\) for every
\(\xi\in\partial\Omega\).
\end{proposition}

\begin{proof}
If \(\beta>0\), choose
\(
C\ge
\max\left\{
\|g\|_{L^\infty(\partial\Omega)},
\frac{\|\ell\|_{L^\infty(\overline\Omega\times U)}}{\beta}
\right\}.
\)
Then

\[
F(x,-C,0,0)\le-\beta C+\|\ell\|_\infty\le0,
\qquad
F(x,C,0,0)\ge\beta C-\|\ell\|_\infty\ge0,
\]
so \(-C\) and \(C\) belong to the lower and upper Perron families. If \(\beta=0\), \cref{assumps:structure}\textnormal{(i)} implies
uniform ellipticity, and the global elliptic barriers of
\cite[sect.~3, especially Prop.~3.2]{crandall1999}
provide the required sub- and supersolutions. Thus the Perron
families are nonempty and their envelopes finite.

The standard Perron argument
\cite[Thm.~4.1]{crandall1992}, applied to the two
envelopes as in \cite[sects.~2--3]{ramaswamy1994}, together with
\cref{assumps:boundary,thm:jensen-comparison}, identifies the
envelopes as a unique viscosity solution \(V\in C(\Omega)\) strongly
attaining \(g\). Extending \(V\) by \(V=g\) on \(\partial\Omega\) gives \(V\in C(\overline{\Omega})\).
\end{proof}

\begin{remark}\label[remark]{remark:boundary-barriers}
Comparison with two-sided local \(g\)-barriers gives \(g\)-regularity in \cref{assumps:boundary} (cf. \cite[sect.~4]{ramaswamy1994}). With smooth approximation of \(g\), it suffices to have \(\mathfrak b\in C(\overline\Omega)\), \(C^2\) near \(\partial\Omega\), positive in \(\Omega\), zero on \(\partial\Omega\), and satisfying \(\inf_{c\in U}\{\beta\mathfrak b-\mathcal L^c\mathfrak b\}\ge1\) near \(\partial\Omega\). This is a local stationary analogue of the strict-barrier condition in \cite[Assumps.~2.1--2.2 and Lem.~2.5]{dong2007}. For \(C^2\) boundaries, uniform boundary-normal nondegeneracy over \(U\) verifies this criterion by the distance-barrier calculation of \cite[Ex.~4.6]{crandall1992}.
\end{remark}

By \cref{prop:strong-dirichlet}, the network's approximation target is
the unique viscosity solution \(V\in C(\overline\Omega)\). Although comparison and Perron's method are formulated for u.s.c./l.s.c.\ functions, each network approximation \(V_\theta\) and every uniform critic limit belong to \(C(\overline\Omega)\); our approximation scheme is therefore restricted to this space.

\subsection{Physics-informed neural networks as PDE solvers}\label{sec:pinns}PINNs \cite{raissi2019} solve differential equations by fitting a neural network $V_\theta$ to satisfy boundary/initial data and to minimize a physics residual computed by automatic differentiation (AD). A core justification for representing solutions with neural networks is the Universal Approximation Theorem (UAT), including versions that control derivatives.

\begin{remark}
Neural networks are dense in $C(\overline{\Omega})$ under mild assumptions on the activation \cite{cybenko1989,hornik1990,hornik1991}, but uniform approximation does not control derivatives. Since HJB value functions are typically nondifferentiable, strong-form residual minimization based on $(\nabla V_\theta,\nabla^2 V_\theta)$ is ill-posed as a proxy for viscosity error.
\end{remark}

PINNs leverage AD to evaluate the differential operator applied to $V_\theta$. With smooth activations (e.g., $\tanh$), $V_\theta\in C^\infty$ and $\nabla V_\theta$, $\nabla^2V_\theta$ exist everywhere; with ReLU, derivatives exist a.e. For the operator $F$ in \eqref{eq:F-def}, the pointwise residual is
\begin{equation*}
\mathcal{R}_\theta(x)
:=F\!\big(x,\,V_\theta(x),\,\nabla V_\theta(x),\,\nabla^2 V_\theta(x)\big).
\end{equation*}
However, the differentiability enabling AD is a property of the network $V_\theta$, not of the true solution $V$. When $V$ is nonsmooth, minimizing a strong-form residual does not provide a stable proxy for viscosity error. In particular, small $L^2$ residuals do not preclude violations of the viscosity conditions on sets of measure zero, nor do they encode the selection principle that ensures uniqueness of the viscosity solution.

Accordingly, we enforce the viscosity inequalities of \cref{def:visc_ineq} through an adversarial min--max objective targeting worst-case violations. Rather than enforcing a pointwise equality residual everywhere, we generate touching test functions via matrix-kernel inf/sup-convolutions (Moreau envelopes) introduced in \cref{subsec:envelopes}.

\subsection{Envelope-jet penalty: a viscosity-consistent formulation}\label{subsec:envelopes}We introduce the envelope construction and the associated envelope-jet penalty.

\begin{definition}
Let \(V:\overline{\Omega}\to\R\) be continuous. For \((x,M)\in\overline{\Omega}\times\mathbb S_{++}(n)\), define
\begin{align*} \label{envelope}
\mathcal{E}_{\inf}[V](x;M)
&:= \inf_{\zeta\in\overline{\Omega}}
\Big\{V(\zeta)+\tfrac12(x-\zeta)^\top M(x-\zeta)\Big\},\\
\mathcal{E}_{\sup}[V](x;M)
&:= \sup_{\zeta\in\overline{\Omega}}
\Big\{V(\zeta)-\tfrac12(x-\zeta)^\top M(x-\zeta)\Big\}.
\end{align*}
Choose measurable minimizers and maximizers
\[
\zeta_x\in\argmin_{\zeta\in\overline{\Omega}}
\Big\{V(\zeta)+\tfrac12(x-\zeta)^\top M(x-\zeta)\Big\}, \quad
\zeta_x^\star\in\argmax_{\zeta\in\overline{\Omega}}
\Big\{V(\zeta)-\tfrac12(x-\zeta)^\top M(x-\zeta)\Big\}.
\]
The associated jet data are then
\[
p^{-}(x,M)=M(x-\zeta_x),\quad A^{-}=-M,
\qquad
p^{+}(x,M)=-\,M(x-\zeta_x^\star),\quad A^{+}=+M.
\]
\end{definition}

The extremizers $\zeta_x$ and $\zeta_x^\star$ are the contact points; existence follows by Weierstrass. In particular, if a contact lies in $\Omega$, the associated quadratics yield legitimate second-order jets $(p^{-},A^{-}) \in J^{2,-}V(\zeta_x)$ and $(p^{+},A^{+}) \in J^{2,+}V(\zeta_x^\star)$; see \cref{app:envelopes}.

\begin{assumption}\label[assumption]{assump:sampling}
Let \(\nu_{\mathcal M}\) be a fixed probability measure on \(\mathbb S_{++}(n)\). Given any probability measure \(\nu_X\) on \(\overline\Omega\), for \(x\sim\nu_X\), the curvature matrices \(M_1,\dots,M_K\sim\nu_{\mathcal M}\) are i.i.d.\ and independent of \(x\).
\end{assumption}

\begin{definition}
Let \(F\) be continuous. For \((x,M)\) and continuous \(V\), define
\begin{subequations}\label{eq:hinges}
\begin{align}
\mathcal{J}_{\text{super}}(x,M;V)
&:= \max\{-\,F(\zeta_x,V(\zeta_x),p^{-},A^{-}),0\},\\
\mathcal{J}_{\text{sub}}(x,M;V)
&:=\max\{\,F(\zeta_x^\star,V(\zeta_x^\star),p^{+},A^{+}),0\},
\end{align}
\end{subequations}
where each quantity is set to zero when its
corresponding contact lies on \(\partial\Omega\). Given a sampling measure
\(\nu_X\otimes\nu_{\mathcal M}\) on
\(\overline\Omega\times\mathbb S_{++}(n)\), the envelope-jet penalty is
\begin{equation}\label{eq:visc}
\mathcal{J}(V)
:= \E_{(x,M)\sim\nu_X\otimes \nu_{\mathcal M}}
\big[\mathcal{J}_{\text{super}}+\mathcal{J}_{\text{sub}}\big].
\end{equation}
\end{definition}

Comparison in $C(\overline\Omega)$ is the selection principle. Accordingly, the relevant notion of error is a worst-case (sup-norm) viscosity-inequality violation.

\begin{proposition}\label{prop:jet-iff}
Let \(V\in C(\overline\Omega)\). Consider
\begin{multline}\label{eq:visc_quant}
\sup_{\zeta\in\Omega}\;
\sup_{(p,A)\in J^{2,-}V(\zeta)}
\max\{-F(\zeta,V(\zeta),p,A),0\}\\
+\sup_{\zeta\in\Omega}\;
\sup_{(p,A)\in J^{2,+}V(\zeta)}
\max\{F(\zeta,V(\zeta),p,A),0\},
\end{multline}
with $\sup\emptyset=0$. This quantity vanishes if and only if $V$ is a viscosity solution of $F=0$ in $\Omega$. If, in addition, $V=g$ on $\partial\Omega$,  $V$ is the unique continuous viscosity solution of the Dirichlet problem \eqref{eq:HJB}.
\end{proposition}

\begin{proof}
Assume \eqref{eq:visc_quant} vanishes. Each of the two summands is nonnegative, hence each must itself be
equal to $0$. Therefore, for every $\zeta\in\Omega$ and every $(p,A)\in J^{2,-}V(\zeta)$,
\[
\max\{-F(\zeta,V(\zeta),p,A),0\}=0.
\]
Since $\max\{s,0\}=0$ if and only if $s\le 0$, this gives $-F(\zeta,V(\zeta),p,A)\le 0$, that is,
\[
F(\zeta,V(\zeta),p,A)\ge 0
\quad
\text{for every }(p,A)\in J^{2,-}V(\zeta).
\]
By \cref{lem:jet-char}, $V$ is a viscosity supersolution of $F=0$ in
$\Omega$. The argument for subsolution is identical.

Conversely, assume that $V$ is a viscosity solution of $F=0$ in $\Omega$.
 $V$ is both a viscosity supersolution and a viscosity subsolution.
By \cref{lem:jet-char}, for every $\zeta\in\Omega$ and
$(p,A)\in J^{2,-}V(\zeta)$,
\[
\max\{-F(\zeta,V(\zeta),p,A),0\}=0,
\]
and for every $(p,A)\in J^{2,+}V(\zeta)$,
\[
\max\{F(\zeta,V(\zeta),p,A),0\}=0.
\]
Taking the inner and outer suprema with the $\sup\emptyset=0$ convention, shows that both summands above are equal to $0$.

Now, by the first part of the proposition, $V$ is a viscosity solution of $F=0$ in $\Omega$. If $W\in C(\overline\Omega)$ is another viscosity solution with the same boundary trace $g$, then \cref{thm:jensen-comparison}, applied with $W$ as subsolution and $V$ as supersolution, gives
\[
\sup_{\Omega}\max\{W-V,0\}\le \sup_{\partial\Omega}\max\{W-V,0\}=0,
\]
so $W\le V$ in $\Omega$. By symmetry $V\le W$, hence $V=W$ (and on $\overline{\Omega}$ by continuity).
\end{proof}

\begin{corollary}
\label{cor:envelope-necessary}
Let \cref{assumps:structure} hold and let
\(V\in C(\overline\Omega)\) be a viscosity solution of \(F=0\) in
\(\Omega\). Then \(\mathcal J(V)=0\). In particular, if
\eqref{eq:lipschitz} and \cref{assumps:boundary} also hold, the unique
continuous viscosity solution of the Dirichlet problem
\eqref{eq:HJB} lies in the zero set of \(\mathcal J\). Moreover,
vanishing envelope penalty \eqref{eq:visc} enforces the viscosity
inequalities on the sampled interior envelope test family.
\end{corollary}

\begin{proof}
Assume that $V$ is a viscosity solution of $F=0$ in $\Omega$.
By \cref{prop:jet-iff}, the jet-violation quantity \eqref{eq:visc_quant} vanishes. Fix a sampled pair $(x,M)\in\overline{\Omega}\times\mathbb S_{++}(n)$.
If the corresponding inf-envelope contact $\zeta_x$ lies in $\Omega$, then \cref{thm:inf-to-subjet} gives
\[
(p^{-},A^{-})\in J^{2,-}V(\zeta_x),
\qquad
p^{-}=M(x-\zeta_x),\quad A^{-}=-M.
\]
Since \eqref{eq:visc_quant} vanishes, it follows that $F\big(\zeta_x,V(\zeta_x),p^{-},A^{-}\big)\ge 0$. Boundary-contact terms vanish by definition; the sup-envelope case is analogous. Hence, for every sampled pair $(x,M)$, the corresponding envelope hinge terms
vanish:
\[
\mathcal J_{\text{super}}(x,M;V)=0,
\qquad
\mathcal J_{\text{sub}}(x,M;V)=0.
\]
Taking expectation with respect to
$\nu_X\otimes\nu_{\mathcal M}$ gives $\mathcal J(V)=0$. The Dirichlet assertion follows from
\cref{prop:strong-dirichlet}, and the final claim follows  from the nonnegativity of the hinge terms \eqref{eq:hinges} for $(\nu_X\otimes\nu_{\mathcal M})\text{-a.e.}$ $(x,M)$, whenever the contacts lie in $\Omega$. 
\end{proof}

%%%%%%%%%%%%%%%%%%%%%%%%%%%%%%%%%%%%%%%%%%%%%%%%%%%%%%%%%%%%%%%%%%%%%%%%%%%%%%%%%%%%%%%%

\section{Methodology}\label{sec:method}We introduce V-GAC (Viscosity-Informed Generative Actor-Critic), a three-network, on-policy method that enforces sampled envelope-jet conditions for the HJB via matrix-kernel envelopes evaluated at adversarial contact points, consisting of:
\begin{romannum}
\item Actor $\pi_\phi(\cdot\mid x)$ on $U$, with $c \sim \pi_\phi(\cdot\mid x)$.
\item Critic $V_\theta:\overline{\Omega}\to\R$.
\item Proximal Network $P_\psi:\overline{\Omega}\times\mathbb{S}_{++}(n)\times\{-1,+1\}\to\overline{\Omega}$, which, given an ``anchor'' $x\in\overline{\Omega}$, a curvature $M\in\mathbb{S}_{++}(n)$, and a polarity $b\in\{-1,+1\}$, outputs a contact point $\zeta$ intended to maximize the current viscosity-inequality violation against $V_\theta$.
\end{romannum}
A measurable feedback $\bar\pi_\phi:\overline{\Omega}\to U$ associated with the policy
$\pi_\phi(\cdot\mid x)$ is fixed and used for viscosity testing.

\subsection{A policy-conditioned HJB surrogate}Viscosity inequalities are enforced using a policy-conditioned HJB operator, defined by fixing the control in the Hamiltonian term at the current policy feedback. Let $\pi_\phi(\cdot\mid x)$ denote the stochastic policy generating trajectories, and let $\bar\pi_\phi:\overline{\Omega}\to U$ denote the deterministic feedback induced by the location parameter of $\pi_\phi(\cdot\mid x)$. The resulting policy-conditioned operator is given by
\begin{equation}\label{eq:Hpi-def}
\mathcal{H}^{\bar\pi_\phi}(x,V,p,A)
:= \beta\,V - \Big[\ell(x,\bar\pi_\phi(x))+p^\top f(x,\bar\pi_\phi(x))+\tfrac12\Tr\!\big(a(x,\bar\pi_\phi(x))A\big)\Big].
\end{equation}
This operator allows us to relate the true HJB operator, $F$ in \eqref{eq:F-def}, to our test operator via the fundamental decomposition
\begin{equation*}
F(x,V,p,A)=\mathcal{H}^{\bar\pi_\phi}(x,V,p,A) \;+\; \text{gap}_{\bar\pi_\phi}(x,p,A).
\end{equation*}
Here, the term $\text{gap}_{\bar\pi_\phi}$ represents the nonnegative
suboptimality gap of the current policy, defined as the difference between
the cost incurred by the actor's action and the minimum possible cost over
all actions
\begin{equation}\label{eq:gap}
\text{gap}_{\bar\pi_\phi}(x,p,A)
:= H\!\big(x,p,A;\bar\pi_\phi(x)\big)
-\inf_{c\in U} H(x,p,A;c)\ \ge\ 0.
\end{equation}
This decomposition has direct consequences for the viscosity tests. Since $\text{gap}_{\bar\pi_\phi}\ge 0$, the condition $\mathcal{H}^{\bar\pi_\phi}\ge 0$ is sufficient (and generally conservative) for the supersolution inequality $F\ge 0$. On the subsolution side, $F\le 0$ is equivalent to $\mathcal{H}^{\bar\pi_\phi}\le -\,\text{gap}_{\bar\pi_\phi}$; in particular,
enforcing $\mathcal{H}^{\bar\pi_\phi}\le 0$ is necessary but not sufficient unless the gap is small. Accordingly, the actor update is augmented (via the jet-alignment term introduced below) to reduce $\text{gap}_{\bar\pi_\phi}$ at the envelope-generated jets. Note that when the diffusion is control-independent ($\Sigma(x,c) \equiv \Sigma(x)$), the second-order term vanishes from \eqref{eq:gap}.
\begin{remark}
In what follows, viscosity inequalities are enforced for the policy-conditioned operator $\mathcal{H}^{\bar\pi_\phi}$; the deviation from the true HJB operator $F$ is entirely quantified by \eqref{eq:gap}. The actor is therefore formulated to induce $\text{gap}_{\bar\pi_\phi}\to0$ at the envelope-generated jets.
\end{remark}

\subsection{Viscosity loss with adversarial contacts}\label{sec:visc-loss}For each anchor $x$ and curvature matrix $M$, the Proximal Net proposes contact points
\[
\zeta^{-}_{x,M}=P_\psi(x,M,-1),\qquad \zeta^{+}_{x,M}=P_\psi(x,M,+1),
\]
interpreted as approximate inf-/sup-envelope contacts.
The induced jet data are
\[
p^{-}_{x,M}=M(x-\zeta^{-}_{x,M}),\quad A^{-}_{x,M}=-M,\qquad
p^{+}_{x,M}=-\,M(x-\zeta^{+}_{x,M}),\quad A^{+}_{x,M}=+M.
\]
We evaluate the policy-conditioned operator \eqref{eq:Hpi-def} by defining the pointwise violations
\begin{equation}\label{eq:pointwise-violations}
\begin{aligned}
g_{\text{super}}(x,M;\theta,\psi)
&:= -\mathcal{H}^{\bar\pi_\phi}\!\big(\zeta^-_{x,M},\,V_\theta(\zeta^-_{x,M}),\,p^-_{x,M},\,A^-_{x,M}\big),\\
g_{\text{sub}}(x,M;\theta,\psi)
&:= \ \ \mathcal{H}^{\bar\pi_\phi}\!\big(\zeta^+_{x,M},\,V_\theta(\zeta^+_{x,M}),\,p^+_{x,M},\,A^+_{x,M}\big).
\end{aligned}
\end{equation}
which are nonpositive when the corresponding viscosity inequalities hold for $\mathcal{H}^{\bar\pi_\phi}$ at the induced jet. Moreover, $g_{\text{super/sub}}=0$ whenever $\zeta^b_{x,M}\in\partial\Omega$. Per minibatch, a small set of symmetric positive definite matrices is drawn, defined by

\begin{equation} \label{eq:M-bank}
\mathcal{M}=\Big\{\,M_q=R_q^\top \operatorname{diag}(\alpha_{q,1},\ldots,\alpha_{q,n})R_q \ \Big|\ q=1,\ldots,K,\ M_q\in\mathbb{S}_{++}(n)\Big\}.
\end{equation}
The eigenvalues of each $M_q$ are sampled independently from a log-uniform distribution on $[\alpha_{\min},\alpha_{\max}]$, and $R_q\in O(n)$ is sampled uniformly. The notion of ``for all touching tests'' is approximated through maximizing over this parametrized family of curvature matrices $M\in\mathcal M$. Then, based on the envelope hinges \eqref{eq:hinges}, we define 
\begin{align*}
\label{eq:visc-hinges}
\mathcal{J}^{\bar\pi_\phi}_{\text{super}}(x;\theta,\psi)
&:= \left[\max\left\{\max_{M\in\mathcal M} g_{\text{super}}(x,M;\theta,\psi),\,0\right\}\right]^2,\\
\mathcal{J}^{\bar\pi_\phi}_{\text{sub}}(x;\theta,\psi)
&:= \left[\max\left\{\max_{M\in\mathcal M} g_{\text{sub}}(x,M;\theta,\psi),\,0\right\}\right]^2.
\end{align*}
For a sampled set \(\mathcal M\), the viscosity loss is the expected sum of these penalties,
\begin{equation*}\label{eq:Lvisc}
L_{\text{Visc}}(\theta,\psi)
=\E_{x\sim\nu_X}\big[\mathcal{J}^{\bar\pi_\phi}_{\text{super}}(x;\theta,\psi)+\mathcal{J}^{\bar\pi_\phi}_{\text{sub}}(x;\theta,\psi)\big].
\end{equation*}

The anchor law \(\nu_X\) in \cref{assump:sampling}, at outer iteration \(k\), is  the mixture \(\nu_{X,k}:=(1-\rho_{\mathrm{cover}})\nu_{\mathrm{on},k} +\rho_{\mathrm{cover}}\nu_{\mathrm{cov}}\), where \(\rho_{\mathrm{cover}}\in(0,1]\), \(\nu_{\mathrm{on},k}\) is the empirical state law generated by rollouts of \eqref{eq:sde} under \(\pi_{\phi_k}\), and \(\nu_{\mathrm{cov}}\) is a fixed covering distribution with full support on \(\overline\Omega\).

\subsection{Training objectives}The V-GAC framework is trained via a co-dependent optimization of its three constituent networks. In particular, the critic network $V_\theta$ is trained to minimize a hybrid objective that balances consistency with empirical data against compliance with the PDE structure. Its loss function is a weighted sum of a temporal difference (TD) error and our proposed viscosity loss
\begin{equation}
\label{eq:Lcritic}
L_{\text{Critic}}(\theta,\psi)=
\lambda_{\text{TD}}L_{\text{TD}}(\theta)\;+\;\lambda_{\text{visc}}\,L_{\text{Visc}}(\theta,\psi)+\lambda_{\text{bdy}}\,\mathbb{E}_{x\sim\nu_\Gamma}\!\big[\,|V_\theta(x)-g(x)|^2\,\big],
\end{equation}
where $L_{\text{TD}}(\theta)=\E_t[(V_\theta(x_t)-\hat V_t)^2]$ measures the squared error against target values $\hat V_t$. In particular, the target values $\hat V_t$ are computed from the discretized $\ell(x,c)$ through a discrete-time recursion with discount $\gamma \;:=\; e^{-\beta\,\Delta t}$. The viscosity loss term, $L_{\text{Visc}}$, is composed of hinge penalties for the supersolution ($g_{\text{super}}$) and subsolution ($g_{\text{sub}}$) conditions, evaluated at contact points generated by the proximal network;  $\lambda_{\text{visc}}>0$ is a regularization parameter. Boundary or terminal value constraints are imposed via the penalty term with $\lambda_{\text{bdy}}>0$,
where $\Gamma\subseteq\partial\Omega$ denotes the corresponding set
and $\nu_\Gamma$ is a probability measure supported on $\Gamma$.

The proximal network $P_\psi$ is trained adversarially to find large sampled violations of the viscosity conditions for the current critic. The proximal loss is given by
\begin{equation*}
\begin{aligned}
L_{\text{Prox}}(\psi;\theta)
:= {}& -\lambda_{\text{adv}}\E_{x\sim\nu_X}\Big[\max_{M\in\mathcal M}g_{\text{super}}(x,M;\theta,\psi) + \max_{M\in\mathcal M}g_{\text{sub}}(x,M;\theta,\psi) \Big]\\ 
&\quad +\lambda_{\text{env}}L_{\text{env}}(\psi;\theta) +\lambda_{\text{prox-opt}}L_{\text{prox-opt}}(\psi;\theta).
\end{aligned}
\end{equation*}
The proximal mapping $P_\psi:\overline{\Omega}\times\mathbb{S}_{++}(n)\times\{-1,+1\}\to\overline{\Omega}$
associates to each $(x,M,b)$ a candidate envelope contact point
$\zeta^b_{x,M}:=P_\psi(x,M,b)$, where $b=-1$ and $b=+1$ correspond to
inf- and sup-convolution contacts, respectively. The maximization is performed over a sampled set of test matrices $\mathcal{M}$ to find the worst-case violation. With $\lambda_{\text{prox-opt}}>0,$ the contact-optimality term $L_{\text{prox\text{-}opt}}$ penalizes violations of a projection-compatible first-order stationarity condition for the constrained envelope problems over $\overline{\Omega}$. Let $\Pi_{\overline{\Omega}}$ denote Euclidean projection onto $\overline{\Omega}$ and fix $\eta>0$.\footnote{If $\overline{\Omega}$ is nonconvex, $\Pi_{\overline{\Omega}}$ denotes any fixed measurable selection of $y\mapsto\arg\min_{z\in\overline{\Omega}}\|y-z\|$.} Define
\begin{equation*}
\label{eq:Lproxopt}
\begin{aligned}
G^{-}(x,M)
&:=\frac{1}{\eta}\Big(\zeta^-_{x,M}-\Pi_{\overline{\Omega}}\big(\zeta^-_{x,M}-\eta(\nabla V_\theta(\zeta^-_{x,M})-p^{-}_{x,M})\big)\Big),\\
G^{+}(x,M)
&:=\frac{1}{\eta}\Big(\zeta^+_{x,M}-\Pi_{\overline{\Omega}}\big(\zeta^+_{x,M}-\eta(p^+_{x,M}-\nabla V_\theta(\zeta^+_{x,M}))\big)\Big),
\end{aligned}
\end{equation*}
and set
\[
L_{\text{prox-opt}}(\psi;\theta)
:=
\E_{(x,M)\sim\nu_X\otimes\nu_{\mathcal M}}\Big[\|G^{-}(x,M)\|^2+\|G^{+}(x,M)\|^2\Big].
\]

Additionally, the envelope regularizer $L_{\text{env}}$ biases the proximal outputs toward minimizers of the inf-envelope and maximizers of the sup-envelope at the adversarial curvatures, which stabilizes the search for the contact points. Define
\[
\begin{aligned}
E_{\inf}(x,M)
&:= V_\theta(\zeta^-_{x,M})
+\tfrac12 (x-\zeta^-_{x,M})^\top M (x-\zeta^-_{x,M}), \\
E_{\sup}(x,M)
&:= V_\theta(\zeta^+_{x,M})
-\tfrac12 (x-\zeta^+_{x,M})^\top M (x-\zeta^+_{x,M}).
\end{aligned}
\]
and select indices
\[
q_{\text{super}}(x)\in \argmax_{1\le q\le K}
g_{\text{super}}(x,M_q;\theta,\psi),
\qquad
q_{\text{sub}}(x)\in \argmax_{1\le q\le K}
g_{\text{sub}}(x,M_q;\theta,\psi).
\]
Then
\[
L_{\text{env}}(\psi;\theta):=\E_{x\sim\nu_X}\Big[E_{\inf}\big(x,M_{q_{\text{super}}(x)}\big)-E_{\sup}\big(x,M_{q_{\text{sub}}(x)}\big)\Big].
\]
In regimes where \(\overline\Omega\) is convex and the inf-/sup-envelope objectives are differentiable and convex/concave, respectively, \(L_{\text{prox-opt}}=0\) is sufficient for the proximal network outputs to recover global inf-/sup-envelope contacts for almost every sampled \((x,M)\), since first-order stationarity characterizes global optimality. On convex \(\overline\Omega\), if the envelope objectives are differentiable but nonconvex, the stationarity condition enforced by \(L_{\text{prox-opt}}\) remains necessary but is generally insufficient
for global optimality. In general, however, when \(\overline\Omega\) is nonconvex, the projected-gradient condition enforced by \(L_{\text{prox-opt}}\) is neither necessary nor sufficient for global optimality. To guide the contact search in this general case, we additionally use an envelope regularizer \(L_{\text{env}}\).  The envelope-consistency term $L_{\text{env}}$ does not restore sufficiency, but acts as a stabilizing regularizer by coupling the inf- and sup-envelope energies at adversarial curvatures, discouraging degenerate or one-sided stationary solutions and improving the conditioning of the min--max game.

Finally, the actor network $\pi_\phi$ is updated using the clipped surrogate objective from PPO~\cite{schulman2017}, with an additional entropy regularization term to encourage exploration, augmented by a jet-alignment penalty that pushes the policy toward greedy minimizers of the HJB inner term at both inf- and sup-envelope contact jets
\begin{equation*}
\label{eq:actor-total}
\begin{split}
L_{\text{Actor}}(\phi)
&= -\,\mathbb{E}_t\!\Big[\min\big(\varrho _t(\phi)\hat{A}_t,\ \text{clip}(\varrho _t(\phi),1-\epsilon,1+\epsilon)\hat{A}_t\big)\Big] \\
&\qquad -\ \lambda_{\text{ent}}\,\mathbb{E}_t[\mathcal{S}(\pi_\phi(\cdot|x_t))] +\ \lambda_{\text{jet}}\,L_{\text{jet}}.
\end{split}
\end{equation*}
Here, $\varrho_{t}(\phi)=\frac{\pi_\phi(c_t|x_t)}{\pi_{\phi_{\text{old}}}(c_t|x_t)}$ is the likelihood ratio, $\hat A_t:=V_\theta(x_t)-\hat V_t$ denotes the corresponding advantage estimate, where $\hat V_t$ is the target value used in the critic update and $\mathcal{S}(\pi_\phi(\cdot\mid x_t))$ is the entropy of the policy distribution\footnote{$\lambda_{\text{ent}}$ may be annealed downward over training iterations.}. The jet-alignment term is
\begin{equation*}\label{eq:Ljet}
L_{\text{jet}}(\phi;\psi)
:= \E_{(x,M)\sim \nu_X\otimes\nu_{\mathcal M}}
\left[\sum_{\substack{b\in\{-1,+1\}\\ \zeta^b_{x,M}\notin\partial\Omega}}
H\!\Big(\zeta^b_{x,M},\,p^b_{x,M},\,A^b_{x,M};\bar\pi_\phi(\zeta^b_{x,M})\Big)\right],
\end{equation*}
where $\zeta^-_{x,M}=P_\psi(x,M,-1)$ and $\zeta^+_{x,M}=P_\psi(x,M,+1)$.

\begin{lemma}\label{lem:jet_gap_decomp}
Under \cref{assumps:structure} (ii), define the policy-independent baseline
\[
L_{\text{jet}}^\star(\psi)
:=\E_{(x,M)\sim \nu_X\otimes\nu_{\mathcal M}}
\left[
\sum_{\substack{b\in\{-1,+1\}\\ \zeta^b_{x,M}\notin\partial\Omega}}
\min_{c\in U}
H(\zeta^b_{x,M},\,p^b_{x,M},\,A^b_{x,M};c)
\right].
\]
Then $L_{\text{jet}}(\phi;\psi)-L_{\text{jet}}^\star(\psi) \ge 0$. Moreover, equality holds if and only if
$\text{gap}_{\bar\pi_\phi}=0$ at the sampled interior jets $\nu_X\otimes\nu_{\mathcal M}$-a.s.
\end{lemma}

\begin{proof}
Fix $(x,M)$ and $b\in\{-1,+1\}$. By \eqref{eq:gap} and compactness/continuity,
\[
\begin{aligned}
H(\zeta^b_{x,M},\,p^b_{x,M},\,A^b_{x,M};\bar\pi_\phi(\zeta^b_{x,M}))
&=\min_{c\in U} H(\zeta^b_{x,M},\,p^b_{x,M},\,A^b_{x,M};c)\\
&\quad+
\text{gap}_{\bar\pi_\phi}(\zeta^b_{x,M},\,p^b_{x,M},\,A^b_{x,M}).
\end{aligned}
\]
Summing over the two jet polarities and taking expectation yields
\[
L_{\text{jet}}(\phi;\psi)-L_{\text{jet}}^\star(\psi)
=\mathbb{E}_{(x,M)\sim\nu_X\otimes\nu_{\mathcal M}}
\left[
\sum_{\substack{b\in\{-1,+1\}\\ \zeta^b_{x,M}\notin\partial\Omega}}
\text{gap}_{\bar\pi_\phi}(\zeta^b_{x,M},\,p^b_{x,M},\,A^b_{x,M})
\right].
\]
Nonnegativity follows from $\text{gap}_{\bar\pi_\phi}\ge0$. If the expectation is zero, then the nonnegative sum above vanishes almost surely, hence each nonnegative term in the sum vanishes almost surely. The converse is immediate.
\end{proof}

\subsection{V-GAC as a min--max game}The overall training dynamic is best understood as a competition between two players. The first player is the actor-critic pair, parameterized by $\phi$ and $\theta$, whose objective is to find an optimal policy supported by a value function consistent with both empirical data and the viscosity conditions. The second player is the proximal network, parameterized by $\psi$, which acts as an adversary by searching for the most egregious violations of viscosity conditions.
We solve the following min--max game
\begin{equation*}
\min_{\theta,\phi}\ \Big( L_{\text{Actor}}(\phi) + L_{\text{Critic}}(\theta,\psi) \Big)
\quad\text{versus}\quad
\max_{\psi}\ \big\{-L_{\text{Prox}}(\psi;\theta)\big\}.
\end{equation*}
In practice, we approximate this problem by alternating gradient updates. Within each training epoch, we first strengthen the adversary by performing $K_{\text{adv}}$ steps of gradient descent on $L_{\text{Prox}}(\psi;\theta)$ to find possible worst-case contact points. We then update the actor and critic by performing gradient descent on $L_{\text{Actor}}(\phi)$ and $L_{\text{Critic}}(\theta,\psi)$. This alternating procedure seeks a saddle point where the critic reduces the viscosity-inequality violations at the adversarial contact points and the proximal network cannot further increase the violations. Algorithm \ref{alg:vgac} outlines the overall scheme.% 
\footnote{SGD denotes stochastic gradient descent; each SGD step corresponds to one optimizer update using a sampled minibatch.}

\begin{algorithm}[tbhp] % may need to use [tbhp] or [H]
\caption{Viscosity-Informed Generative Actor-Critic}
\label{alg:vgac}
\begin{algorithmic}[1]
\STATE Initialize actor $\pi_\phi$, critic $V_\theta$, proximal net $P_\psi$.
\FOR{$k=1,2,\dots$}
\STATE Collect on-policy trajectories with $\pi_\phi$ for $S$ steps and compute the target values and advantages, forming the dataset $\mathcal{D}_{\text{on-policy}}=\{(x_t,\,c_t,\,\Delta\ell_t,\,\hat V_t,\,\hat{A}_t)\}_{t=0}^{S-1}$.
\STATE Induce the empirical on-policy law $\nu_{\mathrm{on},k}$ from $\mathcal{D}_{\text{on-policy}}$.
\FOR{\texttt{PPO\_epoch} $=1,\dots,N$}
\FOR{\texttt{SGD\_step} $=1,\dots,Z$}
\STATE Sample an on-policy PPO/TD minibatch of size $B$ from $\mathcal{D}_{\text{on-policy}}$.
\STATE Construct the viscosity/prox anchors $\{x_i\}_{i=1}^{B}$ via the mixing law $\nu_{X,k}$.
\STATE Sample the set \(\mathcal M=\{M_q\}_{q=1}^K\).
\STATE \textbf{(Prox)} For each $(x_i,M_q)$ compute contacts $\zeta^{b}_{i,q}=P_\psi(x_i,M_q,b)$, \\
then take $K_{\text{adv}}$ gradient descent steps on
$L_{\text{Prox}}(\psi;\theta)$ (with $(\theta,\phi)$ fixed).
\STATE \textbf{(Critic)} With $\psi$ fixed (detach $\zeta^b$), compute $L_{\text{Visc}}(\theta,\psi)$ on anchors $\{x_i\};$ \\
update $\theta$ by gradient descent on $L_{\text{Critic}}(\theta,\psi)$.
\STATE \textbf{(Actor)} Update $\phi$ by gradient descent on the PPO objective $L_{\text{Actor}}(\phi)$ \\
(with entropy regularization using the PPO/TD minibatch states $\{x_j\}$ and jet-alignment using anchors $\{x_i\}$ with $\zeta^b$).
\ENDFOR
\ENDFOR
\ENDFOR
\end{algorithmic}
\end{algorithm}

\subsection{Viscosity consistency of the learned critic}\label{convergence}We now establish subsequential viscosity consistency for the critic sequence on the sampled SPD-envelope jet family. Consider the following assumptions.

\begin{assumption}\label[assumption]{assumps:limit} We assume the following  regularity and optimality conditions:
\begin{romannum}
\item \label[assumption]{assump:prox-delta-opt} For sequences $(\theta_k,\psi_k)$ there exist $\delta_k\to 0^+$ such that the proximal network outputs
\[
\zeta^{-}_{k}(x,M)=P_{\psi_k}(x,M,-1),\qquad \zeta^{+}_{k}(x,M)=P_{\psi_k}(x,M,+1)
\]
are $\delta_k$-optimal minimizers/maximizers corresponding to  inf/sup-convolution problems over $\overline{\Omega}$.

\item \label[assumption]{assump:critic-regularity} The critic class $\{V_\theta\}$ is equibounded and equicontinuous on $\overline{\Omega}$
(e.g., enforced by spectral/weight-norm constraints or explicit regularization yielding a uniform modulus of continuity).

\item \label[assumption]{assump:actor_gap}
For $\{\phi_k\}$ there exist $\varepsilon_k\to 0$ such that $L_{\text{jet}}(\phi_k;\psi_k)-L_{\text{jet}}^\star(\psi_{k}) \le \varepsilon_k$.

\item \label[assumption]{assump:boundary_consistency}
Either $V_{\theta_k}=g$ on $\partial\Omega$ for every $k$, or $\Gamma=\partial\Omega$, $\operatorname{supp}\nu_\Gamma=\partial\Omega$, and 
\[
\mathbb E_{x\sim\nu_\Gamma}\!\big[\,|V_{\theta_k}(x)-g(x)|^2\,\big]\to 0.\]
\end{romannum}
\end{assumption}

Under these assumptions the following viscosity consistency result holds.

\begin{theorem}\label{thm:critic-correct}
Under \cref{assumps:structure,assumps:boundary,assump:sampling,assumps:limit}, let
\[
V_k:=V_{\theta_k}\in C(\overline{\Omega}),\qquad
P_k:=P_{\psi_k},\qquad
\bar\pi_k:=\bar\pi_{\phi_k}.
\]
Suppose $\E_{\mathcal M}\big[L_{\text{Visc}}(\theta_k,\psi_k)\big]\to0$.
Then every uniform limit \(V\) of \((V_k)\) belongs to \(C(\overline\Omega)\) and satisfies $V=g$ on $\partial\Omega$. For $(\nu_{\mathrm{cov}}\otimes\nu_{\mathcal M})$-a.e. $(x,M)$, there exist exact inf-/sup-envelope contacts \(\zeta_x,\zeta_x^\star\) of $V$ such that each inequality below holds whenever its contact
lies in $\Omega$
\[
F\big(\zeta_x,V(\zeta_x),M(x-\zeta_x),-M\big)\ge0,\qquad
F\big(\zeta_x^\star,V(\zeta_x^\star),-M(x-\zeta_x^\star),M\big)\le0.
\]
\end{theorem}

\begin{proof}
See \cref{app:critic-correct}.
\end{proof}

\begin{remark}
Consistency is established only on the sampled SPD-envelope jet family. By \cref{cor:envelope-necessary}, any viscosity solution satisfies $\mathcal J(V)=0$, i.e.\ the viscosity inequalities hold on this test family. In general however, the converse may fail since envelope tests generate only jets with second-order part $A=\pm M$, $M\in\mathbb S_{++}(n)$. In practice, sampling $M$ with $\alpha_{\min}I\preceq M\preceq\alpha_{\max}I$ yields uniformly strongly convex/concave kernels and a well-conditioned contact search.

Barles and Souganidis \cite[Thm.~2.1]{barles1991} prove convergence of monotone, stable, and consistent schemes under comparison. Our proof likewise passes contact inequalities to limits, but establishes only sampled SPD-envelope consistency without assuming scheme monotonicity. 
\end{remark}

\begin{remark}
By \cref{prop:strong-dirichlet}, any subsequential uniform limit $V$ of $(V_k)$ arising in \cref{thm:critic-correct} that is also a viscosity solution of \eqref{eq:HJB} must coincide with the unique continuous viscosity solution of the Dirichlet problem \eqref{eq:HJB}. In particular, if every subsequential uniform limit of $(V_k)$ is a viscosity solution of \eqref{eq:HJB}, then the entire sequence $V_k$ converges uniformly on \(\overline{\Omega}\) to that solution.
\end{remark}

%%%%%%%%%%%%%%%%%%%%%%%%%%%%%%%%%%%%%%%%%%%%%%%%%%%%%%%%%%%%%%%%%%%%%%%%%%%%%%%%%%%%%%%%

\section{Experiments}\label{sec:experiments}We evaluate V-GAC on stationary, degenerate elliptic HJB equations: the one-dimensional Eikonal equation, with an explicitly specified differential operator; stochastic rigid-body stabilization and a minimum-time Van der Pol control problem, both with analytically specified controlled dynamics and HJB operators; Gymnasium MuJoCo continuous-control benchmarks \cite{towers2024} for high-dimensional model-free evaluation; and Safety Gymnasium navigation~\cite{ji2023}, which provides safety-critical evaluation under perturbations.

On the benchmark tasks we compare PPO \cite{schulman2017}, the second-order variant of HJBPPO \cite{mukherjee2023}, and V-GAC. Unlike the Euler and Van der Pol examples, the MuJoCo and Safety Gymnasium benchmarks are treated as continuing discounted control problems on noncompact state spaces, with finite episode truncations used for simulation and evaluation. For these benchmarks, the coefficients in \eqref{eq:Hpi-def} are not available in closed form; instead, \(f\) and \(a\) are estimated from one-step simulator increments by local first and second-moment matching, and \(\ell\) is taken from the environment. These estimates define the policy-conditioned operator used in the benchmark losses. An exit target and boundary data are not part of the underlying continuing discounted-control formulation. Instead, the Bellman operator on the full state space is contractive in the bounded-value class and therefore admits a unique fixed point. Under a
controlled-diffusion approximation, its small-step consistency limit is
the stationary HJB equation, whose solution is selected by comparison
in an admissible growth class rather than by Dirichlet data
\cite[sect.~5.D]{crandall1992}. Accordingly, these benchmarks test
V-GAC's sampled enforcement of interior viscosity inequalities outside
the scope of \cref{thm:critic-correct}.

All control experiments use $\lambda_{\text{TD}}=1/2$, except for the Van der Pol problem, which uses $\lambda_{\text{TD}}=1$. The target-value recursion of
\cref{sec:method} is using generalized advantage estimation (GAE \cite{schulman2017})
on the cost side. Let $(x_t,\,c_t,\,\Delta\ell_t,\,\hat V_t,\,\hat{A}_t)_{t=0}^{S-1}$ denote a rollout
segment, where $\Delta\ell_t$ is the one-step discretized cost contribution
entering the update, and let $m_t\in\{0,1\}$ denote the continuation mask used
in the recursion. Then the target values satisfy
\[
\hat V_t
=\Delta\ell_t+
\gamma m_t\Bigl((1-\lambda_{\text{GAE}})V_\theta(x_{t+1})
+\lambda_{\text{GAE}}\hat V_{t+1}\Bigr),
\qquad
\hat V_S:=V_\theta(x_S),
\]
or, equivalently,
\[
\hat V_t
=V_\theta(x_t)+
\sum_{l=0}^{S-t-1}
(\gamma\lambda_{\text{GAE}})^l
\Bigl(\prod_{j=0}^{l-1}m_{t+j}\Bigr)\Bigl(\Delta\ell_{t+l}
+\gamma m_{t+l}V_\theta(x_{t+l+1})
- V_\theta(x_{t+l})\Bigr).
\]

For PPO and HJBPPO, we use the hyperparameters reported in \cite{mukherjee2023}, following the default configurations wherever applicable. For the Euler and Van der Pol examples, the reported trajectories and policy/value diagnostics are computed from on-policy rollouts during training. For the benchmark tasks, evaluation metrics reported in tables are averaged over 150 evaluation episodes per seed. 

We conclude the experimental section with a cross-benchmark endpoint summary under nominal and perturbed evaluation. Detailed experimental parameters are reported in \cref{sec:sm-hparams}.

\subsection{One-dimensional Eikonal equation}\label{exp:eikonal}Before proceeding to the controlled examples, a one-dimensional problem is considered to isolate the solution-selection issue discussed in \cref{sec:pinns}. Specifically, consider the Eikonal Dirichlet problem
\begin{equation}
\label{eq:eikonal-1d}
\begin{cases}
    \left\lvert\nabla V(x)\right\rvert-1=0,
        & x\in(0,1),\\
    V(0)=V(1)=0.
\end{cases}
\end{equation}
Its unique viscosity solution is the distance to the boundary,
\(V^*(x)=\min\{x,1-x\}\), which is nondifferentiable at \(x=1/2\).
However, \eqref{eq:eikonal-1d} admits infinitely many solutions \(V\in W^{1,\infty}(0,1)\) that satisfy the equation almost everywhere. We approximate the solution of \eqref{eq:eikonal-1d} using both
V-GAC and a strong-form PINN.

The unit-scale Kru\v{z}kov transformation
\(v(x)=1-e^{-V(x)}\) places \eqref{eq:eikonal-1d} in the strictly proper form
\(v(x)+\lvert\nabla v(x)\rvert-1=0\) in \((0,1)\), with
\(v(0)=v(1)=0\). For comparison on the same transformed equation, the PINN is trained
using the corresponding strong-form residual \(\mathcal{R}_\theta(x)
=v_\theta(x)+\lvert\nabla v_\theta(x)\rvert-1
=e^{-V_\theta(x)}
\bigl(\lvert\nabla V_\theta(x)\rvert-1\bigr)\).

For this example, we use an actor-free V-GAC implementation; consequently,
the TD, jet-alignment, and PPO loss terms are omitted. The resulting implementation retains the critic \(V_\theta\) and proximal network
\(P_\psi\), together with the viscosity and boundary losses for training \(V_\theta\) and the adversarial, envelope, and proximal-optimality losses for training \(P_\psi\). Both V-GAC and the PINN were trained for \(10{,}000 \) steps.

\begin{figure}[htbp]
    \centering
    \includegraphics[width=\textwidth]
        {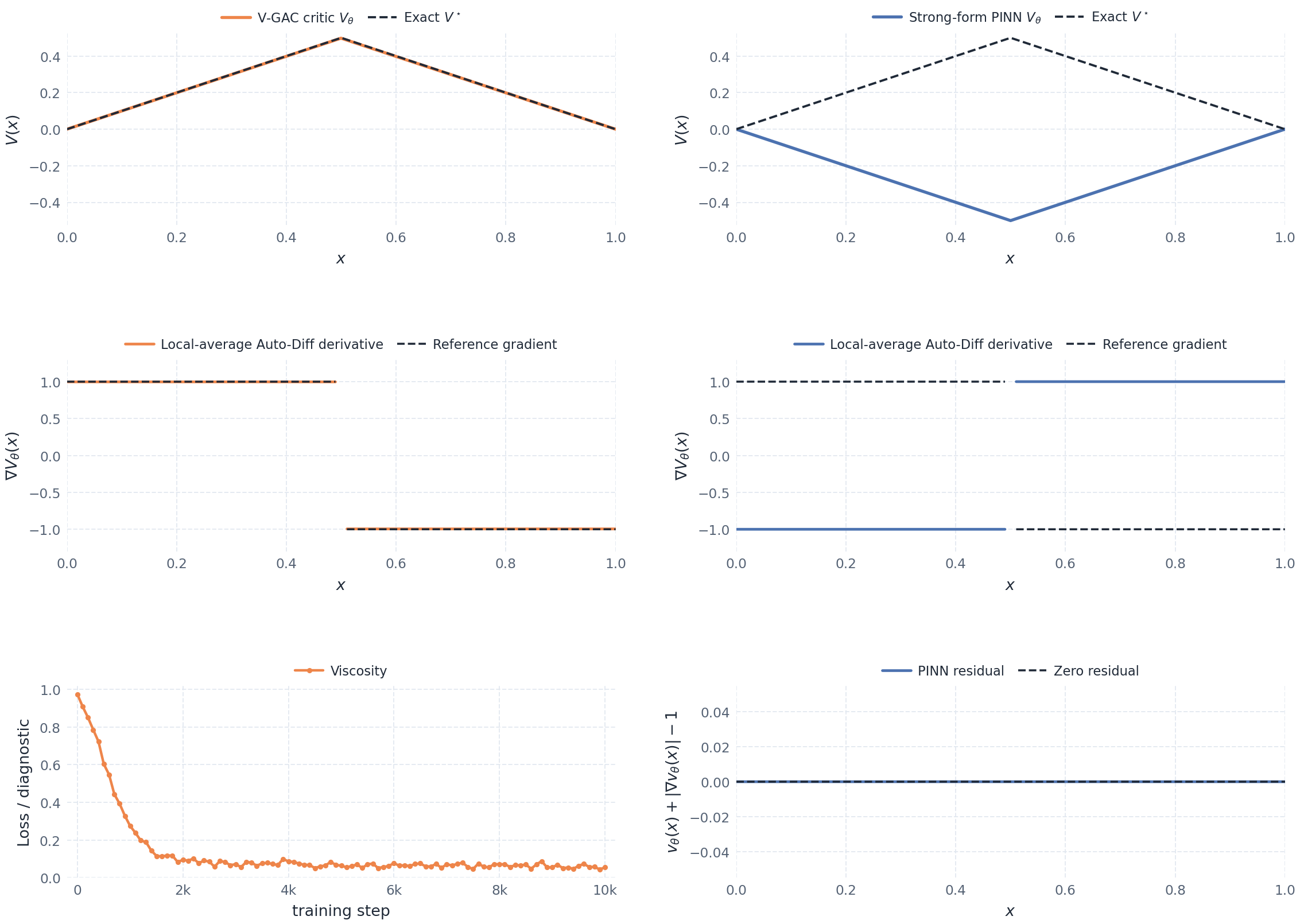}
    \caption{One-dimensional Eikonal experiment comparing V-GAC (left column)
    and a strong-form PINN (right column). The top row shows the approximations \(V_\theta\) obtained using the Kru\v{z}kov-transformed formulation, the middle row shows local-average gradients, and the bottom row shows
    the corresponding training diagnostics.}
    \label{fig:eikonal_vgac_pinn}
\end{figure}

Inspecting \cref{fig:eikonal_vgac_pinn}, the V-GAC critic closely follows
\(V^*\), while its local-average gradient approaches \(+1\) and \(-1\) on the two sides of the cusp, respectively, and the viscosity loss decreases during training. In contrast, the displayed PINN approximation approaches \(-V^*\) while satisfying the endpoint conditions and transformed strong-form residual to numerical precision. The viscosity supersolution violation occurs at the interior minimum, a set of measure zero that is not detected by the mean-square strong-form residual but is detected by the viscosity supersolution test.

The displayed PINN outcome is not universal across initializations. Over ten random seeds, each trained for \(10{,}000\) steps, the PINN with the transformed residual produced five runs classified as viscosity-like and five as nonviscosity Sobolev-like. With the untransformed residual, the corresponding counts were four and six. Thus neither residual formulation reliably excluded nonviscosity Sobolev-like outcomes. In contrast, V-GAC converged to the viscosity solution for all ten seeds. The settings common to PINN and V-GAC, including the critic architecture, optimizer, and anchor sampling law, were matched across methods and held fixed across the ten random seeds. 

\subsection{Euler Dynamics}\label{exp:euler}The next example we evaluate V-GAC on is a discounted stochastic stabilization problem for a rigid body in principal axes. The state is the angular velocity $X_t=\omega(t)\in\R^{3}$ and the control is a bounded torque $u(t)\in U=[-15,15]^3$. 
The controlled It\^o dynamics are
\[
d\omega_t = f(\omega_t,u_t)\,dt + \Sigma\,dW_t, \qquad \Sigma=0.05\,I_3,
\]
with drift
\[
f(\omega,u)=\begin{bmatrix}
\frac{I_2-I_3}{I_1}\omega_2\omega_3+\frac{u_1}{I_1}\\[2pt]
\frac{I_3-I_1}{I_2}\omega_3\omega_1+\frac{u_2}{I_2}\\[2pt]
\frac{I_1-I_2}{I_3}\omega_1\omega_2+\frac{u_3}{I_3}
\end{bmatrix},\qquad (I_1,I_2,I_3)=(1,2,3).
\]
The running cost is $\ell(\omega,u)=\sum_{i=1}^3 \omega_i^2 + 0.1\sum_{i=1}^3 u_i^2$, with discount $\beta=0.8$. 
The PDE domain is
\(\Omega=\{5\times10^{-3}<\|\omega\|<5\}\).
During training, the target region
\(\mathcal T=\{\|\omega\|\le5\times10^{-3}\}\) is sampled with
positive probability so that the critic can learn its terminal value.%
\footnote{This auxiliary target-region supervision is an
implementation choice used for numerical stabilization in place of
direct sampling from \(\Gamma=\partial\mathcal T\); it does not alter
the PDE formulation.} The boundary data are \(g=0\) on
\(\|\omega\|=5\times10^{-3}\) and \(g=50\) on \(\|\omega\|=5\),
the latter representing the outer-exit penalty. Trajectories are simulated by Euler--Maruyama with $\Delta t=10^{-3}$. Since \(a=\Sigma\Sigma^\top=0.05^2I_3\) and \(\partial\Omega\) is smooth, the HJB is uniformly elliptic, and the criterion in \cref{remark:boundary-barriers} yields strong attainment of \(g\).

\begin{figure}[htbp]
\centering
\includegraphics[width=0.8\textwidth]{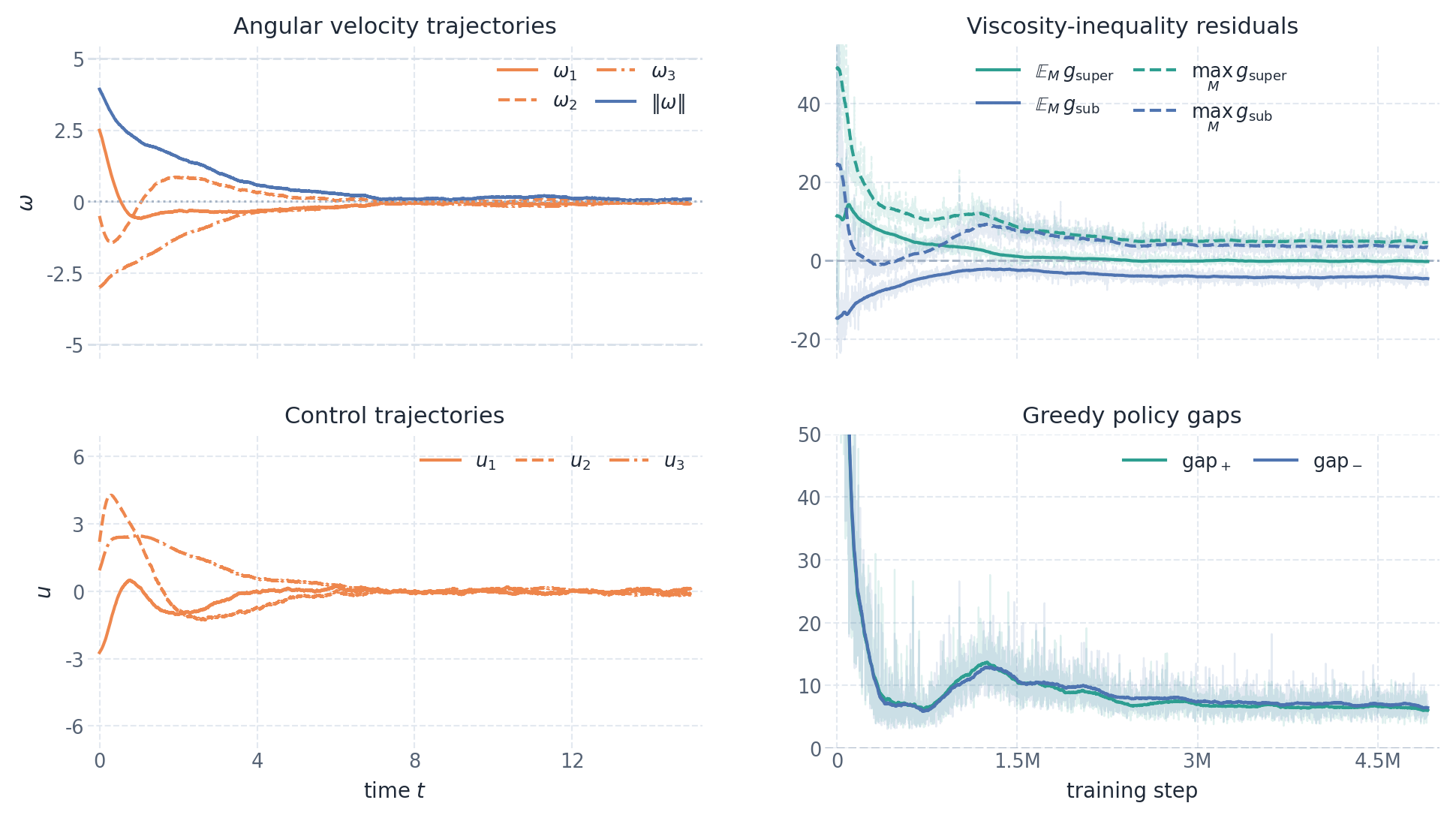}
\caption{Stochastic rigid-body stabilization. The figure shows a sample state trajectory in angular-velocity coordinates $\omega_1,\omega_2,\omega_3$ and $\|\omega\|$, initialized at $\omega(0)=[2.5,-0.5,-3.0]$ (top left), together with the corresponding control trajectory $u=(u_1,u_2,u_3)$ (bottom left). The viscosity-inequality residuals $\mathbb{E}_{M} g_{\text{super}}$, $\mathbb{E}_{M} g_{\text{sub}}$ and their $\max_M$ analogues are shown in the top-right panel, while the greedy policy gaps $\text{gap}_{+}$ and $\text{gap}_{-}$, defined in \eqref{eq:gap} and averaged over the sup- and inf-envelope jet families, respectively, are shown in the bottom-right panel.}
\label{fig:euler_traj_violations}
\end{figure}

The main complication here, as throughout the experiments, is that the critic $V_{\theta}$ is trained against a policy-conditioned HJB operator rather than the exact minimizing Hamiltonian. Therefore, reducing the critic loss alone does not guarantee consistency with the true HJB equation; consistency additionally requires the greedy gap to be small at the sampled jets. In \cref{fig:euler_traj_violations}, the quantities \(g_{\text{super}}\) and \(g_{\text{sub}}\) measure violations of the supersolution and subsolution viscosity inequalities at the envelope-generated jets. Therefore, their decay toward zero, both on average and in the worst sampled direction, indicates that the critic better satisfies the policy-conditioned viscosity inequalities on the sampled SPD-envelope jet family. The greedy gaps play a different role: a small gap means that the actor is nearly minimizing the Hamiltonian at those same jets, so the policy-conditioned operator is close to the true HJB operator there. 

\begin{figure}[htbp]
\centering
\includegraphics[width=\textwidth]{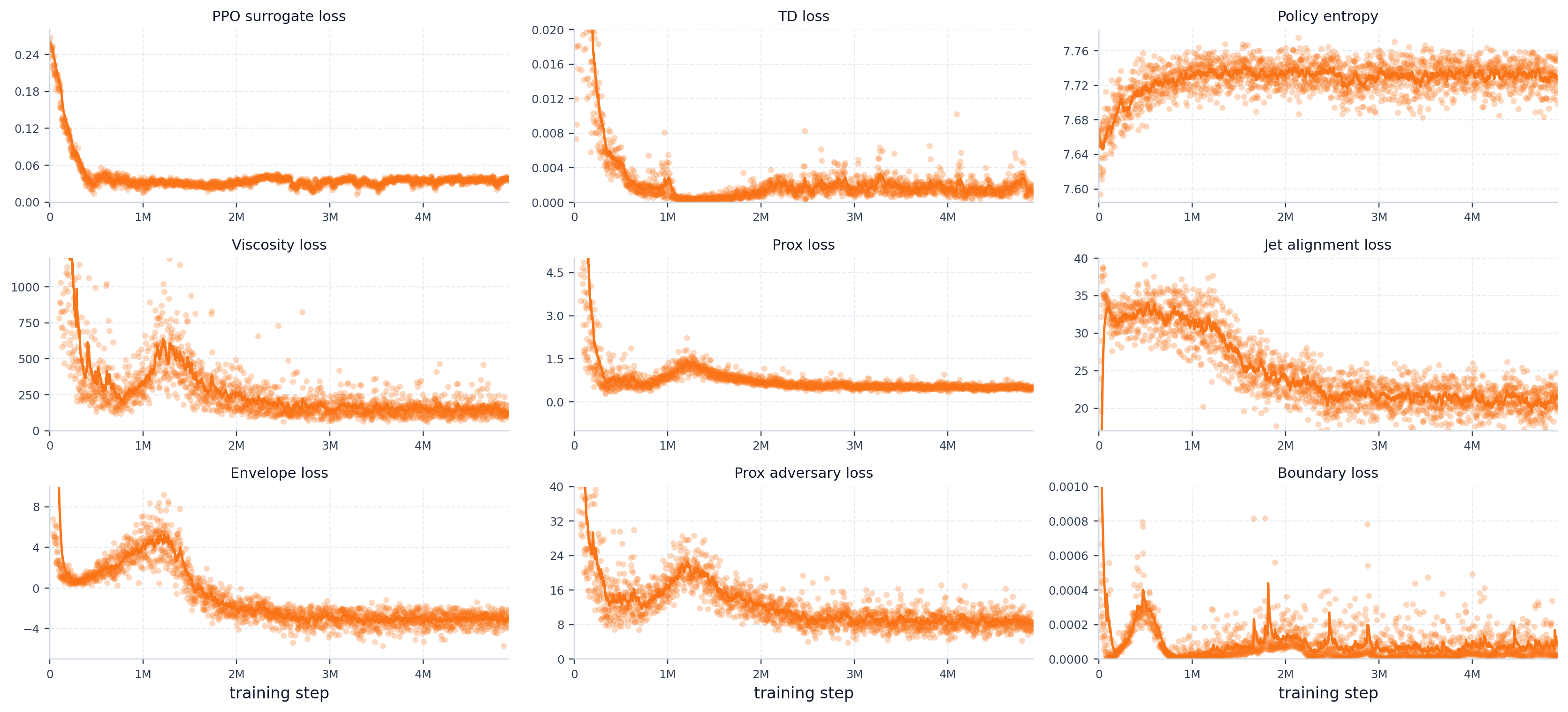}
\caption{Training losses for the stochastic rigid-body stabilization experiment. The panels show the policy-surrogate loss, TD loss, policy entropy, viscosity loss, proximal-optimality loss, and jet-alignment loss, together with the raw envelope term, proximal adversary term, and boundary loss, all plotted against training steps. Points denote logged values and solid curves denote moving averages.}
\label{fig:euler_losses}
\end{figure}

In \cref{fig:euler_losses}, the viscosity loss and proximal-optimality loss are initially large and then decrease substantially after a secondary transient, reflecting improved satisfaction of the PDE-side and contact-optimality constraints. The bottom-row panels provide additional context for this behavior. The raw proximal adversary term follows the viscosity loss, the boundary loss remains close to zero after a short transient, and the envelope term settles after the initial search phase. Since the envelope term is not sign-definite, its stabilization is more informative than its sign alone.

\subsection{Minimum time Van der Pol oscillator}\label{exp:vdp}We next consider Problem 2 of \cite{cristiani2010}, namely the controlled Van der Pol oscillator
\[
\dot y_1 = y_2,\qquad
\dot y_2 = -y_1 + y_2(1-y_1^2)+u,\qquad U=[-1,1].
\]
This is a deterministic minimum-time problem on \(\R^2\) with bang-bang optimal control. Small-time controllability at the target provides continuity of the minimum-time value there \cite[Chap.~IV, Props.~1.2 and~1.6]{bardi1997}. Following \cite{cristiani2010}, we numerically truncate the whole-space problem to \([-2,2]^2\) and, for rollout termination, replace the singleton target by a ball of radius \(0.05\) centered at the origin. Rather than learning the time-to-target function $T$ directly, we learn the scaled Kru\v{z}kov transform $v(y)=1-e^{-\beta T(y)}$ with $\beta=0.1$, so that
\[
\beta v(y)-\inf_{c\in[-1,1]}\{\beta+f(y,c)\cdot\nabla v(y)\}=0
\quad\text{in }(-2,2)^2 \setminus \mathcal{T},
\]
with $\mathcal{T}=\{\|y  \|\le 5\times 10^{-2}\}$ and $v=0$ on $\mathcal{T}$. Since $\beta>0$, the transformed operator is strictly proper in the sense of \cref{def:structure}. At the artificial outer boundary, exit before reaching the target is
treated as terminal failure, corresponding to transformed value \(1\). For the truncated problem, the target-reachable set consists of states from which \(\mathcal T\) can be reached in finite time before outer exit; elsewhere \(T=+\infty\) and \(v=1\). Accordingly, the critic is parameterized in transformed coordinates, $v_\theta(y)\in[0,1)$, and the boundary terms enforce $v_\theta\approx 0$ on the target and $v_\theta\approx 1$ on exit states. Rollouts are deterministic and are integrated by a fourth-order Runge--Kutta (RK4) scheme with macro-step $\Delta t=0.05$ and internal step $10^{-3}$. Over a realized step length $\Delta t$, since $\ell(y,u)\equiv \beta$, the total one-step cost contribution is
\[
\Delta\ell
=\bigl(1-e^{-\beta\Delta t}\bigr)
+\mathbf{1}_{\{\tau_{\Omega}<\tau_{\mathcal T}\}}e^{-\beta\Delta t}.
\]
Thus an outer-boundary hit yields a one-step cost contribution of $1$, whereas target hits incur no exit penalty. Rollout actions are sampled from the Gaussian policy, while the viscosity, proximal, and jet losses evaluate the policy-conditioned Hamiltonian at the deterministic feedback map $\bar\pi_\phi(y)=u_{\max}\tanh(\mu_\phi(y))$. Time-limit rollout truncations are bootstrapped rather than imposed as terminal boundary data.

\begin{figure}[htbp]
\centering
\includegraphics[width=0.85\textwidth]{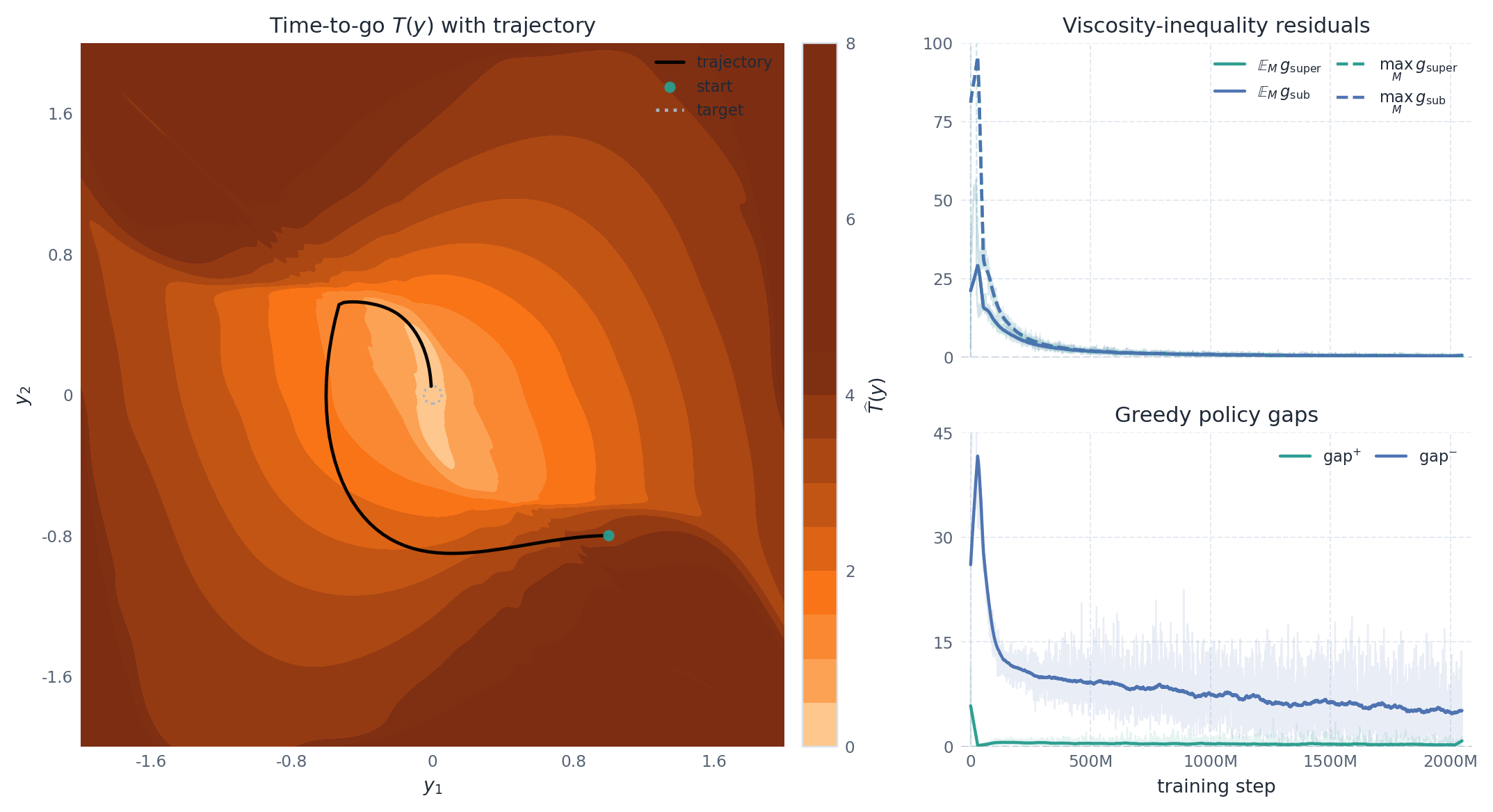}
\caption{Gaussian-actor results for the minimum-time Van der Pol problem. Left: reconstructed time-to-go $T(y)=-\beta^{-1}\log(1-v_\theta(y))$ on $[-2,2]^2$, together with a sampled closed-loop trajectory from $y(0)=[1,-0.8]$ to the target (origin). Right top: mean and worst-case viscosity-inequality residuals over the sampled jet family during training. Right bottom: corresponding policy gaps.}
\label{fig:vdp_gaussian_violations}
\end{figure}

\cref{fig:vdp_gaussian_violations} shows that the learned time-to-go field exhibits a locus of nondifferentiability across the state plane, consistent with a switching curve. Although the Gaussian actor is continuous, its mean action saturates near the control bounds away from a narrow transition layer, so the policy can approximate the underlying bang-bang structure while preserving the associated nonsmooth feature in the value function. The mean and worst-case viscosity residuals decay rapidly while the policy gaps decrease over training, indicating that the critic approaches sampled viscosity consistency. For the displayed trajectory, the simulated time-to-go is $\tau=3.814$, compared with the Pontryagin maximum principle value $3.837$ reported in \cite{cristiani2010}.

We computed a grid-based reference solution using a semi-Lagrangian scheme \cite{bardi1997} to obtain a reference solution, and evaluated the error in the value function $T$ on 3999 reachable nodes (excluding the target set), obtaining $L^2_{\mathrm{rms}}=5.568\times10^{-1}$, $L^\infty=8.072\times10^{-1}$, relative $L^2=1.728\times10^{-1}$, and relative $L^\infty=1.907\times10^{-1}$.
The reference solver used a time step of $5\times10^{-3}$, natural-neighbor interpolation on a quasi-uniform simplex grid, and $\beta_{\text{ref}}=1$ so that $T_{\mathrm{ref}}(y) = -\log\big(1 - v_{\mathrm{ref}}(y)\big)$.

\begin{figure}[htbp]
\centering
\includegraphics[width=\textwidth]{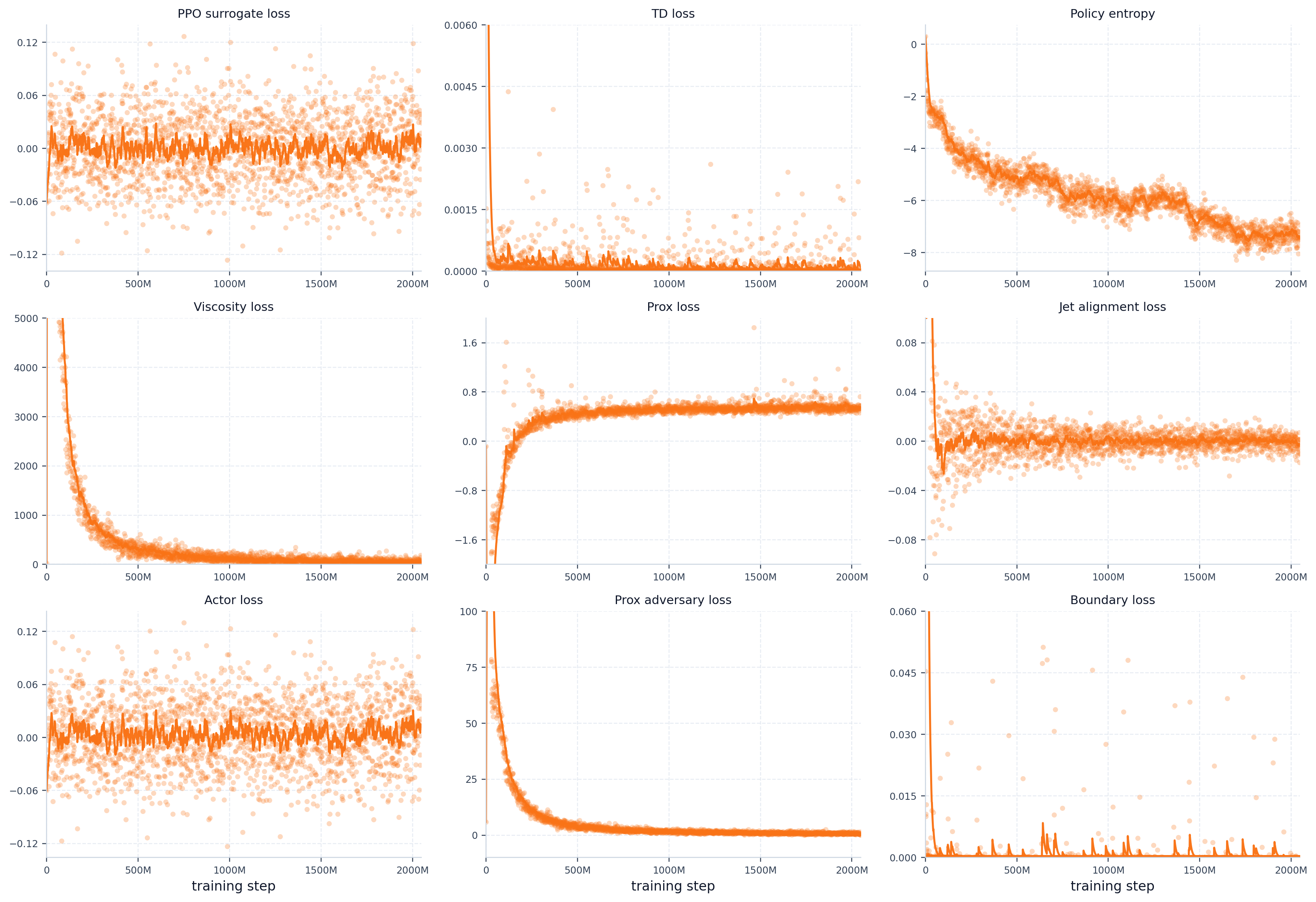}
\caption{Training losses for the time-optimal Van der Pol problem. Top row: PPO surrogate loss, TD loss, and policy entropy. Middle row: viscosity loss, total prox loss, and jet-alignment loss. Bottom row: total actor loss, raw proximal adversary term, and boundary loss. Light markers show raw logged values and solid curves show moving averages.}
\label{fig:vdp_gaussian_losses}
\end{figure}

Turning to the training losses, \cref{fig:vdp_gaussian_losses} shows that the critic-side quantities stabilize rapidly. The TD loss, viscosity loss, and raw proximal adversary term all decrease substantially, while the boundary loss remains close to zero after a short transient. Since the envelope term is disabled in this run, the total prox loss reflects the balance between the adversarial term and the proximal-optimality regularizer. Thus, its sign is less informative than the decay of the raw adversary term.

\subsection{MuJoCo locomotion and high-dimensional control}\label{exp:loco}The MuJoCo suite spans increasingly contact-sensitive locomotion tasks:
HalfCheetah-v5, Hopper-v5, Walker2d-v5, and Ant-v5. We also include
Humanoid-v5 as a higher-dimensional stress test where stable forward
motion requires coordinated whole-body control and repeated balance
recovery \cite{towers2024}. In this setting, the main empirical question is whether the viscosity-informed critic yields smaller mean greedy gaps and improved robustness under off-manifold perturbations.  The reported quantities are empirical averages or maxima of sampled envelope-jet hinge violations and greedy gaps, and therefore serve as surrogate consistency measures rather than exact norms of viscosity error. The reported greedy gap is the average of the sup- and inf-envelope gap values.

During noisy evaluation, each step is perturbed by an Euler--Maruyama increment $\sigma_{\text{dyn}}\sqrt{\Delta t}\,\xi_k$, $\xi_k\sim\mathcal{N}(0,I)$; nominal evaluation uses $\sigma_{\text{dyn}}=0$. To isolate the effect of the critic objective, all methods use the same PPO backbone and training configuration within each benchmark; HJBPPO uses a strong-form HJB residual penalty, whereas V-GAC uses the viscosity-based objective of \cref{sec:visc-loss}. Unless otherwise noted, all results are reported as mean $\pm$ standard deviation over independent seeds. All reported diagnostic trajectories in \cref{fig:critic_diagnostics_mujoco,fig:training_losses_mujoco,fig:critic_diagnostics_safety,fig:training_losses_safety} are computed under Brownian-perturbed dynamics.

\begin{table}[htbp]
\centering
\caption{MuJoCo locomotion and humanoid control: episodic return (mean $\pm$ std), evaluated under nominal dynamics and Brownian-perturbed dynamics ($\sigma_{\text{dyn}}=0.10$ in normalized coordinates). Higher is better. MuJoCo locomotion results are averaged over 20 seeds; Humanoid-v5 results are averaged over 10 seeds. $n$ represents the number of dimensions for each problem.}
\label{tab:mujoco_returns}
\setlength{\tabcolsep}{5pt}
\resizebox{\linewidth}{!}{%
\begin{tabular}{@{}lcccccc@{}}
\hline
& \multicolumn{2}{c}{PPO} & \multicolumn{2}{c}{HJBPPO} & \multicolumn{2}{c}{V-GAC} \\
Task & Nominal & Noisy & Nominal & Noisy & Nominal & Noisy \\
\hline
HalfCheetah-v5 ($n=17$) & 1813 $\pm$ 311 & 1105 $\pm$ 137 & 1832 $\pm$ 121 & 1259 $\pm$ 262 & 1831 $\pm$ 102 & 1784 $\pm$ 130 \\
Hopper-v5 ($n=11$)      & 2120 $\pm$ 312 & 1511 $\pm$ 176 & 1990 $\pm$ 117 & 1699 $\pm$ 332 & 2003 $\pm$ 100 & 1968 $\pm$ 146 \\
Walker2d-v5 ($n=17$)    & 2388 $\pm$ 298 & 1566 $\pm$ 229 & 2412 $\pm$ 338 & 1710 $\pm$ 331 & 2341 $\pm$ 183 & 2263 $\pm$ 311 \\
Ant-v5 ($n=105$)        & 1113 $\pm$ 98  & 321 $\pm$ 97  & 1105 $\pm$ 86  & 719 $\pm$ 79  & 1080 $\pm$ 38  & 998 $\pm$ 74 \\
Humanoid-v5 ($n=348$)   & 5342 $\pm$ 618 & 2216 $\pm$ 487 & 5498 $\pm$ 572 & 3164 $\pm$ 529 & 5579 $\pm$ 541 & 4931 $\pm$ 566 \\
\hline
\end{tabular}%
}
\end{table}

\begin{figure}[htbp]
\centering
\includegraphics[width=\textwidth]{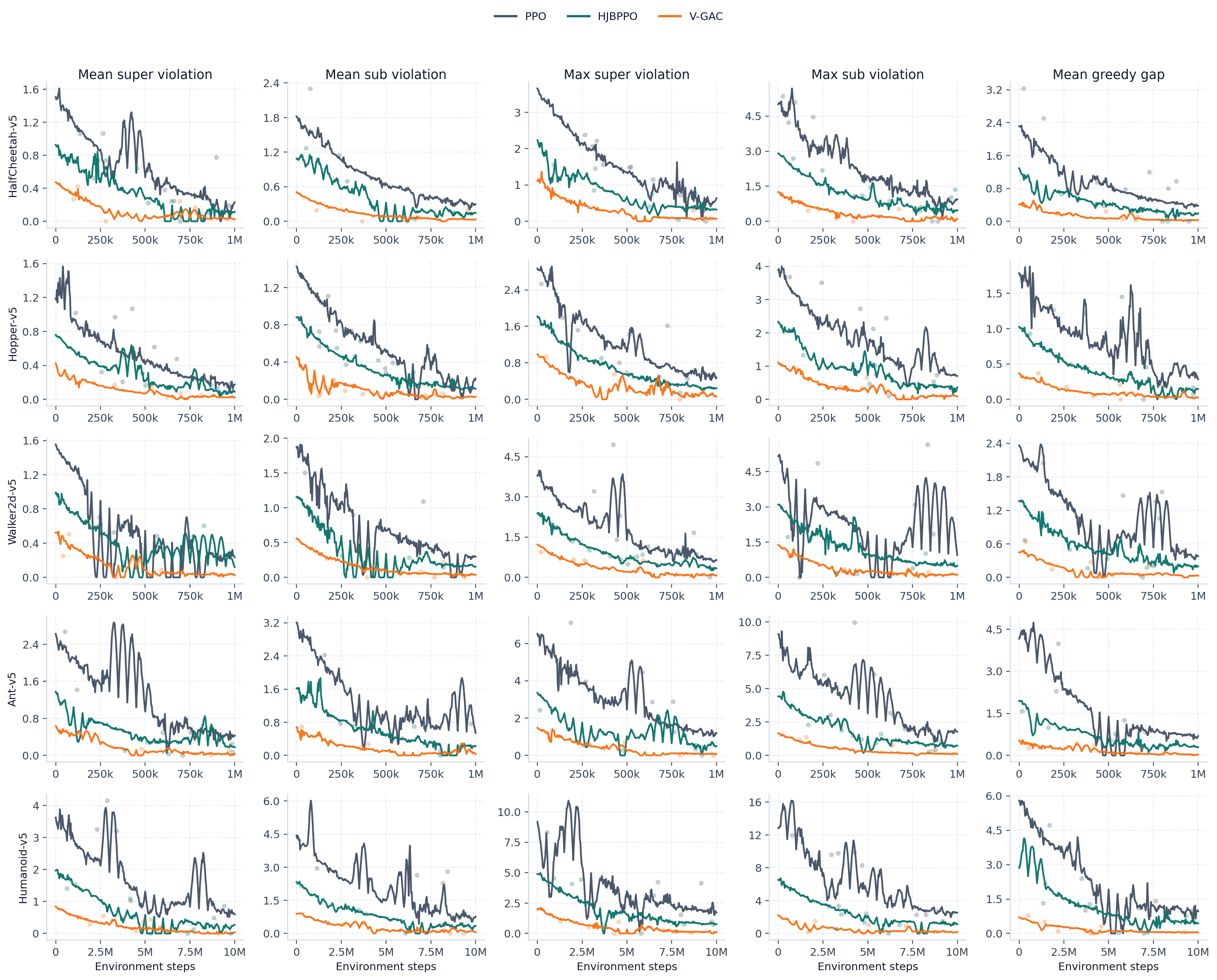}
\caption{Sampled envelope-jet diagnostic metrics over training for the MuJoCo locomotion and humanoid tasks across PPO, HJBPPO, and V-GAC.}
\label{fig:critic_diagnostics_mujoco}
\end{figure}

\cref{tab:mujoco_returns} shows broadly matched nominal returns but much smaller degradation under Brownian-perturbed evaluation for V-GAC, with the clearest gains on the more contact-sensitive and higher-dimensional tasks, especially Humanoid-v5. One plausible interpretation is that the envelope-jet tests used by V-GAC become more consequential in strongly coupled systems, where small local critic-gradient errors can have larger closed-loop effects. \cref{fig:critic_diagnostics_mujoco} shows that V-GAC reduces the sampled envelope-jet violations faster and with fewer rebounds, with HJBPPO typically intermediate and PPO the most erratic. \cref{fig:training_losses_mujoco} suggests that the shared actor-critic losses alone do not explain the robustness gap.

\begin{figure}[htbp]
\centering
\includegraphics[width=1.00135\textwidth]{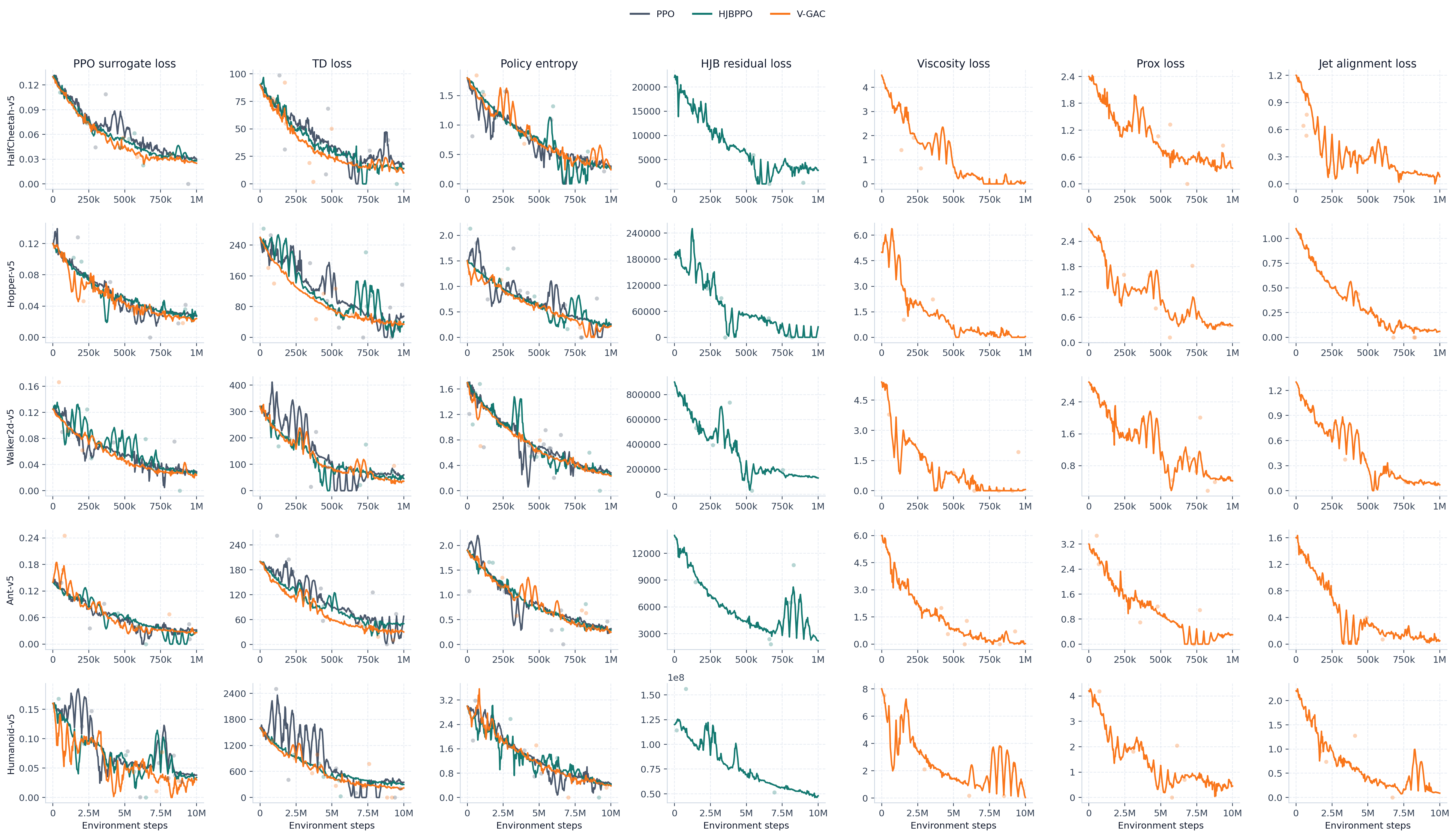}
\caption{Training-loss trajectories for the MuJoCo locomotion and humanoid tasks.}
\label{fig:training_losses_mujoco}
\end{figure}

\begin{figure}[htbp]
\centering
\includegraphics[width=\textwidth]{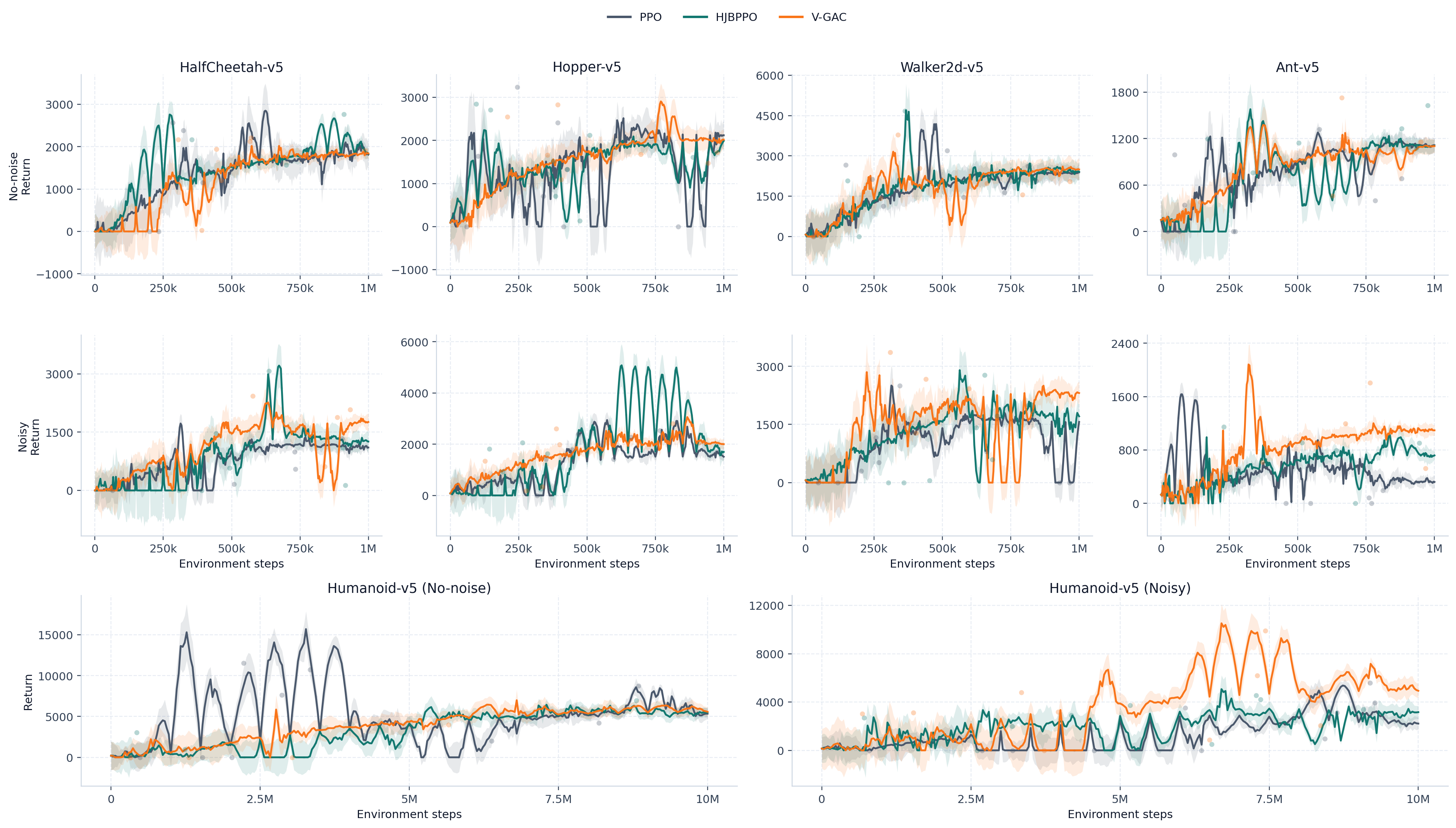}
\caption{Learning curves for HalfCheetah-v5, Hopper-v5, Walker2d-v5, Ant-v5, and Humanoid-v5 under nominal (labeled ``no-noise'') and noisy evaluation. For HalfCheetah-v5, Hopper-v5, Walker2d-v5, and Ant-v5, the upper and middle rows show nominal and noisy evaluation, respectively; the bottom row shows Humanoid-v5 under nominal and noisy evaluation.}
\label{fig:mujoco_learning_curves}
\end{figure}

\cref{fig:mujoco_learning_curves} is consistent with \cref{tab:mujoco_returns}. In the nominal setting, the three methods attain comparable returns by the end of training, whereas under noisy evaluation, the separation becomes more pronounced with V-GAC typically exhibiting the smallest degradation from nominal to noisy performance. This contrast is clearest in Humanoid-v5, where the nominal curves largely overlap by the end of training, while the noisy curves separate more visibly.

\subsection{Safety Gymnasium navigation}\label{exp:safety}For safety-critical navigation, we benchmark on Safety Gymnasium environments \cite{ji2023}, which provide both a reward and a cost signal at each step. For comparison across methods, we optimize the scalarized running cost
\begin{equation*}
\ell(x,c)= -R_{\text{mdp}} + \lambda_{\kappa}\,\kappa(x,c),
\end{equation*}
where $R_{\text{mdp}}$ is the environment reward and $\kappa$ is the Safety Gymnasium cost; $\ell$ is used in the discounted objective \eqref{eq:J}. In addition to the scalarized objective, we report the undiscounted episode cost and the violation rate.

\begin{table}[htbp]
\centering
\caption{Safety Gymnasium: episode reward and cost (mean $\pm$ std over 15 seeds) under nominal and Brownian-perturbed dynamics ($\sigma_{\text{dyn}}=0.10$ in normalized coordinates). Cost-limit: $\kappa_{\max}=25$.}
\label{tab:safetygym}
\setlength{\tabcolsep}{4pt}
\resizebox{\linewidth}{!}{%
\begin{tabular}{llcccccc}
\hline
Environment & Method
& \multicolumn{2}{c}{Reward}
& \multicolumn{2}{c}{Cost}
& \multicolumn{2}{c}{Violation (\%)} \\
& & Nominal & Noisy & Nominal & Noisy & Nominal & Noisy \\
\hline
\multirow{3}{*}{SafetyPointGoal1-v0}
& PPO    & 36 $\pm$ 4 & 13 $\pm$ 6 & 75 $\pm$ 4 & 134 $\pm$ 8 & 9 $\pm$ 1 & 57 $\pm$ 6 \\
& HJBPPO & 37 $\pm$ 4 & 15 $\pm$ 6 & 78 $\pm$ 3 & 133 $\pm$ 6 & 8 $\pm$ 1 & 53 $\pm$ 7 \\
& V-GAC  & 36 $\pm$ 3 & 34 $\pm$ 4 & 74 $\pm$ 4 & 79 $\pm$ 6 & 9 $\pm$ 1 & 12 $\pm$ 2 \\
\hline
\multirow{3}{*}{SafetyPointPush1-v0}
& PPO    & 25 $\pm$ 5 & 9 $\pm$ 6 & 69 $\pm$ 3 & 121 $\pm$ 5 & 11 $\pm$ 2 & 47 $\pm$ 2 \\
& HJBPPO & 27 $\pm$ 5 & 11 $\pm$ 6 & 67 $\pm$ 2 & 118 $\pm$ 5 & 13 $\pm$ 1 & 44 $\pm$ 3 \\
& V-GAC  & 31 $\pm$ 4 & 29 $\pm$ 5 & 70 $\pm$ 2 & 73 $\pm$ 4 & 11 $\pm$ 3 & 15 $\pm$ 4 \\
\hline
\multirow{3}{*}{SafetyCarGoal1-v0}
& PPO    & 30 $\pm$ 2 & 11 $\pm$ 1 & 60 $\pm$ 8 & 118 $\pm$ 5 & 7 $\pm$ 1 & 44 $\pm$ 1 \\
& HJBPPO & 31 $\pm$ 2 & 15 $\pm$ 2 & 63 $\pm$ 7 & 127 $\pm$ 5 & 6 $\pm$ 1 & 40 $\pm$ 1 \\
& V-GAC  & 29 $\pm$ 2 & 28 $\pm$ 2 & 63 $\pm$ 7 & 71 $\pm$ 4 & 8 $\pm$ 1 & 13 $\pm$ 2 \\
\hline
\end{tabular}
}
\end{table}

\begin{figure}[htbp]
\centering
\includegraphics[width=\textwidth]{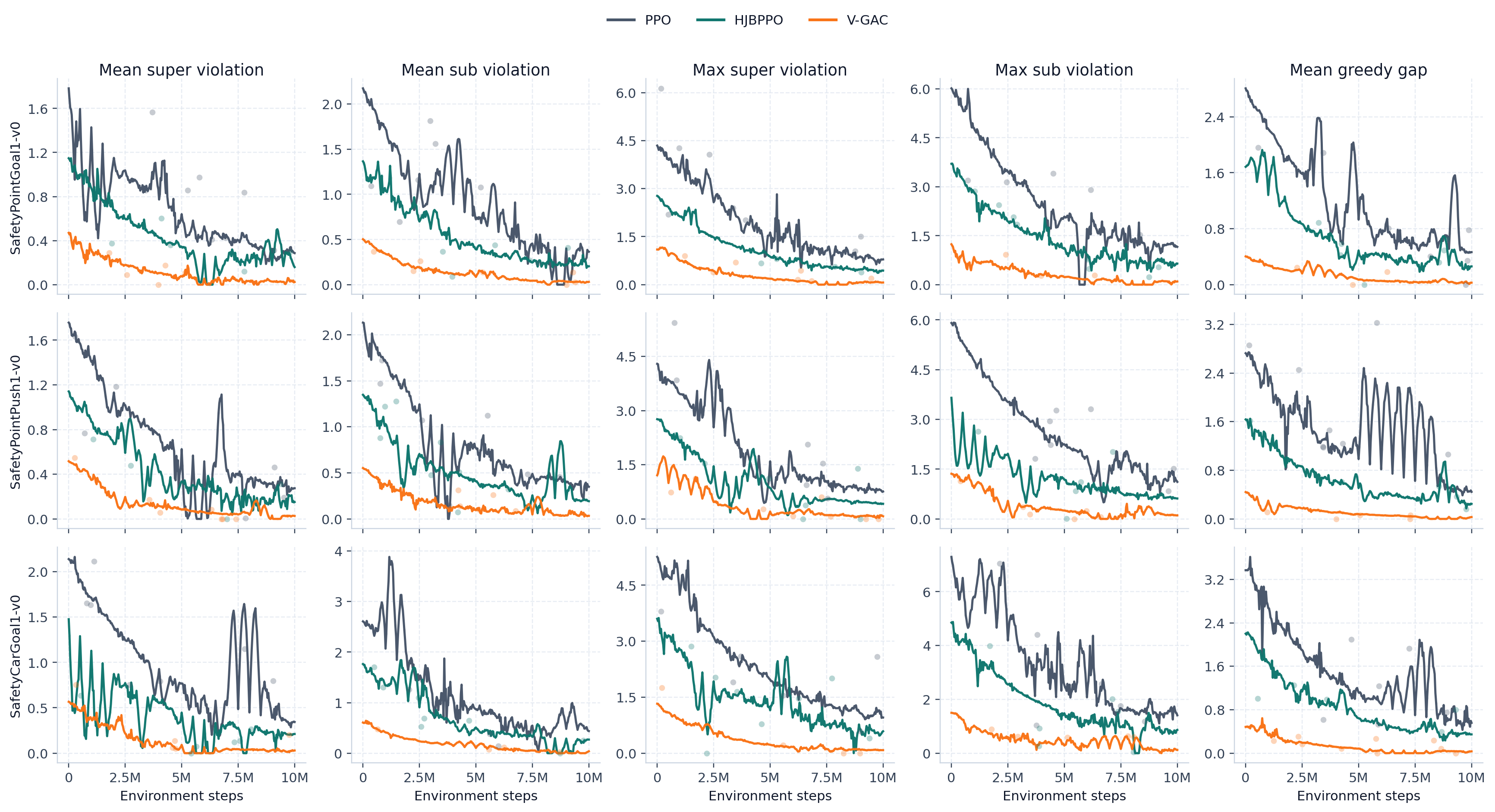}
\caption{Sampled envelope-jet diagnostic metrics over training for the Safety Gymnasium environments across the three methods.}
\label{fig:critic_diagnostics_safety}
\end{figure}
\begin{figure}[htbp]
\centering
\includegraphics[width=\textwidth]{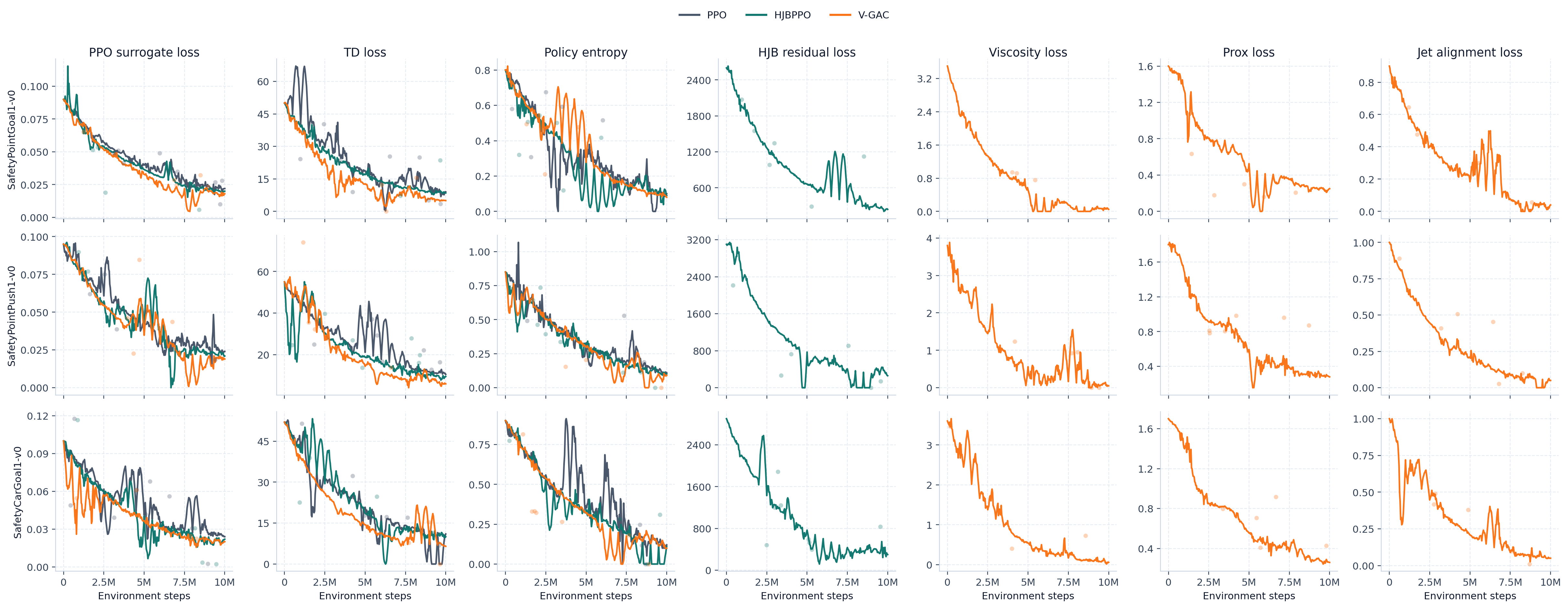}
\caption{Training-loss trajectories for the Safety Gymnasium environments across the three methods.}
\label{fig:training_losses_safety}
\end{figure}

In the nominal setting, \cref{tab:safetygym} shows broadly similar reward--cost operating points across methods, suggesting that the noisy-evaluation advantage of V-GAC is not solely due to a more conservative nominal operating point. In parallel, \cref{fig:critic_diagnostics_safety} shows that V-GAC drives the diagnostics down more steadily, while PPO and HJBPPO exhibit longer plateaus and larger rebounds. This diagnostic behavior is informative because the largest performance differences emerge under perturbed dynamics. Similarly, \cref{fig:training_losses_safety} shows that the shared losses alone do not separate the methods strongly. One empirical difference is that V-GAC continues to reduce its viscosity-specific diagnostics while maintaining stronger reward--safety performance under noisy evaluation.

\begin{figure}[htbp]
\centering
\includegraphics[width=\textwidth]{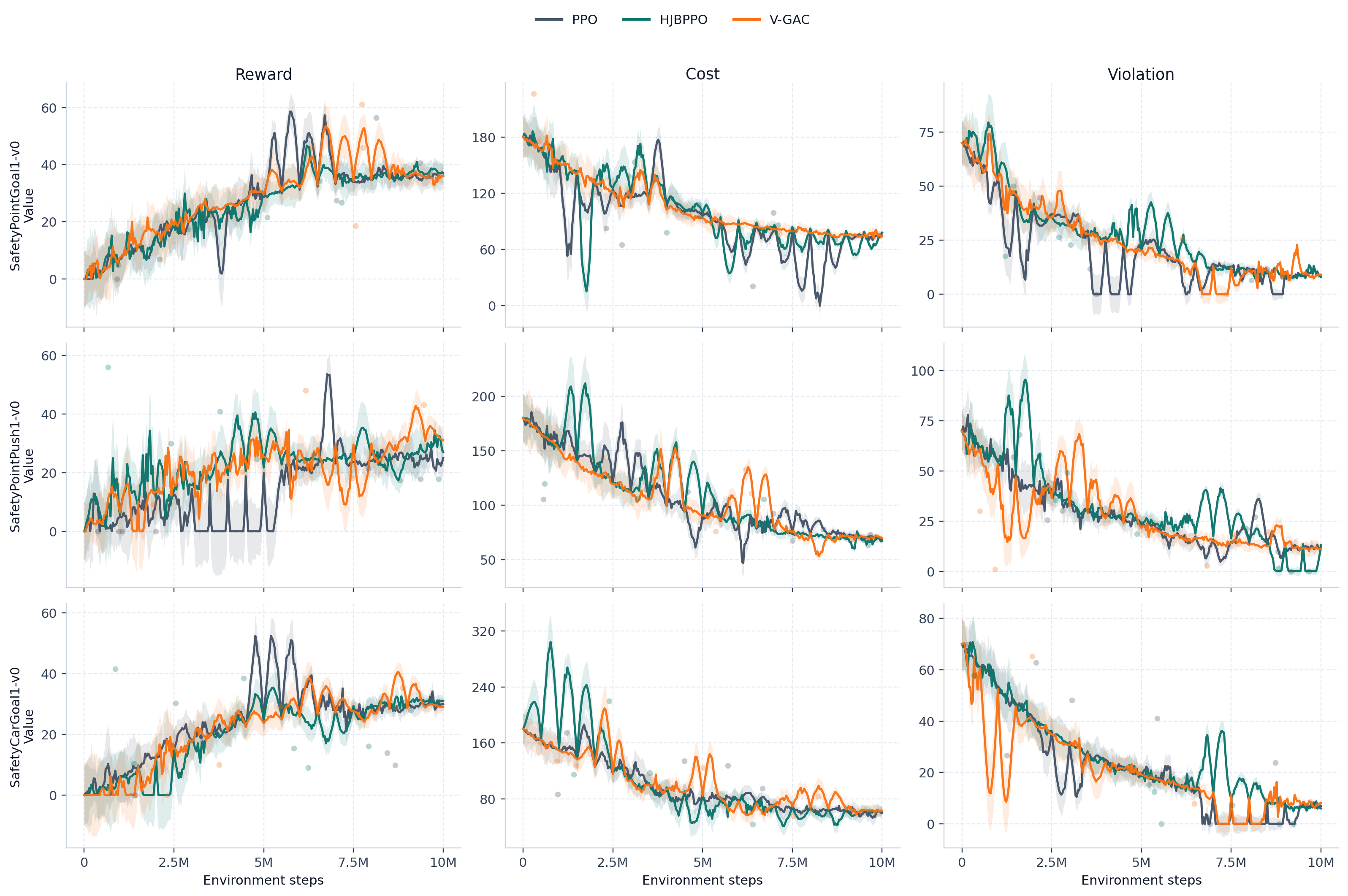}
\caption{Learning curves for the Safety Gymnasium environments under nominal evaluation. Columns show reward, episode cost, and violation count per episode; rows correspond to SafetyPointGoal1-v0, SafetyPointPush1-v0, and SafetyCarGoal1-v0.}
\label{fig:safety_learning_curves_clean}
\end{figure}

\begin{figure}[htbp]
\centering
\includegraphics[width=0.99\textwidth]{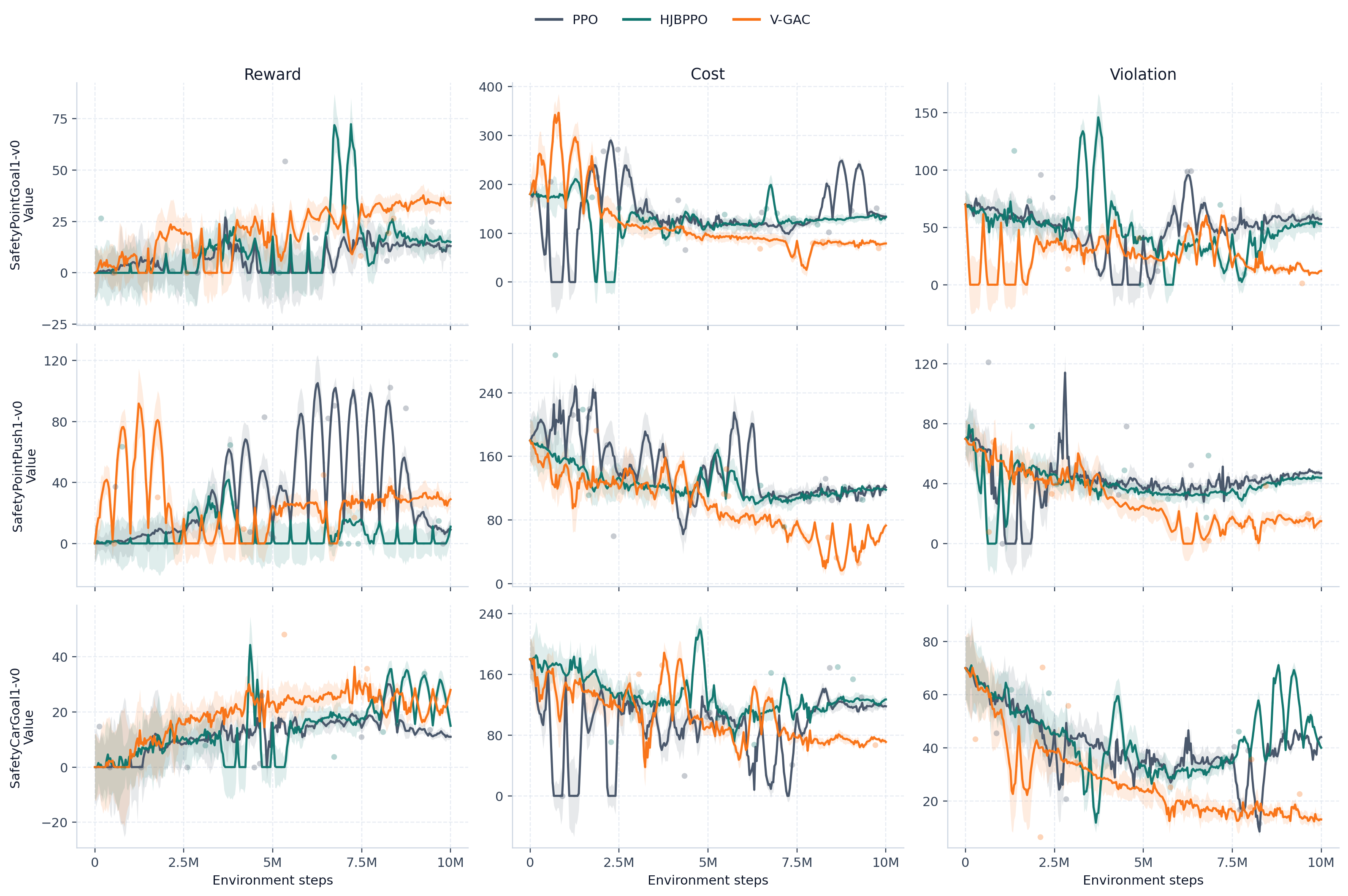}
\caption{Learning curves for the Safety Gymnasium environments under perturbed evaluation. Columns show reward, episode cost, and violation count per episode; rows correspond to SafetyPointGoal1-v0, SafetyPointPush1-v0, and SafetyCarGoal1-v0.}
\label{fig:safety_learning_curves_noisy}
\end{figure}

\cref{fig:safety_learning_curves_clean,fig:safety_learning_curves_noisy} show that the three methods are broadly comparable under nominal evaluation. In the perturbed setting, PPO and HJBPPO lose reward as episode cost and violation count increase, whereas V-GAC remains closer to the nominal operating regime. The main distinction is the joint behavior of the three reported metrics, with V-GAC more consistently attaining higher reward while obtaining lower cost and lower violation count.

%%%%%%%%%%%%%%%%%%%%%%%%%%%%%%%%%%%%%%%%%%%%%%%%%%%%%%%%%%%%%%%%%%%%%%%%%%%%%%%%%%%%%%%%

\subsection{Endpoint summaries across benchmarks}We now present the final performance across tasks of \cref{exp:loco,exp:safety} using endpoint statistics. A cross-benchmark comparison under nominal and perturbed evaluation is provided. \cref{fig:final_endpoints_returns} collects return endpoints for the locomotion and Humanoid-v5 benchmarks, while \cref{fig:final_endpoints_safety} reports reward, cost, and violation endpoints for the safety tasks. Points show individual seeds, where each point is computed as the average over 150 evaluation episodes.

\begin{figure}[htbp]
\centering
\includegraphics[width=\textwidth]{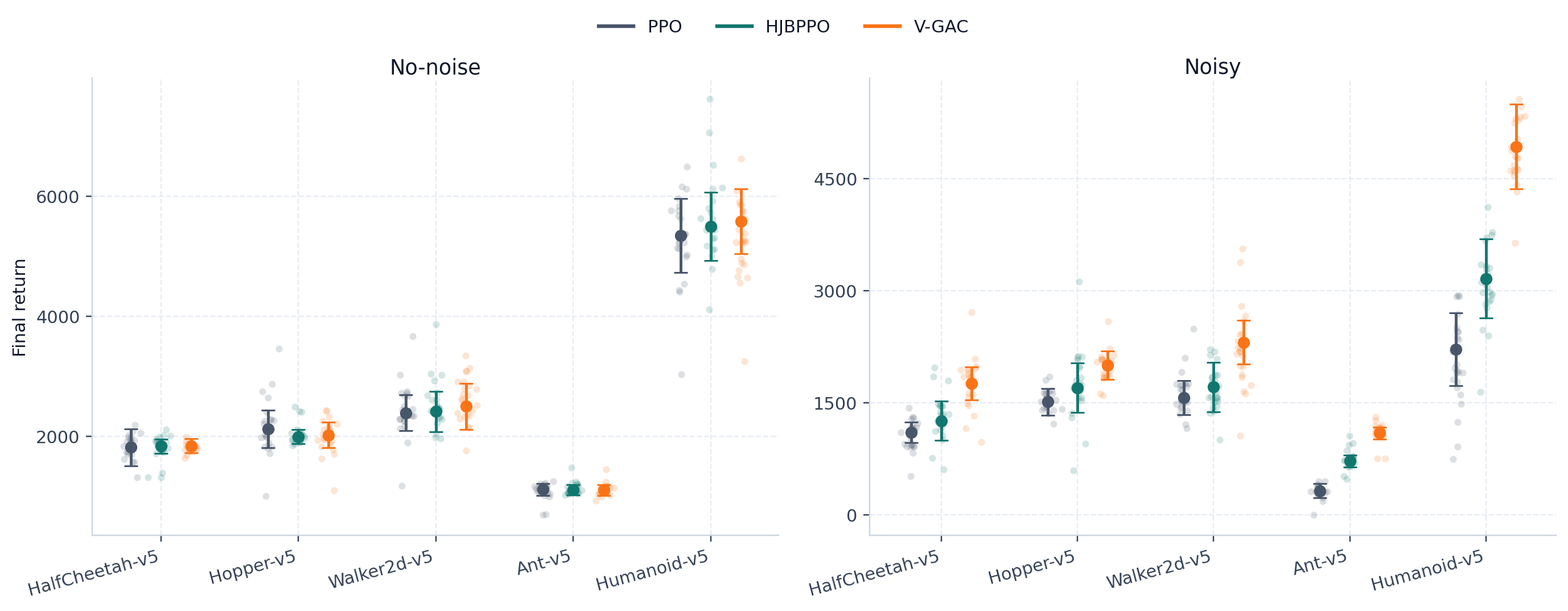}
\caption{Final return endpoints for the locomotion and Humanoid-v5 benchmarks. Left and right panels correspond to nominal and Brownian-perturbed evaluation, respectively. Points show individual seeds; markers and bars show the mean and spread.}
\label{fig:final_endpoints_returns}
\end{figure}

\cref{fig:final_endpoints_returns} summarizes the endpoint return statistics across the locomotion and Humanoid benchmarks. Under Brownian-perturbed evaluation, the V-GAC endpoint clouds are typically shifted upward and exhibit a reduced lower tail relative to PPO and HJBPPO.

\begin{figure}[htbp]
\centering
\includegraphics[width=\textwidth]{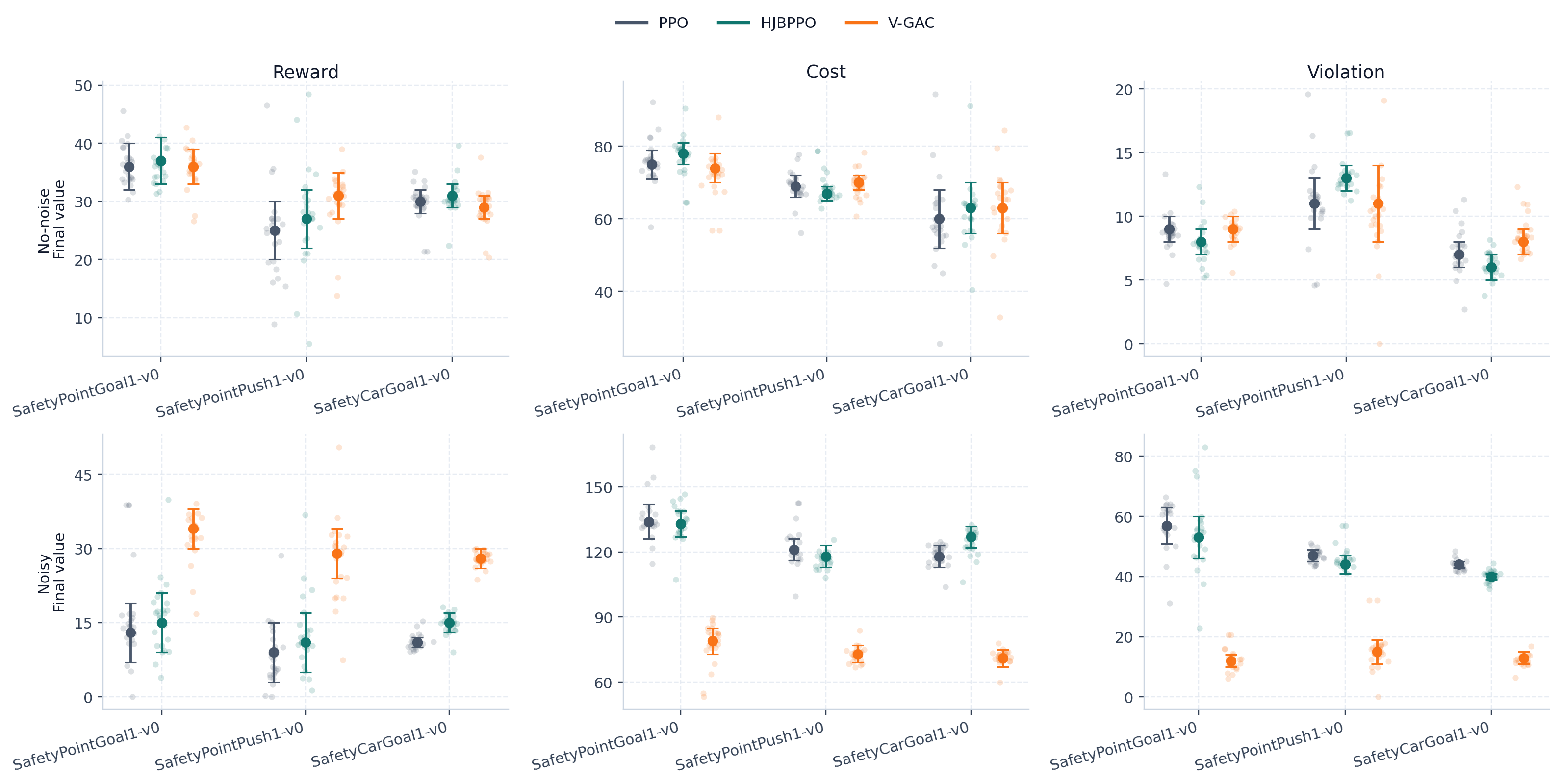}
\caption{Final safety endpoints under nominal and Brownian-perturbed evaluation. The top row corresponds to nominal evaluation and the bottom row to Brownian-perturbed evaluation. Columns report reward, cost, and violation rate. Points show individual seeds; markers and bars show the mean and spread.}
\label{fig:final_endpoints_safety}
\end{figure}

The safety-task endpoint statistics in \cref{fig:final_endpoints_safety} show that, under Brownian-perturbed evaluation, V-GAC more consistently attains higher reward together with lower cost and lower violation rate across environments.

\FloatBarrier
%%%%%%%%%%%%%%%%%%%%%%%%%%%%%%%%%%%%%%%%%%%%%%%%%%%%%%%%%%%%%%%%%%%%%%%%%%%%%%%%%%%%%%%%

\section{Concluding Remarks}For nonsmooth stationary HJB problems, the relevant learning target is the comparison-based viscosity selection mechanism. Under the stated assumptions, vanishing expected viscosity loss together with vanishing jet-wise policy gap yields subsequential critic limits satisfying the exact HJB inequalities on the sampled SPD-envelope jet family. The numerical results indicate that the primary effect is not improved nominal return, but improved local viscosity consistency and reduced Hamiltonian suboptimality at sampled envelope jets. This improvement becomes consequential under Brownian-perturbed evaluation where trajectories are displaced from the nominal rollout distribution and V-GAC degrades substantially less. Interpreted at the PDE level, the noisy tests probe nearby second-order operators $F_\sigma$, with $F_\sigma \to F_0$ locally uniformly as $\sigma \to 0$, and the resulting robustness is consistent with the stability theory of viscosity solutions under such perturbations.

%%%%%%%%%%%%%%%%%%%%%%%%%%%%%%%%%%%%%%%%%%%%%%%%%%%%%%%%%%%%%%%%%%%%%%%%%%%%%%%%%%%%%%%%

\appendix

\section*{Appendix overview}This appendix is organized as follows. \Cref{app:envelopes} develops the connection between generalized Moreau-envelope contacts and Crandall--Lions semijets, \cref{app:critic-correct} provides the proof of the sampled SPD-envelope viscosity consistency result, \cref{app:cl-completion-estimate} quantifies the completion error between the viscosity residual restricted to sampled SPD-envelope jets and the viscosity residual over the full Crandall--Lions semijet family in the sense of \cref{def:CL-jetdef}, and \cref{sec:sm-hparams} records the implementation parameters used in the numerical experiments. A Python implementation example is available at \url{https://github.com/golpashin/V-GAC}, providing additional implementation
details for \cref{alg:vgac}.

Throughout the appendix, $F$ denotes the operator \eqref{eq:F-def} and the standing assumptions remain in force.

\section{Generalized Moreau envelopes as test functions}\label[appendix]{app:envelopes}We recall the jet characterization of viscosity solutions; see \cite[sect.~2]{crandall1992}. The equivalence is made precise in \cref{lem:jet-char}.

\begin{definition}\label{def:CL-jetdef}
Let $V:\overline{\Omega}\to\mathbb{R}$ and $\zeta\in\Omega$.
\begin{enumerate}
\item The second-order \emph{subjet} of $V$ at $\zeta$, denoted $J^{2,-}V(\zeta)$, is the set of pairs $(p,A)\in\R^{n}\times\mathbb{S}(n)$ such that
\begin{equation*}
V(x) \ge V(\zeta)+ p\cdot(x-\zeta)+\tfrac12(x-\zeta)^\top A (x-\zeta) + o(\|x-\zeta\|^2)\quad\text{as }x\to \zeta.
\end{equation*}
\item The second-order \emph{superjet} $J^{2,+}V(\zeta)$ is defined by reversing the inequality.
\end{enumerate}
\end{definition}

\begin{lemma}\label{lem:jet-char}
For \(V\in C(\overline\Omega)\), the following are equivalent:
\begin{romannum}
\item \(V\) is a viscosity supersolution of \(F=0\) in \(\Omega\);
\item \(F(\zeta,V(\zeta),p,A)\ge0\) for every
\(\zeta\in\Omega\) and every \((p,A)\in J^{2,-}V(\zeta)\).
\end{romannum}
Similarly, \(V\) is a viscosity subsolution of \(F=0\) in
\(\Omega\) iff \(F(\zeta,V(\zeta),p,A)\le0\) for every
\(\zeta\in\Omega\) and every \((p,A)\in J^{2,+}V(\zeta)\);
see \cite[Def.~2.2 and Remarks~2.3]{crandall1992}.
\end{lemma}

\begin{theorem} \label{thm:inf-to-subjet}
Let $V:\overline{\Omega}\to\R$ be l.s.c., $M\in\mathbb{S}_{++}(n)$, and $x\in\overline{\Omega}$.
If
\[
\zeta_x\in\argmin_{\zeta\in\overline{\Omega}}
\Big\{V(\zeta)+\tfrac12 (x-\zeta)^\top M(x-\zeta)\Big\},
\]
and we set $p:=M(x-\zeta_x)$ and $A:=-M$, then
\[
V(\zeta)\ \ge\ V(\zeta_x) + p\cdot(\zeta-\zeta_x) + \tfrac12 (\zeta-\zeta_x)^\top A (\zeta-\zeta_x)
\quad \forall\,\zeta\in\overline{\Omega}.
\]
In particular, if $\zeta_x\in\Omega$, then $(p,A)\in J^{2,-}V(\zeta_x)$.
\end{theorem}

\begin{proof}
By minimality of $\zeta_x$,
\[
V(\zeta_x)+\tfrac12(x-\zeta_x)^\top M(x-\zeta_x)\ \le\
V(\zeta)+\tfrac12(x-\zeta)^\top M(x-\zeta)\quad\forall \zeta\in\overline{\Omega}.
\]
Rearrange and expand $(x-\zeta)^\top M(x-\zeta)$ with $h:=\zeta-\zeta_x$ to obtain
\[
V(\zeta)\ \ge\ V(\zeta_x) + M(x-\zeta_x)\cdot h - \tfrac12 h^\top M h
= V(\zeta_x) + p\cdot(\zeta-\zeta_x) + \tfrac12(\zeta-\zeta_x)^\top (-M)(\zeta-\zeta_x),
\]
which is the stated inequality with $A=-M$.
\end{proof}

\begin{theorem} \label{thm:sup-to-superjet}
Let \(V:\overline{\Omega}\to\R\) be u.s.c., \(M\in\mathbb{S}_{++}(n)\), and \(x\in\overline{\Omega}\).
If 
\[
\zeta_x^\star\in\argmax_{\zeta\in\overline{\Omega}}\{V(\zeta)-\tfrac12 (x-\zeta)^\top M(x-\zeta)\},
\]
and we set $p:=-\,M(x-\zeta_x^\star)$ and $A:=+M$, then
\begin{equation*}\label{eq:superjet-ineq}
V(\zeta)\ \le\ V(\zeta_x^\star) + p\cdot(\zeta-\zeta_x^\star) + \tfrac12 (\zeta-\zeta_x^\star)^\top A (\zeta-\zeta_x^\star)\quad \forall\,\zeta\in\overline{\Omega}.
\end{equation*}
Similarly, if $\zeta_x^\star\in\Omega$, then $(p,A)\in J^{2,+}V(\zeta_x^\star)$.
\end{theorem}
\begin{proof}
The proof is identical to that of \cref{thm:inf-to-subjet}, with the inequalities reversed.
\end{proof}

We now connect envelope-generated jets to the viscosity definition.

\begin{corollary}\label{cor:visc_sup_sub} Let \(V\in C(\overline\Omega)\). The following hold for \(F\): 
\begin{romannum}
\item If $V$ is a viscosity supersolution, then for all $(x,M) \in \overline\Omega \times \mathbb{S}_{++}(n)$, letting $\zeta_x\in\argmin_{\zeta\in\overline{\Omega}}\{V(\zeta)+\tfrac12(x-\zeta)^\top M(x-\zeta)\}$ and $(p,A)=(M(x-\zeta_x),-M)$, we have
\[
F(\zeta_x,V(\zeta_x),p,A)\ge 0\quad\text{whenever }\zeta_x\in\Omega.
\]
\item If $V$ is a viscosity subsolution, then for all $(x,M) \in \overline\Omega \times \mathbb{S}_{++}(n)$, letting $\zeta_x^\star\in\argmax_{\zeta\in\overline{\Omega}}\{V(\zeta)-\tfrac12(x-\zeta)^\top M(x-\zeta)\}$ and $(p,A)=(-M(x-\zeta_x^\star),+M)$, gives
\[
F(\zeta_x^\star,V(\zeta_x^\star),p,A)\le 0\quad\text{whenever }\zeta_x^\star\in\Omega.
\]
\end{romannum}
\end{corollary}

\begin{proof}
By \cref{thm:inf-to-subjet} and \cref{lem:jet-char}, \((p,A)\in J^{2,-}V(\zeta_x)\), hence the supersolution inequality at \((\zeta_x,p,A)\). Similarly, \cref{thm:sup-to-superjet} yields \((p,A)\in J^{2,+}V(\zeta_x^\star)\) and the subsolution inequality at \((\zeta_x^\star,p,A)\).
\end{proof}

\section{\texorpdfstring{Proof of \cref{thm:critic-correct}}{Proof of the sampled viscosity consistency theorem}}\label[appendix]{app:critic-correct}We begin with the following auxiliary lemma, after which we prove the viscosity consistency result.

\begin{lemma}
\label{lem:delta-contact}
Let $\Xi$ be compact, let $\Phi_k,\Phi\in C(\Xi)$ with $\|\Phi_k-\Phi\|_{L^\infty(\Xi)}\to 0$, and let $z_k\in \Xi$ satisfy $\Phi_k(z_k)\le \inf_{z\in \Xi}\Phi_k(z)+\delta_k$, for $\delta_k\to0$. Then every cluster point of $(z_k)$ belongs to $\argmin_{\Xi} \Phi$.
\end{lemma}

\begin{proof}
Let $z_{k_j}\to z$. For any $y\in \Xi$,
$\Phi_{k_j}(z_{k_j})\le \Phi_{k_j}(y)+\delta_{k_j}$.
Passing to the limit as $j\to\infty$ and using uniform convergence of
$\Phi_{k_j}$ to $\Phi$ (and continuity of $\Phi$) gives $\Phi(z)\le \Phi(y)$. Since $y\in\Xi$ was arbitrary, $z\in\argmin_\Xi\Phi$.
\end{proof}
An analogous statement holds for $\delta_k$-optimal maximizers.

\begin{proof}[Proof of \cref{thm:critic-correct}]
By \cref{assumps:limit} (ii), the sequence $\{V_k\}_{k\in\mathbb N}$ is equibounded and equicontinuous on $\overline{\Omega}$. By Arzel\`a--Ascoli,
every sequence admits a uniformly convergent subsequence. Fix such a subsequence, still denoted by $k$, and let $V_k\to V$ uniformly on $\overline{\Omega}$, for some $V\in C(\overline{\Omega})$. For the boundary condition, exact traces pass to the uniform limit in the first case of \cref{assumps:limit} \textnormal{(iv)}. In the second, uniform convergence and the vanishing boundary loss give \(V=g\) \(\nu_\Gamma\)-a.e.; continuity and \(\operatorname{supp}\nu_\Gamma=\partial\Omega\) then give \(V=g\) on \(\partial\Omega\).

Let \(\mathcal M=\{M_{q}\}_{q=1}^K\) be a generic draw as in
\cref{assump:sampling} and write 
\[
\zeta_{k}^{-}(x,M):=P_k(x,M,-1),\qquad
\zeta_{k}^{+}(x,M):=P_k(x,M,+1),
\]
for \((x,M)\in\overline\Omega\times\mathbb S_{++}(n)\), with associated jet data
\begin{align*}
p_k^{-}(x,M) &:= M\bigl(x-\zeta_k^{-}(x,M)\bigr), & A^{-} &:= -M,\\
p_k^{+}(x,M) &:= -\,M\bigl(x-\zeta_k^{+}(x,M)\bigr), & A^{+} &:= +M.
\end{align*}
Denote the violations in \eqref{eq:pointwise-violations} corresponding to \((V_k,P_k,\bar\pi_k)\) by \(g_{\text{super},k}\) and \(g_{\text{sub},k}\). The maximum over \(\mathcal M\) dominates the corresponding hinge at any \(M_q\in\mathcal M\), while \cref{assump:sampling} gives \(M_q\sim\nu_{\mathcal M}\) independently of \(x\), and \(\nu_{X,k}\ge\rho_{\mathrm{cover}}\nu_{\mathrm{cov}}\) by the mixing law. Hence
\[
\begin{aligned}
&\E_{\mathcal M}\big[L_{\text{Visc}}(\theta_k,\psi_k)\big]\\
&\quad\ge \rho_{\mathrm{cover}}
\E_{(x,M)\sim\nu_{\mathrm{cov}}\otimes\nu_{\mathcal M}}
\Big[
\max\{g_{\text{super},k}(x,M),0\}^2
+\max\{g_{\text{sub},k}(x,M),0\}^2
\Big].
\end{aligned}
\]
Since the left-hand side tends to zero and
\(\rho_{\mathrm{cover}}>0\), both positive parts converge to zero in
\(L^2(\nu_{\mathrm{cov}}\otimes\nu_{\mathcal M})\). Passing to a
further subsequence, still denoted by $k$, $\max\{g_{\mathrm{super},k}(x,M),0\}\to0$ and $\max\{g_{\mathrm{sub},k}(x,M),0\}\to0$ for \((\nu_{\mathrm{cov}}\otimes\nu_{\mathcal M})\)-a.e. \((x,M)\).

Next, applying \cref{lem:jet_gap_decomp} with
\(\nu_X=\nu_{X,k}\), and using
\cref{assumps:limit}\textnormal{(iii)} and the mixing law, gives
\[
\begin{aligned}
0
&\le \rho_{\mathrm{cover}}
\E_{(x,M)\sim\nu_{\mathrm{cov}}\otimes\nu_{\mathcal M}}
\Bigg[
\sum_{b\in\{-1,+1\}}
\text{gap}_{\bar\pi_k}
\bigl(\zeta_k^b(x,M),p_k^b(x,M),A^b\bigr)
\Bigg]\\
&\le L_{\text{jet}}(\phi_k;\psi_k)
-L_{\text{jet}}^\star(\psi_k)
\le\varepsilon_k\to0,
\end{aligned}
\]
where each gap term is set to zero whenever its contact lies on \(\partial\Omega\). Hence each term converges to zero in \(L^1(\nu_{\mathrm{cov}}\otimes\nu_{\mathcal M})\). Passing to a further subsequence, still denoted by \(k\), $\text{gap}_{\bar\pi_k} \bigl(\zeta_k^b(x,M),p_k^b(x,M),A^b\bigr)\to0$, $b\in\{-1,+1\}$, for \((\nu_{\mathrm{cov}}\otimes\nu_{\mathcal M})\)-a.e. \((x,M)\).

Let $\mathcal A\subset\overline\Omega\times\mathbb S_{++}(n)$
be a $(\nu_{\mathrm{cov}}\otimes\nu_{\mathcal M})$-full-measure set
on which all four convergences above hold. Fix
\((x,M)\in\mathcal A\).
Since $\overline{\Omega}$ is compact, the sequences
$\zeta_k^{\pm}(x,M)$ admit cluster points. Choose a subsequence such that $\zeta_k^{\pm}(x,M)\to \zeta^{\pm}$. Then define the envelope objective functionals

\[
\Phi_k^{\inf}(\zeta):=V_k(\zeta)+\tfrac12(x-\zeta)^\top M(x-\zeta),
\qquad
\Phi_k^{\sup}(\zeta):=V_k(\zeta)-\tfrac12(x-\zeta)^\top M(x-\zeta),
\]
with their respective limits $\Phi^{\inf}(\zeta)$ and $\Phi^{\sup}(\zeta)$. Because $V_k\to V$ uniformly on $\overline{\Omega}$,
\[
\|\Phi_k^{\inf}-\Phi^{\inf}\|_{L^\infty(\overline{\Omega})}\to 0,
\qquad
\|\Phi_k^{\sup}-\Phi^{\sup}\|_{L^\infty(\overline{\Omega})}\to 0.
\]
By \cref{assumps:limit} (i), $\zeta_k^{-}(x,M)$ is a
$\delta_k$-optimal minimizer of $\Phi_k^{\inf}$ and $\zeta_k^{+}(x,M)$ is a
$\delta_k$-optimal maximizer of $\Phi_k^{\sup}$, with $\delta_k\to 0$.
\cref{lem:delta-contact} therefore implies that
\[
\zeta^{-}\in \argmin_{\zeta\in\overline{\Omega}} \Phi^{\inf}(\zeta),
\qquad
\zeta^{+}\in \argmax_{\zeta\in\overline{\Omega}} \Phi^{\sup}(\zeta).
\]
Thus $\zeta_x:=\zeta^-$ and $\zeta_x^\star:=\zeta^+$ define exact inf-/sup-envelope contacts of $V$ at $(x,M)$.

Set $p^-:=M(x-\zeta^-)$ and $p^+:=-\,M(x-\zeta^+)$.
If $\zeta^{-}\in\Omega$, then, since $\Omega$ is open and
$\zeta_k^{-}(x,M)\to\zeta^{-}$, we have $\zeta_k^{-}(x,M)\in\Omega$ for all
large $k$. For such $k$,
\[
F\bigl(\zeta_k^{-},V_k(\zeta_k^{-}),p_k^{-},A^{-}\bigr)
=-\,g_{\text{super},k}(x,M)+\text{gap}_{\bar\pi_k}\!\bigl(\zeta_k^{-}(x,M),p_k^{-}(x,M),A^{-}\bigr).
\]
Since $\text{gap}_{\bar\pi_k}\ge 0$, we obtain $F\bigl(\zeta_k^{-},V_k(\zeta_k^{-}),p_k^{-},A^{-}\bigr) \ge -\max\{g_{\text{super},k}(x,M),0\}$.
Letting $k\to\infty$ along the chosen subsequence and using continuity of $F$,
uniform convergence of $V_k$ and $\zeta_k^{-}(x,M)\to\zeta^{-}$ yields $F\big(\zeta_x,V(\zeta_x),M(x-\zeta_x),-M\big)\ge 0$.

Likewise, if $\zeta^{+}\in\Omega$, then $\zeta_k^{+}(x,M)\in\Omega$ for all
large $k$, and
\begin{equation*}
F\bigl(\zeta_k^{+},V_k(\zeta_k^{+}),p_k^{+},A^{+}\bigr)
=g_{\text{sub},k}(x,M)+\text{gap}_{\bar\pi_k}\!\bigl(\zeta_k^{+}(x,M),p_k^{+}(x,M),A^{+}\bigr).
\end{equation*}
So, $F\bigl(\zeta_k^{+},V_k(\zeta_k^{+}),p_k^{+},A^{+}\bigr) \le \max\{g_{\text{sub},k}(x,M),0\} + \text{gap}_{\bar\pi_k}\!\bigl(\zeta_k^{+}(x,M),p_k^{+}(x,M),A^{+}\bigr)$.
Passing to the limit then gives $F\big(\zeta_x^\star,V(\zeta_x^\star),-M(x-\zeta_x^\star),M\big)\le 0$.

Since \((x,M)\in\mathcal A\) was arbitrary, the asserted contact inequalities hold for $(\nu_{\mathrm{cov}}\otimes\nu_{\mathcal M})$-a.e. $(x,M)$.
\end{proof}

\section{Completion error for sampled SPD-envelope tests}\label[appendix]{app:cl-completion-estimate}The sampled SPD-envelope loss enforces viscosity inequalities on the
quadratic tests generated by the envelope construction. Here we quantify
the gap between those sampled tests and the full Crandall--Lions semijet
conditions on a bounded region of jet space.

Throughout this section, $W\in C(\overline{\Omega})$ denotes a continuous
candidate function. Let
$\mathcal Q\subset\overline{\Omega}\times\mathbb S_{++}(n)$ denote the
anchor--curvature set used for testing, i.e. $\mathcal Q$ may be
the support of $\nu_X\otimes\nu_{\mathcal M}$ or a finite validation set.
For $(x,M)\in\mathcal Q$, the contacts $\zeta_x,\zeta_x^\star$ and associated
data $p^\pm(x,M),A^\pm$ are understood in the sense of the envelope
construction in \cref{subsec:envelopes}. If an envelope contact is not unique, fix an arbitrary measurable selection.

Since full Crandall--Lions semijets need not be bounded, we restrict the analysis
to a compact jet tube
\[
\Upsilon \subset
\overline{\Omega}\times\mathbb R\times\mathbb R^n\times\mathbb S(n).
\]
All jet sets, residuals, active sets, and coverage radii below are understood
to be restricted to jet data lying in $\Upsilon$. Thus, for example, a datum
$Y=(\zeta,W(\zeta),p,A)$ belongs to a jet set only when
$Y\in\Upsilon$. This compact restriction is suppressed from the notation. The restricted residuals are therefore finite and the estimate below is merely a compact-tube comparison. 

Since $F$ is continuous, let
$\omega_F$ be a nondecreasing right-continuous modulus of continuity for $F$
on $\Upsilon$, so that
\[
|F(Y)-F(\hat{Y})|\le \omega_F(d(Y,\hat{Y})),
\]
for all $Y,\hat{Y}\in\Upsilon$. Here, if
$Y=(\zeta,r,p,A)$ and $\hat{Y}=(\hat{\zeta},\hat{r},\hat{p},\hat{A})$, then
\[
d(Y,\hat{Y}):=
\|\zeta-\hat{\zeta}\|+|r-\hat{r}|+\|p-\hat{p}\|+\|A-\hat{A}\|_F .
\]

Define the full Crandall--Lions jet families
\[
\mathfrak J^{-}(W)
:=\left\{(\zeta,W(\zeta),p,A):\zeta\in\Omega,\;
(p,A)\in J^{2,-}W(\zeta)\right\},
\]
and
\[
\mathfrak J^{+}(W)
:=\left\{(\zeta,W(\zeta),p,A):\zeta\in\Omega,\;
(p,A)\in J^{2,+}W(\zeta)\right\}.
\]
Note that these are the full supersolution and subsolution test families, respectively,
restricted to $\Upsilon$ by the convention above.

Similarly, define the sampled SPD-envelope jet families
\[
\mathfrak S_{\mathcal Q}^{-}(W)
:=\left\{(\zeta_x,W(\zeta_x),p^{-}(x,M),A^{-}):
(x,M)\in\mathcal Q,\;\zeta_x\in\Omega\right\},
\]
and
\[
\mathfrak S_{\mathcal Q}^{+}(W)
:=\left\{(\zeta_x^\star,W(\zeta_x^\star),p^{+}(x,M),A^{+}):
(x,M)\in\mathcal Q,\;\zeta_x^\star\in\Omega\right\}.
\]
By the envelope-to-jet construction,
$\mathfrak S_{\mathcal Q}^{-}(W)\subset\mathfrak J^{-}(W)$ and
$\mathfrak S_{\mathcal Q}^{+}(W)\subset\mathfrak J^{+}(W)$.

The restricted full Crandall--Lions residuals are
\[
\mathfrak R_{\rm CL}^{-}(W)
:=\sup_{Y\in\mathfrak J^{-}(W)}\max\{-\,F(Y),0\},
\qquad
\mathfrak R_{\rm CL}^{+}(W)
:=\sup_{Y\in\mathfrak J^{+}(W)}\max\{F(Y),0\}.
\]
The restricted sampled SPD-envelope residuals are
\[
\mathfrak R_{\mathcal Q}^{-}(W)
:=\sup_{\hat{Y}\in\mathfrak S_{\mathcal Q}^{-}(W)}
\max\{-\,F(\hat{Y}),0\},
\qquad
\mathfrak R_{\mathcal Q}^{+}(W)
:=\sup_{\hat{Y}\in\mathfrak S_{\mathcal Q}^{+}(W)}
\max\{F(\hat{Y}),0\},
\]
with convention $\sup\emptyset=0$. Since the sampled SPD-envelope jets
are contained in the full Crandall--Lions semijets, we have
\[
\mathfrak R_{\mathcal Q}^{-}(W)\le \mathfrak R_{\rm CL}^{-}(W),
\qquad
\mathfrak R_{\mathcal Q}^{+}(W)\le \mathfrak R_{\rm CL}^{+}(W).
\]
Define the restricted completion errors
\[
\mathfrak C_{\mathcal Q}^{-}(W)
:=\mathfrak R_{\rm CL}^{-}(W)-\mathfrak R_{\mathcal Q}^{-}(W),
\qquad
\mathfrak C_{\mathcal Q}^{+}(W)
:=\mathfrak R_{\rm CL}^{+}(W)-\mathfrak R_{\mathcal Q}^{+}(W).
\]
These quantities measure the full viscosity-inequality violation, within
$\Upsilon$, not detected by the sampled SPD-envelope tests.

For $\epsilon>0$, define the nearly active full-test sets
\[
\mathfrak A^{-}(W)
:=\left\{Y\in\mathfrak J^{-}(W):\max\{-\,F(Y),0\}
\ge
\mathfrak R_{\rm CL}^{-}(W)-\epsilon\right\},
\]
and
\[
\mathfrak A^{+}(W)
:=\left\{Y\in\mathfrak J^{+}(W):\max\{F(Y),0\}
\ge
\mathfrak R_{\rm CL}^{+}(W)-\epsilon\right\}.
\]
Finally, we also define the directed coverage radii
\[
\mathfrak r_{\mathcal Q}^{-}(W)
:=\sup_{Y\in\mathfrak A^{-}(W)}
\inf_{\hat{Y}\in\mathfrak S_{\mathcal Q}^{-}(W)}d(Y,\hat{Y}),
\qquad
\mathfrak r_{\mathcal Q}^{+}(W)
:=\sup_{Y\in\mathfrak A^{+}(W)}
\inf_{\hat{Y}\in\mathfrak S_{\mathcal Q}^{+}(W)}d(Y,\hat{Y}),
\]
with conventions $\sup\emptyset=0$, $\inf\emptyset=+\infty$ and
$\omega_F(+\infty)=+\infty$.

\begin{proposition}\label{prop:semijet-coverage-completion}Let $W\in C(\overline{\Omega})$, and let $\Upsilon$ be a compact jet tube as
above. Then, for every $\epsilon>0$,
\[
\mathfrak C_{\mathcal Q}^{-}(W)
\le
\omega_F\!\left(\mathfrak r_{\mathcal Q}^{-}(W)\right) + \epsilon,
\qquad
\mathfrak C_{\mathcal Q}^{+}(W)
\le
\omega_F\!\left(\mathfrak r_{\mathcal Q}^{+}(W)\right) + \epsilon.
\]
\end{proposition}

\begin{proof}
 If
$\mathfrak A^{-}(W)=\emptyset$, then
$\mathfrak R_{\rm CL}^{-}(W)=\mathfrak R_{\mathcal Q}^{-}(W)=0$, and the
claim is immediate. Otherwise, take any nearly active full subjet $Y\in\mathfrak A^{-}(W)$.
By definition of $\mathfrak A^{-}(W)$,
\[
\max\{-\,F(Y),0\}
\ge
\mathfrak R_{\rm CL}^{-}(W)-\epsilon.
\]
Equivalently,
\[
\mathfrak R_{\rm CL}^{-}(W)
\le
\max\{-\,F(Y),0\}+\epsilon.
\]

Now, for any $\varsigma>0$, choose an SPD-envelope jet $\hat{Y}\in\mathfrak S_{\mathcal Q}^{-}(W)$,
such that
\[
d(Y,\hat{Y})
\le
\mathfrak r_{\mathcal Q}^{-}(W)+\varsigma.
\]
If no such $\hat{Y}$ exists, then
$\mathfrak r_{\mathcal Q}^{-}(W)=+\infty$ and the estimate is trivial.

For the map $s\mapsto \max\{-s,0\}$, we have
\[
\max\{-\,F(Y),0\}
\le
\max\{-\,F(\hat{Y}),0\} + |F(Y)-F(\hat{Y})|.
\]
The first term is bounded by the SPD-envelope residual, $\mathfrak R_{\mathcal Q}^{-}(W)$, and the second term is bounded by the modulus of continuity of $F$.  Therefore,
\[
\max\{-\,F(Y),0\}
\le
\mathfrak R_{\mathcal Q}^{-}(W) +
\omega_F\!\left(\mathfrak r_{\mathcal Q}^{-}(W)+\varsigma \right).
\]
Letting the $\varsigma$ term vanish gives
\[
\max\{-\,F(Y),0\}
\le
\mathfrak R_{\mathcal Q}^{-}(W) +
\omega_F\!\left(\mathfrak r_{\mathcal Q}^{-}(W)\right).
\]
Substituting this into
\[
\mathfrak R_{\rm CL}^{-}(W)
\le
\max\{-\,F(Y),0\}+\epsilon
\]
yields
\[
\mathfrak R_{\rm CL}^{-}(W)
\le
\mathfrak R_{\mathcal Q}^{-}(W) +
\omega_F\!\left(\mathfrak r_{\mathcal Q}^{-}(W)\right) + \epsilon.
\]
Subtracting $\mathfrak R_{\mathcal Q}^{-}(W)$ from both sides gives
\[
\mathfrak C_{\mathcal Q}^{-}(W)
= \mathfrak R_{\rm CL}^{-}(W)-\mathfrak R_{\mathcal Q}^{-}(W)
\le
\omega_F\!\left(\mathfrak r_{\mathcal Q}^{-}(W)\right) + \epsilon.
\]
The subsolution estimate is the same,
replacing $\max\{-\,F,0\}$ by $\max\{F,0\}$.
\end{proof}

\begin{remark} The modulus $\omega_F$ determines how strongly the jet mismatch affects the completion error. If $F$ is Lipschitz on $\Upsilon$, then
$\omega_F(t)\le L_Ft$, and \cref{prop:semijet-coverage-completion} gives
\[
\mathfrak C_{\mathcal Q}^{\pm}(W)
\le
L_F\,\mathfrak r_{\mathcal Q}^{\pm}(W)+\epsilon.
\]
For the HJB operator \eqref{eq:F-def}, this Lipschitz constant can be read from the coefficients. On a compact jet tube where
$\|p\|,\|\hat p\|\le P^{\dagger}$ and $\|A\|_F,\|\hat A\|_F\le A^{\dagger}$, suppose
$f$, $a$, and $\ell$ are uniformly Lipschitz in the state variable, with constants $L_f$, $L_a$, and $L_\ell$, and set
$\bar f:=\sup_{x,c}\|f(x,c)\|$ and
$\bar a:=\sup_{x,c}\|a(x,c)\|_F$. Then
\[
\begin{aligned}
&|F(\zeta,r,p,A)-F(\hat{\zeta},\hat r,\hat p,\hat A)| \\
&\qquad\le
\beta |r-\hat r| +
\bar f\|p-\hat p\| +
\frac12 \bar a\|A-\hat A\|_F +
\left( L_\ell+P^{\dagger}L_f+\frac12 A^{\dagger} L_a
\right)\|\zeta-\hat{\zeta}\|.
\end{aligned}
\]
Thus the Hessian mismatch is weighted by $\frac12\sup_{x,c}\|a(x,c)\|_F$. In deterministic first-order problems $(a\equiv0)$, the operator is insensitive to the Hessian component of the coverage distance. For small diffusion, curvature mismatch contributes weakly to the completion error; for larger diffusion, coverage of the second-order part becomes correspondingly more important.
\end{remark}

\begin{remark}The estimate in \cref{prop:semijet-coverage-completion} shows that the completion error is controlled by the directed distance from active full Crandall--Lions tests to sampled SPD-envelope tests. To see the terms entering this distance, let $Y=(\zeta,W(\zeta),p,A)\in\mathfrak J^{-}(W)$ and let $\hat{Y}=(\hat{\zeta},W(\hat{\zeta}),M(x-\hat{\zeta}),-M)
\in\mathfrak S_{\mathcal Q}^{-}(W)$. Then
\[
d(Y,\hat{Y})
= \|\zeta-\hat{\zeta}\| + |W(\zeta)-W(\hat{\zeta})| + 
\|p-M(x-\hat{\zeta})\| + \|A+M\|_F .
\]
Notice that the supersolution-side coverage distance contains contact-location
mismatch, value mismatch, gradient/anchor mismatch, and curvature mismatch;
the curvature component is $\|A+M\|_F$. Similarly, if
$Y=(\zeta,W(\zeta),p,A)\in\mathfrak J^{+}(W)$ and
$\hat{Y} = (\hat{\zeta},W(\hat{\zeta}),-M(x-\hat{\zeta}),M)
\in\mathfrak S_{\mathcal Q}^{+}(W)$, then
\[
d(Y,\hat{Y}) =
\|\zeta-\hat{\zeta}\| + |W(\zeta)-W(\hat{\zeta})| +
\|p+M(x-\hat{\zeta})\| + \|A-M\|_F .
\]
Thus the subsolution-side coverage distance again contains contact-location
mismatch, value mismatch, gradient/anchor mismatch, and curvature mismatch;
the curvature component is $\|A-M\|_F$. Consequently, the coverage radius
measures a local second-order distance between the active Crandall--Lions tests
and the sampled envelope quadratics.
\end{remark}

\begin{remark}The previous remark shows that the curvature component of the coverage radius is determined by the distance between the second-order part $A$ of an active full semijet and the signed SPD curvature generated by the sampled envelope test. On the supersolution side this signed curvature is $-M$, while on the subsolution side it is $+M$. Hence, when the sampled matrices satisfy a spectral restriction such as $\alpha_{\min}I\preceq M\preceq\alpha_{\max}I$, the curvature mismatch compares the eigenvalues of active full-test Hessians against the signed bands $[-\alpha_{\max},-\alpha_{\min}]$ and
$[\alpha_{\min},\alpha_{\max}]$. Sampling more matrices refines coverage inside the chosen signed band, while the remaining distance is the spectral mismatch between the active full-test curvature and that band.
\end{remark}

\section{Experimental parameters}\label[appendix]{sec:sm-hparams}We record the implementation parameters used in the experiments of \cref{sec:experiments}. Values are rounded for readability when additional digits do not affect the learned critic or policy. We list the parameters that materially affect the dynamics, sampling, losses, and optimization. Throughout, $\text{lr}$ denotes the learning rate. Workers refers to the number of parallel rollout processes and steps $S$ denotes the number of time steps per worker in each update; their product gives the total number of samples collected per update. There was no additional $\alpha$-scaling of the loss terms in any of the presented experiments.

\begin{remark}
In our implementation, the spectral band $[\alpha_{\min},\alpha_{\max}]$ can be interpreted as a curvature-resolution parameter for the sampled family of SPD-envelope jets. When this band is numerically wide, rescaling the viscosity and proximal losses by factors of order $\alpha_{\max}^{-2}$ and $\alpha_{\max}^{-1}$ improves numerical stability. For the proximal-optimality loss, a step size of $\eta \sim (1+\alpha_{\max})^{-1}$ is suggested. This scales the gradient step in the projected proximal update with the largest sampled curvature.
\end{remark}

The problem and training parameters for the experiment of \cref{exp:eikonal} are listed in \cref{tab:sm-eikonal-problem,tab:sm-eikonal-core,tab:sm-eikonal-visc-adv}. Both the transformed and untransformed formulations use ReLU critics with exact endpoint enforcement and
fresh uniformly sampled anchors, while all other experiments use
\(\tanh\) activations. The PINN settings also apply to both the original and transformed residuals. Each V-GAC training step consists of three proximal updates followed by one critic update. Dashes denote absent components, zero coefficients denote omitted loss terms, and $\rho_{\text{cover}}=1$ expresses the uniform anchor sampling in the notation of the mixing law. For the derivative plots, local-average slopes are computed from interpolated endpoint values over windows of width $0.02$, truncated at the domain boundary; windows crossing $x=1/2$ are omitted.

\begin{table}[htbp]
\footnotesize
\caption{Problem and sampling parameters for the one-dimensional Eikonal experiment.}
\label{tab:sm-eikonal-problem}
\centering
{\setlength{\tabcolsep}{4pt}
\renewcommand{\arraystretch}{1.04}
\begin{tabular}{@{}p{0.40\textwidth}@{\hspace{1em}}p{0.25\textwidth}@{}}
\hline
\textbf{Parameter} & \textbf{Value} \\
\hline
$\Omega$ & $(0,1)$ \\

Anchor law $\nu_X$ & $\operatorname{Unif}(0,1)$ \\
Boundary enforcement & exact affine correction \\
\hline
\end{tabular}}
\end{table}

\begin{table}[htbp]
\footnotesize
\caption{Network architecture and optimization parameters for the one-dimensional Eikonal experiment. Learning rates are initial values; a dash denotes an absent component.}
\label{tab:sm-eikonal-core}
\centering
{\setlength{\tabcolsep}{4pt}
\renewcommand{\arraystretch}{1.04}
\begin{tabular}{@{}p{0.420\textwidth}@{\hspace{1em}}p{0.21\textwidth}@{\hspace{1em}}p{0.21\textwidth}@{}}
\hline
\textbf{Parameter} & \textbf{V-GAC} & \textbf{PINN} \\
\hline
Critic hidden widths & $(16)$ & $(16)$ \\
Proximal hidden widths & $(16,16)$ & --- \\
Hidden-layer activation & ReLU & ReLU \\
Linear-layer normalization & none & none \\
Critic input & $2x-1$ & $2x-1$ \\
Critic weight initialization & default uniform & default uniform \\
Proximal output weight/bias initialization & $0$ & --- \\
$B$ (minibatch size) & $256$ & $256$ \\
$\lambda_{\text{TD}}$ & $0$ & --- \\
Optimizer & AdamW & AdamW \\
$\text{lr}_{\text{critic}}$ & $3\times10^{-4}$ & $3\times10^{-4}$ \\
$\text{lr}_{\text{prox}}$ & $2\times10^{-4}$ & --- \\
Critic learning-rate schedule & cosine & cosine \\
Proximal learning-rate schedule & cosine & --- \\
Minimum/initial learning-rate ratio & $0.03$ & $0.03$ \\
Weight decay & $10^{-6}$ & $10^{-6}$ \\
Gradient clipping & $10$ & $10$ \\
Outer training iterations & $10^{4}$ & $10^{4}$ \\
Random seeds & $0,\ldots,9$ & $0,\ldots,9$ \\
\hline
\end{tabular}}
\end{table}

\begin{table}[htbp]
\footnotesize
\caption{Residual, viscosity-envelope, and proximal-adversary parameters for the one-dimensional Eikonal experiment. The PINN boundary coefficient is configured but inactive under exact endpoint enforcement.}
\label{tab:sm-eikonal-visc-adv}
\centering
{\setlength{\tabcolsep}{4pt}
\renewcommand{\arraystretch}{1.04}
\begin{tabular}{@{}p{0.40\textwidth}@{\hspace{1em}}p{0.21\textwidth}@{\hspace{1em}}p{0.21\textwidth}@{}}
\hline
\textbf{Parameter} & \textbf{V-GAC} & \textbf{PINN} \\
\hline
$K$ (curvature matrices) & $64$ & --- \\
$\alpha_{\min}$ & $2$ & --- \\
$\alpha_{\max}$ & $80$ & --- \\
$\alpha$-range schedule & fixed; no widening & --- \\
$\rho_{\text{cover}}$ & $1$ & $1$ \\
$\mathcal{R}$ (strong-form residual) weight & --- & $1$ \\
$\lambda_{\text{visc}}$ & $0.05$ & --- \\
$\lambda_{\text{bdy}}$ (exact affine correction) & $1$ & $1$ \\
$\lambda_{\text{jet}}$ & $0$ & --- \\
$\lambda_{\text{adv}}$ & $0.013$ & --- \\
$\lambda_{\text{env}}$ & $0.1$ & --- \\
$\lambda_{\text{prox-opt}}$ & $1$ & --- \\
$K_{\text{adv}}$ & $3$ & --- \\
$\eta$ & $1/81$ & --- \\
Boundary-contact tolerance & $2\times10^{-6}$ & --- \\
\hline
\end{tabular}}
\end{table}

The problem and training parameters for the experiment of \cref{exp:euler} are listed in \cref{tab:sm-euler-problem,tab:sm-euler-core,tab:sm-euler-visc-adv}. The entropy coefficient and learning rates are held fixed throughout training for this problem.

\begin{table}[htbp]
\footnotesize
\caption{Problem, HJB, and rollout-discretization parameters for the rigid-body stabilization experiment.}
\label{tab:sm-euler-problem}
\centering
{\setlength{\tabcolsep}{4pt}
\renewcommand{\arraystretch}{1.04}
\begin{tabular}{@{}p{0.40\textwidth}@{\hspace{1em}}p{0.25\textwidth}@{}}
\hline
\textbf{Parameter} & \textbf{Value} \\
\hline
$\Delta t$ & $10^{-3}$ \\
$\beta$ & $0.8$ \\
$\Omega$ & $\{r_{\text{tgt}}<\|\omega\|<5\}$ \\
$r_{\text{tgt}}$ & $5\times 10^{-3}$ \\
$U$ & $[-15,15]^3$ \\
$\sigma$ & $0.05\,I_3$ \\
$(I_1,I_2,I_3)$ & $(1,2,3)$ \\
Weights in $\ell(\omega,u)$ & $Q=I_3,\;R=0.1\,I_3$ \\
Initial-state support & $r_{\text{tgt}}<\|\omega(0)\|<5$ \\
Terminal exit penalty & $50$ \\
Horizon-outside penalty & $0$ \\
Maximum episode length & $10^{6}$ \\
\hline
\end{tabular}}
\end{table}

\begin{table}[H]
\footnotesize
\caption{Network architecture, policy parameterization, PPO, and optimization parameters for the rigid-body stabilization experiment.}
\label{tab:sm-euler-core}
\centering
{\setlength{\tabcolsep}{4pt}
\renewcommand{\arraystretch}{1.04}
\begin{tabular}{@{}p{0.50\textwidth}@{\hspace{1em}}p{0.21\textwidth}@{}}
\hline
\textbf{Parameter} & \textbf{Value} \\
\hline
Actor hidden widths & $(64,64)$ \\
Critic hidden widths & $(64,64)$ \\
Proximal hidden widths & $(64,64)$ \\
Hidden-layer activation & $\tanh$ \\
Linear-layer normalization & weight normalization\\
Action limit & $15$ \\
Policy actor log-standard-deviation initialization & $0$ \\
Policy log-standard-deviation bounds & $(-5.0,\,-1.5)$ \\
Workers$\times$steps & $16\times 128$ \\
$N$ (PPO epochs) & $6$ \\
$B$ & $64$ \\
$\gamma=e^{-\beta\Delta t}$ & $0.99920032$ \\
$\epsilon$ (PPO clipping parameter) & $5.4\times 10^{-2}$ \\
$\lambda_{\text{GAE}}$ & $0.93$ \\
$\lambda_{\text{ent}}$ & $5.4\times 10^{-6}$ \\
$\lambda_{\text{TD}}$ & $0.5$ \\
$\text{lr}_{\text{actor}}$ & $1.9\times 10^{-7}$ \\
$\text{lr}_{\text{critic}}$ & $1.4\times 10^{-4}$ \\
$\text{lr}_{\text{prox}}$ & $4.8\times 10^{-4}$ \\
Critic learning-rate schedule & fixed \\
Weight decay & $10^{-4}$ \\
Gradient clipping & $10$ \\
Outer training iterations & $400$ \\
Random seed & $0$ \\
\hline
\end{tabular}}
\end{table}

\begin{table}[H]
\footnotesize
\caption{Viscosity-envelope, jet-alignment, and proximal-adversary parameters for the rigid-body stabilization experiment.}
\label{tab:sm-euler-visc-adv}
\centering
{\setlength{\tabcolsep}{4pt}
\renewcommand{\arraystretch}{1.04}
\begin{tabular}{@{}p{0.4\textwidth}@{\hspace{1em}}p{0.18\textwidth}@{}}
\hline
\textbf{Parameter} & \textbf{Value} \\
\hline
$K$ & $4$ \\
$\alpha_{\min}$ & $10^{-2}$ \\
$\alpha_{\max}$ & $10^{2}$ \\
$\alpha$-range schedule & fixed; no widening \\
$\rho_{\text{cover}}$ & $0.84$ \\
$\lambda_{\text{visc}}$ & $1.16\times 10^{-2}$ \\
$\lambda_{\text{bdy}}$ & $4.8$ \\
$\lambda_{\text{jet}}$ & $6.5\times 10^{-6}$ \\
$\lambda_{\text{adv}}$ & $5.1\times 10^{-6}$ \\
$\lambda_{\text{env}}$ & $3.8\times 10^{-5}$ \\
$\lambda_{\text{prox-opt}}$ & $2.6\times 10^{-2}$ \\
$K_{\text{adv}}$ & $2$ \\
$\eta$ & $9.9\times 10^{-3}$ \\
\hline
\end{tabular}}
\end{table}

For the Van der Pol experiment, the problem formulation and HJB equation are presented in subsection 4.3. Advantages are normalized prior to the policy update for this problem. The entropy coefficient and learning rates are similarly held fixed throughout training. In particular, the maximum episode length of 200, along with a macro time-step of 0.05  listed in \cref{tab:sm-vdp-gaussian-problem} imply that the time-limit truncation is 10. The macro time-step is the interval over which the control is held fixed before being updated with a new control. \cref{tab:sm-vdp-gaussian-problem,tab:sm-vdp-gaussian-core,tab:sm-vdp-gaussian-pde} contain the training, PDE and optimization parameters.

\begin{table}[H]
\footnotesize
\caption{Problem, HJB, and rollout-discretization parameters for the Van der Pol experiment.}
\label{tab:sm-vdp-gaussian-problem}
\centering
{\setlength{\tabcolsep}{4pt}
\renewcommand{\arraystretch}{1.04}
\begin{tabular}{@{}p{0.40\textwidth}@{\hspace{1em}}p{0.4\textwidth}@{}}
\hline
\textbf{Parameter} & \textbf{Value} \\
\hline
$\Delta t$ & $5\times 10^{-2}$ \\
RK4 substep $\Delta t_2$ & $10^{-3}$ \\
$\beta$ & $0.1$ \\
Domain \(\Omega\) & $(-2,2)^2\setminus\{\|y\|\le r_{\text{tgt}}\}$ \\
$r_{\text{tgt}}$ & $5\times 10^{-2}$ \\
$U$ & $[-1,1]$ \\
$\sigma$ & $0$ \\
Running cost $\ell(y,u)$ & $0.1$ \\
Initial-state support & $y(0)\in(-2,2)^2$ with $r_{\text{tgt}}<\|y(0)\|<2.8$ \\
Terminal exit penalty & $1$ \\
Maximum episode length & $200$ \\
\hline
\end{tabular}}
\end{table}

\begin{table}[htbp]
\footnotesize
\caption{Network architecture, policy parameterization, PPO, and optimization parameters for  the Van der Pol experiment.}
\label{tab:sm-vdp-gaussian-core}
\centering
{\setlength{\tabcolsep}{4pt}
\renewcommand{\arraystretch}{1.04}
\begin{tabular}{@{}p{0.50\textwidth}@{\hspace{1em}}p{0.21\textwidth}@{}}
\hline
\textbf{Parameter} & \textbf{Value} \\
\hline
Actor widths & $(64,64)$ \\
Critic widths & $(128,128,128)$ \\
Proximal widths & $(64,64)$ \\
Hidden-layer activation & $\tanh$ \\
Linear-layer normalization & weight normalization\\
Action limit & $1$ \\
Policy actor log-standard-deviation initialization & $-1.0$ \\
Policy log-standard-deviation bounds & $(-5.0,\,-1.0)$ \\
Workers$\times$steps & $16\times 128$ \\
$N$ & $4$ \\
$B$ & $512$ \\
$\gamma=e^{-\beta\Delta t}$ & $0.99501248$ \\
$\epsilon$  & $0.10$ \\
$\lambda_{\text{GAE}}$ & $0.95$ \\
$\lambda_{\text{ent}}$ & $5\times 10^{-4}$ \\
$\lambda_{\text{TD}}$ & $1.0$ \\
$\text{lr}_{\text{actor}}$ & $2.0\times 10^{-5}$ \\
$\text{lr}_{\text{critic}}$ & $1.5\times 10^{-4}$ \\
$\text{lr}_{\text{prox}}$ & $1.5\times 10^{-5}$ \\
Critic learning-rate schedule & fixed \\
Weight decay & $0$ \\
Gradient clipping & $10$ \\
Outer training iterations & $2.5\times 10^{5}$ \\
Random seed & $0$ \\
\hline
\end{tabular}}
\end{table}

\begin{table}[htbp]
\footnotesize
\caption{Viscosity-envelope, jet-alignment, and proximal-adversary parameters for the Van der Pol experiment.}
\label{tab:sm-vdp-gaussian-pde}
\centering
{\setlength{\tabcolsep}{4pt}
\renewcommand{\arraystretch}{1.04}
\begin{tabular}{@{}p{0.4\textwidth}@{\hspace{1em}}p{0.18\textwidth}@{}}
\hline
\textbf{Parameter} & \textbf{Value} \\
\hline
$K$  & $64$ \\
$\alpha_{\min}$ & $0.25$ \\
$\alpha_{\max}$ & $12$ \\
$\alpha$-range schedule & fixed; no widening \\
$\rho_{\text{cover}}$ & $0.50$ \\
$\lambda_{\text{visc}}$ & $0.08$ \\
$\lambda_{\text{bdy}}$ & $0.05$ \\
$\lambda_{\text{jet}}$ & $0.007$ \\
$\lambda_{\text{adv}}$ & $0.03$ \\
$\lambda_{\text{env}}$ & $0$ \\
$\lambda_{\text{prox-opt}}$ & $0.001$ \\
$K_{\text{adv}}$ & $2$ \\
$\eta$ & $7.69\times 10^{-2}$ \\
\hline
\end{tabular}}
\end{table}

The MuJoCo and Humanoid experiments are reported using the same categories of implementation details. We list the shared actor-critic architecture, PPO optimization settings, normalization choices, the perturbation level used in noisy evaluation, and the method-dependent critic weights for this benchmark family.

For the MuJoCo baselines, the PPO/HJBPPO settings were based on the appendix of \cite{mukherjee2023}: rollout horizon \(2048\), learning rate \(3\times 10^{-4}\), \(10\) PPO epochs $N$, minibatch size \(B=64\), and discount factor \(\gamma=0.99\). All benchmark domains used: tanh multilayer perceptrons, orthogonal initialization for the actor, weight normalization for critic and proximal networks, state-independent Gaussian log-standard-deviation parameters clipped to \([-5,2]\), input normalization, and no reward normalization. The learning rates were kept fixed throughout the training. For these benchmark tasks, \(\lambda_{\text{bdy}}=0\) because unlike the exit-time Euler and Van der Pol examples, these environments are not posed as bounded-domain Dirichlet problems with known terminal trace $g$. Their learning signal is purely interior. Accordingly, the PDE-side terms should be interpreted as an interior sampled-envelope consistency term rather than as full boundary-trace enforcement of a Dirichlet problem. 

For all MuJoCo, Humanoid, and Safety Gymnasium runs, we used diagonal
curvature banks $M_q=\operatorname{diag}(\alpha_{q,1},\ldots,\alpha_{q,n})$, $q=1,\ldots,K$,
i.e., \(R_q=I\) in \eqref{eq:M-bank}, so that the proximal network
receives \((x,\log\operatorname{diag}M_q,b)\) as input. The complete settings are listed in \cref{tab:sm-mujoco-core-transposed,tab:sm-mujoco-pde-transposed}.

\begin{table}[htbp]
\footnotesize
\caption{Network architecture, policy parameterization, PPO, and optimization parameters for the MuJoCo locomotion and Humanoid-v5 experiments. All methods share the same actor and critic architectures within each environment; V-GAC additionally uses the proximal network listed in the table. The entropy coefficient is linearly annealed to zero over the first \(70\%\) of training.}
\label{tab:sm-mujoco-core-transposed}
\centering
{\setlength{\tabcolsep}{4pt}
\renewcommand{\arraystretch}{1.03}
\resizebox{\linewidth}{!}{%
\begin{tabular}{@{}lccccc@{}}
\hline
\textbf{Hyperparameter} &
\textbf{HalfCheetah-v5} &
\textbf{Hopper-v5} &
\textbf{Walker2d-v5} &
\textbf{Ant-v5} &
\textbf{Humanoid-v5} \\
\hline
Actor\ widths                 & \((64,64)\)          & \((64,64)\)          & \((64,64)\)          & \((128,128)\)        & \((128,128)\) \\
Critic widths                & \((64,64)\)          & \((64,64)\)          & \((64,64)\)          & \((128,128)\)        & \((128,128)\) \\
Proximal widths                  & \((64,64)\)          & \((64,64)\)          & \((64,64)\)          & \((128,128)\)        & \((128,128)\) \\
Workers\(\times\)steps       & \(16\times128\)      & \(16\times128\)      & \(16\times128\)      & \(16\times128\)      & \(32\times64\) \\
$N$                      & \(10\)               & \(10\)               & \(10\)               & \(10\)               & \(10\) \\
$B$ & \(64\) & \(64\) & \(64\) & \(64\) & \(64\) \\
\(\boldsymbol{\epsilon}\)    & \(0.20\)             & \(0.20\)             & \(0.20\)             & \(0.20\)             & \(0.20\) \\
\(\lambda_{\text{GAE}}\) & \(0.95\) & \(0.95\) & \(0.95\) & \(0.95\)
 & \(0.95\) \\ 
\(\boldsymbol{\lambda_{\text{ent},0}}\) 
                              & \(5\times10^{-4}\)   & \(1\times10^{-3}\)   & \(5\times10^{-4}\)   & \(1\times10^{-3}\)   & \(5\times10^{-4}\) \\
\(\lambda_{\text{TD}}\)           & 0.5 & 0.5 & 0.5 & 0.5 & 0.5 \\
\(\text{lr}_{\text{actor,critic}}\)                & \(3.0\times10^{-4}\) & \(3.0\times10^{-4}\) & \(3.0\times10^{-4}\) & \(3.0\times10^{-4}\) & \(2.0\times10^{-4}\) \\
\(\text{lr}_{\text{prox}}\)                  & \(5.0\times10^{-4}\) & \(7.5\times10^{-4}\) & \(5.0\times10^{-4}\) & \(7.5\times10^{-4}\) & \(3.0\times10^{-4}\) \\
Gradient clipping & \(0.5\) & \(0.5\) & \(0.5\) & \(0.5\) & \(0.5\) \\
Total steps                  & \(1.0\times 10^{6}\) & \(1.0\times 10^{6}\) & \(1.0\times 10^{6}\) & \(1.0\times 10^{6}\) & \(1.0\times 10^{7}\) \\
Evaluated every                   & \(2.5\times 10^{4}\) & \(2.5\times 10^{4}\) & \(2.5\times 10^{4}\) & \(2.5\times 10^{4}\) & \(2.5\times 10^{5}\) \\
\hline
\end{tabular}%
}
}
\end{table}

\begin{table}[htbp]
\footnotesize
\caption{Viscosity-envelope, jet-alignment, and proximal-adversary parameters for the MuJoCo locomotion and Humanoid-v5 experiments. The actor and critic learning rates are equal unless noted otherwise. The HJBPPO values of \(\lambda_{\text{HJB}}\) for HalfCheetah/Hopper/Walker2d/Ant/Humanoid follow \cite{mukherjee2023}; the V-GAC-specific coefficients are the settings used in our runs.}
\label{tab:sm-mujoco-pde-transposed}
\centering
{\setlength{\tabcolsep}{4pt}
\renewcommand{\arraystretch}{1.03}
\resizebox{\linewidth}{!}{%
\begin{tabular}{@{}lccccc@{}}
\hline
\textbf{Hyperparameter} &
\textbf{HalfCheetah-v5} &
\textbf{Hopper-v5} &
\textbf{Walker2d-v5} &
\textbf{Ant-v5} &
\textbf{Humanoid-v5} \\
\hline
\(K\) & \(4\) & \(4\) & \(4\) & \(6\) & \(2\) \\
\(\alpha_{\min}\)                    & \(1.0\times10^{-2}\) & \(1.0\times10^{-2}\) & \(1.0\times10^{-2}\) & \(5.0\times10^{-3}\) & \(1.0\times10^{-3}\) \\
\(\alpha_{\max}\)                    & \(1.0\times10^{1}\)  & \(2.0\times10^{1}\)  & \(2.0\times10^{1}\)  & \(5.0\times10^{1}\)  & \(5.0\) \\
$\alpha$-range schedule & fixed & fixed & fixed & fixed & fixed \\
\(\rho_{\text{cover}}\)            & \(0.25\)             & \(0.35\)             & \(0.35\)             & \(0.45\)             & \(0.15\) \\
\(\lambda_{\text{HJB}}\)           & \(1.0\times10^{-1}\) & \(1.0\times10^{-1}\) & \(1.0\times10^{-1}\) & \(1.0\times10^{-1}\) & \(1.0\times10^{-4}\) \\
\(\lambda_{\text{visc}}\)          & \(1.0\times10^{-2}\) & \(1.5\times10^{-2}\) & \(1.25\times10^{-2}\) & \(2.0\times10^{-2}\) & \(5.0\times10^{-3}\) \\
\(\lambda_{\text{jet}}\)           & \(2.0\times10^{-3}\) & \(3.0\times10^{-3}\) & \(2.5\times10^{-3}\) & \(3.0\times10^{-3}\) & \(1.0\times10^{-3}\) \\
\(\lambda_{\text{adv}}\)           & \(1.0\)              & \(1.0\)              & \(1.0\)              & \(1.0\)              & \(1.0\) \\
\(\lambda_{\text{env}}\)           & \(5.0\times10^{-2}\) & \(7.5\times10^{-2}\) & \(5.0\times10^{-2}\) & \(1.0\times10^{-1}\) & \(1.0\times10^{-1}\) \\
\(\lambda_{\text{prox-opt}}\) & \(1.0\times10^{-1}\) & \(1.5\times10^{-1}\) & \(1.25\times10^{-1}\) & \(2.0\times10^{-1}\) & \(2.5\times10^{-1}\) \\
\(K_{\text{adv}}\) & \(1\) & \(2\) & \(1\) & \(2\) & \(2\) \\
\(\eta\) & $5.0\times10^{-2}$ & $5.0\times10^{-2}$ & $5.0\times10^{-2}$ & $2.5\times10^{-2}$ & $2.5\times10^{-2}$ \\ 
\hline
\end{tabular}
}
}
\end{table}

For Safety Gymnasium, we similarly kept the same actor-critic architecture and PPO update scheme, but tuned the PDE parameters more conservatively than in MuJoCo. Empirically, these navigation tasks were more sensitive to critic-side PDE terms under the reported training protocol, so we used more conservative weights than in the locomotion tasks. The scalarization weight in $\ell(x,c)=-R_{\text{mdp}} (x,c)+\lambda_{\kappa} \kappa(x,c)$ was held fixed at $\lambda_{\kappa}=0.10$ across all safety environments and all three methods. The cost limit was set to $\kappa_{\max}=25$. Throughout the training, the learning rates were kept fixed.  The full settings are listed in \cref{tab:sm-safety-core-transposed,tab:sm-safety-pde-transposed}.

\begin{table}[htbp]
\footnotesize
\caption{Network architecture, policy parameterization, PPO, and optimization parameters for   the Safety Gymnasium experiments. The entropy coefficient is linearly annealed to zero over the first \(70\%\) of training.}
\label{tab:sm-safety-core-transposed}
\centering
{\setlength{\tabcolsep}{4pt}
\renewcommand{\arraystretch}{1.03}
\resizebox{\linewidth}{!}{%
\begin{tabular}{@{}lccc@{}}
\hline
\textbf{Hyperparameter} &
\textbf{SafetyPointGoal1-v0} &
\textbf{SafetyPointPush1-v0} &
\textbf{SafetyCarGoal1-v0} \\
\hline
Actor widths                 & \((64,64)\)         & \((64,64)\)         & \((128,128)\) \\
Critic widths                & \((64,64)\)         & \((64,64)\)         & \((128,128)\) \\
Prox widths                  & \((64,64)\)         & \((64,64)\)         & \((128,128)\) \\
Workers\(\times\)steps       & \(16\times128\)     & \(16\times128\)     & \(16\times128\) \\
$N$                          & \(10\)              & \(10\)              & \(10\) \\
$B$ & \(64\) & \(64\) & \(64\) \\
\(\boldsymbol{\epsilon}\)    & \(0.20\)            & \(0.20\)            & \(0.20\) \\
\(\lambda_{\text{GAE}}\) & \(0.95\) & \(0.95\) & \(0.95\) \\ 
\(\boldsymbol{\lambda_{\text{ent},0}}\) & \(1.0\times10^{-3}\) & \(2.0\times10^{-3}\) & \(1.0\times10^{-3}\) \\
\(\lambda_{\text{TD}}\) & 0.5 & 0.5 & 0.5 \\
\(\text{lr}_{\text{actor,critic}}\)                  & \(3.0\times10^{-4}\) & \(3.0\times10^{-4}\) & \(2.5\times10^{-4}\) \\
\(\text{lr}_{\text{prox}}\)                    & \(5.0\times10^{-4}\) & \(7.5\times10^{-4}\) & \(5.0\times10^{-4}\) \\
Gradient clipping & \(0.5\) & \(0.5\) & \(0.5\) \\
Total steps                  & \(1.0\times10^{7}\) & \(1.0\times10^{7}\) & \(1.0\times10^{7}\) \\
Evaluated every                   & \(2.5\times10^{5}\) & \(2.5\times10^{5}\) & \(2.5\times10^{5}\) \\
\hline
\end{tabular}%
}
}
\end{table}
\begin{table}[htbp]
\footnotesize
\caption{Viscosity-envelope, jet-alignment, and proximal-adversary parameters for  the Safety Gymnasium experiments.}
\label{tab:sm-safety-pde-transposed}
\centering
{\setlength{\tabcolsep}{4pt}
\renewcommand{\arraystretch}{1.03}
\resizebox{\linewidth}{!}{%
\begin{tabular}{@{}lccc@{}}
\hline
\textbf{Hyperparameter} &
\textbf{SafetyPointGoal1-v0} &
\textbf{SafetyPointPush1-v0} &
\textbf{SafetyCarGoal1-v0} \\
\hline
\(K\) & \(4\) & \(5\) & \(4\) \\
\(\alpha_{\min}\) & \(5.0\times10^{-3}\) & \(5.0\times10^{-3}\) & \(1.0\times10^{-2}\) \\
\(\alpha_{\max}\)& \(1.0\times10^{1}\)  & \(2.0\times10^{1}\)  & \(1.0\times10^{1}\) \\
$\alpha$-range schedule & fixed & fixed & fixed \\
\(\rho_{\text{cover}}\) & \(0.50\) & \(0.55\) & \(0.45\) \\
\(\lambda_{\text{HJB}}\) & \(1.0\times10^{-3}\) & \(1.0\times10^{-3}\) & \(5.0\times10^{-4}\) \\
\(\lambda_{\text{visc}}\) & \(8.0\times10^{-3}\) & \(1.0\times10^{-2}\) & \(8.0\times10^{-3}\) \\
\(\lambda_{\text{jet}}\) & \(1.5\times10^{-3}\) & \(2.0\times10^{-3}\) & \(1.5\times10^{-3}\) \\
\(\lambda_{\text{adv}}\) & \(1.0\) & \(1.0\) & \(1.0\) \\
\(\lambda_{\text{env}}\) & \(5.0\times10^{-2}\) & \(7.5\times10^{-2}\) & \(7.5\times10^{-2}\) \\
\(\lambda_{\text{prox-opt}}\) & \(1.0\times10^{-1}\) & \(1.5\times10^{-1}\) & \(1.5\times10^{-1}\) \\
\(K_{\text{adv}}\) & \(1\) & \(2\) & \(2\) \\
\(\eta\) & $5.0\times10^{-2}$ & $5.0\times10^{-2}$ & $2.5\times10^{-2}$ \\ 
\hline
\end{tabular}
}
}
\end{table}

\FloatBarrier
%%%%%%%%%%%%%%%%%%%%%%%%%%%%%%%%%%%%%%%%%%%%%%%%%%%%%%%%%%%%%%%%%%%%%%%%%%%%%%%%%%%%%%%%

\section*{Acknowledgments}
Alen Golpashin and Bruce Conway acknowledge the support of the U.S. Air Force Research Laboratory (AFRL) under grant number FA8651-23-1-0001. Any opinions and findings in this paper are of the authors and do not reflect those of the AFRL. This work used the Accelerating Computing for Emerging Sciences (ACES) at Texas A\&M University through allocation MTH240017 from the Advanced Cyberinfrastructure Coordination Ecosystem: Services \& Support (ACCESS) program, which is supported by U.S. National Science Foundation grant numbers 2138259, 2138286, 2138307, 2137603, and 2138296.

\bibliographystyle{siamplain}
\bibliography{references}
\end{document}